\documentclass[11pt]{amsart}
\usepackage[top=1.2in,bottom=1.2in,left=1.2in,right=1.2in,marginpar=1in]{geometry}
\usepackage{amsfonts}
\usepackage{amsmath}
\usepackage{amsthm}
\usepackage{nccmath}
\usepackage{url}
\usepackage[all]{xy}
\usepackage{latexsym}
\usepackage{amssymb}
\usepackage[cp850]{inputenc}
\usepackage{graphicx}
\usepackage{dsfont}
\usepackage{float}
\usepackage{stackrel}

\usepackage{caption}
\usepackage{subcaption}
\usepackage{lipsum}
\usepackage{setspace}
\usepackage{bigints}
\usepackage{stmaryrd} 
\usepackage{calrsfs}
\usepackage{comment}
\usepackage{relsize}
\usepackage{array}
\usepackage{multirow}

\usepackage[dvipsnames]{xcolor}
\usepackage[cal=euler]{mathalfa}
\usepackage{mathtools} %:=
\usepackage{braket}
\usepackage{tikz-cd}

\let\oldtocsection=\tocsection

\let\oldtocsubsection=\tocsubsection

\let\oldtocsubsubsection=\tocsubsubsection

\renewcommand{\tocsection}[2]{\hspace{0em}\oldtocsection{#1}{#2}}
\renewcommand{\tocsubsection}[2]{\hspace{1em}\oldtocsubsection{#1}{#2}}
\renewcommand{\tocsubsubsection}[2]{\hspace{2em}\oldtocsubsubsection{#1}{#2}}

\newcommand*{\Lfaktor}[2]{% \leftfaktor{#1}{#2} -> #2\#1
    {}_{#2}\!\backslash\!^{#1}
}
\newcommand*{\Rfaktor}[2]{% \faktor{#1}{#2} -> #1/#2
	^{ #1}\!/\!_{ #2}% Denominator
}

\newtheorem{sat}{Theorem}[section]		
\newtheorem{lem}[sat]{Lemma}
\newtheorem{kor}[sat]{Corollary}			
\newtheorem{prop}[sat]{Proposition}
\newtheorem{defi}[sat]{Definition}
\newtheorem*{defi*}{Definition}			
\newtheorem*{bei*}{Example}
\newtheorem*{sat*}{Theorem}				
\newtheorem*{kor*}{Corollary}
\newtheorem*{rmk*}{Remark}				
	
\newtheorem*{quest*}{Question}	
\newtheorem{fact}[sat]{Fact}	
\newtheorem*{ex}{Example}

\let\ssection=\section
\renewcommand{\section}{\setcounter{equation}{0}\ssection}

\newtheorem*{namedtheorem}{\theoremname}
\newcommand{\theoremname}{testing}
\newenvironment{named}[1]{\renewcommand{\theoremname}{#1}\begin{namedtheorem}}{\end{namedtheorem}}

\newtheorem*{namedtheoremr}{\theoremnamer}
\newcommand{\theoremnamer}{testing}

\newcommand{\BA}{\mathbb A} 			 
\newcommand{\BC}{\mathbb C}

			\newcommand{\BN}{\mathbb N}
			
			\newcommand{\BR}{\mathbb R}
\newcommand{\BS}{\mathbb S}			\newcommand{\BT}{\mathbb T}

			\newcommand{\BZ}{\mathbb Z}

\newcommand{\CA}{\mathcal A}			\newcommand{\CB}{\mathcal B}

			\newcommand{\CL}{\mathcal L}
\newcommand{\CM}{\mathcal M}		\newcommand{\CN}{\mathcal N}

			\newcommand{\CV}{\mathcal V}
		\newcommand{\CX}{\mathcal X}
			
\newcommand{\CD}{\mathcal D}	

\newcommand{\FM}{\mathfrak m}

\newcommand{\FB}{\mathfrak b}

\newcommand{\FC}{\mathfrak c}

\newcommand{\hyp}{\mathrm{hyp}}
\newcommand{\fix}{\mathrm{fix}}
\newcommand{\ann}{\mathrm{ann}}
\newcommand{\ns}{\mathrm{ns}}
\newcommand{\nsl}{\mathrm{nsL}}
\newcommand{\flip}{\mathrm{flip}}
\newcommand{\Stab}{\mathrm{Stab}}
\usepackage{bm} %bold italic

\newcounter{cst}

\newcommand{\actson}{\curvearrowright}
\newcommand{\D}{\partial}

\DeclareMathOperator{\Diff}{Diff}	%	Diffeomorphimen einer Mf
\DeclareMathOperator{\Id}{Id}		%	Identit\"at
\DeclareMathOperator{\Map}{Map}
\DeclareMathOperator{\PMap}{PMap}

\DeclareMathOperator{\Ker}{Ker}

\DeclareMathOperator{\Aut}{Aut}

\DeclareMathOperator{\Fix}{Fix}
\DeclareMathOperator{\area}{area}

\DeclareMathOperator{\Thu}{Thu}

\DeclareMathOperator{\supp}{Supp}

\DeclareMathOperator{\hal}{\mathbf{half}}
\DeclareMathOperator{\neigh}{\mathbf{neigh}}

\DeclareMathOperator{\Leb}{Leb}
\DeclareMathOperator{\sep}{sep}
\DeclareMathOperator{\sepl}{sepL}
\DeclareMathOperator{\sym}{sym}

\DeclareMathOperator{\Rib}{{\bf Rib}}

\DeclareMathOperator{\product}{prod}

\DeclareMathOperator{\sys}{sys}

\newcommand{\fsubd}{\mathrel{{\scriptstyle\searrow}\kern-1ex^d\kern0.5ex}}
\newcommand{\bsubd}{\mathrel{{\scriptstyle\swarrow}\kern-1.6ex^d\kern0.8ex}}

\renewcommand{\le}{\leqslant}
\renewcommand{\ge}{\geqslant}
\renewcommand{\emptyset}{\varnothing}

\renewcommand{\geq}{\geqslant}
\renewcommand{\leq}{\leqslant}

\definecolor{mygrey}{RGB}{119,122,120}

\definecolor{mypink}{RGB}{219,24,121}

\usepackage{hyperref}
\begin{document}

\title[]{Large genus asymptotics of frequencies of non-simple curves}
\author{Mingkun Liu}
\address{LAGA, Universit\'e Sorbonne Paris Nord}
\email{mingkun.liu@math.univ-paris13.fr}
\author{Kasra Rafi}
\address{Department of Mathematics, University of Toronto}
\email{rafi@math.toronto.edu}
\author{Juan Souto}
\address{IRMAR, Universit\'e de Rennes}
\email{juan.souto@univ-rennes.fr}
\author{Marie Trin}
\address{Max-Plank-Institut f\"ur Mathematik in den Naturwissenschaften}
\email{marie.trin@mis.mpg.de}

\begin{abstract} We give an expression for the frequency of non-simple curves in closed surfaces and exploit it to study relative frequencies of such curves in large genus. This extend to the case of non-simple curves Mirzakhani's expressions of frequencies in terms of Konsevitch polynomials and Delecroix-Goujard-Zograf-Zorich large genus asymptotics for those frequencies. In particular, with $K$ fixed, we identify which types of curves with $K$ intersections are most common.   
\end{abstract}

\maketitle

\section{Introduction} 
The aim of this paper is to determine how typical curves with a given number of self-intersections look like in surfaces of large genus. All our surfaces $X$ are going to be closed and orientable. We consider curves up to free homotopy, and we always assume them to be simple and primitive. In particular, when we endow $X$ with an arbitrary hyperbolic metric, we can identify curves with their unique geodesic representatives. Two curves are {\em of the same type} if they differ by a self-homeomorphism of the surface. Rather, since we identify homotopic curves, $\gamma$ {\em is of type} $\gamma_0$ if it belongs to the mapping class group orbit $\Map(X)\cdot\gamma_0$.
\medskip

Mirzakhani \cite{Maryam simple} proved that randomly chosen, very long, simple closed curves in a hyperbolic surface of genus $g=2$ are $48$ times less likely to be separating than non-separating. To make this statement precise, recall that if $\gamma_0$ is a curve, simple or not, in a closed hyperbolic surface $X$ of genus $g$, then the limit 
\begin{equation}\label{eq absolute frequency exists}
\lim_{L\to\infty}\frac 1{L^{6g-6}}\#\{\gamma\text{ of type }\gamma_0\text{ with }\ell_X(\gamma)\le L\} = \FC(X)\frac{\FC_g(\gamma_0)}{\FB_g}
\end{equation}
exists and is positive \cite{Maryam simple,Maryam general curves,Kasra-Juan,book}. Here $\FB_g$ is a positive constant which only depends on the genus $g$, $\FC(\cdot)$ a proper function on the moduli space of surfaces of genus $g$, and $\ell_X(\gamma)$ is the hyperbolic length of the unique geodesic representative in the homotopy class of $\gamma$. What is key is that, while the length function $\ell_X(\cdot)$ and the quantity $\FC(X)$ depend on the hyperbolic metric, the constant $\FC_g(\gamma_0)$, the {\em frequency of $\gamma_0$}, just depends on the type of $\gamma_0$ and of the fact that we are in a genus $g$ surface. In fact, \eqref{eq absolute frequency exists} remains true if one replaces the hyperbolic length by many other reasonable functions measuring the complexity of curves, replacing $\FC(X)$ by an appropriate number depending on the chosen function---see \cite{ESP,book,didac} for details and many examples of such functions. 

\begin{rmk*}
Due to different choices of normalization, one finds in the literature different values for $\FC_g(\gamma_0)$ and $\FB_g$ in \eqref{eq absolute frequency exists}. In this paper we follow the conventions of \cite{book}. See Section \ref{sec frequency} for details.
\end{rmk*}

Armed with \eqref{eq absolute frequency exists} we can state precisely the above mentioned $\frac 1{48}$ result of Mirzakhani. What she proved is that {\em if $X$ is a closed hyperbolic surface of genus $2$, and $\gamma_{\text{sep}}$ and $\gamma_{\ns}$ are respectively a separating and a non-separating simple closed curve, then we have
$$\frac{\FC_2(\gamma_{\text{sep}})}{\FC_2(\gamma_{\ns})}=\lim_{L\to\infty}\frac{\#\{\gamma\text{ separating simple curve with }\ell_X(\gamma)\le L\}}
{\#\{\gamma\text{ non-separating simple curve with }\ell_X(\gamma)\le L\}}=\frac 1{48}.$$}

\noindent Note that every simple curve in the genus $2$ surface $X$ is of either type $\gamma_{\ns}$ or $\gamma_{\text{sep}}$.
\medskip

\noindent \textbf{Separating vs. non-separating--large genus asymptotics.} Delecroix, Goujard, Zograf, and Zorich investigated how the quantity $\frac 1{48}$ changes when the genus $g$ increases. Let us be more precise. Denote by $\FC_g(\text{simple, sep})$ the sum of $\FC_g(\gamma)$ over all types of separating simple curves $\gamma$ in a closed surface of genus $g$ and let $\FC_g(\text{simple, }\ns)=\FC_g(\gamma_{\ns})$ be the frequency of some--and hence any because there is just one type--non-separating simple curve $\gamma_{\ns}$. In \cite{DGZZ, DGZZ2}, Delecroix-Goujard-Zograf-Zorich showed that
\begin{equation}\label{eq DGZZ asymptotic}
\frac{\FC_g(\text{simple, sep})}{\FC_g(\text{simple, ns})}\sim\sqrt{\frac 2{3\pi g}}\cdot 4^{-g}.
\end{equation}
Here, and in the sequel, the symbol $\sim$ means that the ratio between both sides tends to $1$ when $g\to\infty$. One can thus rephrase \eqref{eq DGZZ asymptotic} as asserting that {\em when the genus is large, the overwhelming majority of all simple closed curves are non-separating}. 

Our first result is that the same happens for curves with any fixed number of intersections. Let us start with a comment on the topology of surfaces. While homotopy implies isotopy for simple curves, this fails when curves intersect themselves. It even fails when we assume that the involved curves have minimal number of self-intersections in their homotopy class. What is still true is that when $\gamma,\gamma'$ are homotopic curves in general position and minimizing the number of self-intersections in their homotopy class, then the complement of $\gamma$ is connected if and only if that of $\gamma'$ is connected. It follows that it makes sense to speak of separating and non-separating homotopy classes of non-simple curves in surfaces.

\begin{rmk*}
    When $K\ge 3$ and the genus is large, then there is more than one mapping class group orbit of non-separating curves with $K$ self-intersections--see with Figure \ref{fig:Curves4Int}. Altogether there are, for any fixed $K$, finitely many types of curves with $K$ self-intersections.
\end{rmk*}

Now that the terminology is clear, note that we get from \eqref{eq absolute frequency exists} that, for given $g$ and $K$, the non-separating curves are a definite proportion of the set of all curves with $K$ self-intersections in a surface of genus $g$. More precisely we have
$$\lim_{L\to\infty}\frac{\#\{\gamma\text{ separating curve with }\iota(\gamma,\gamma)=K\text{ and }\ell_X(\gamma)\le L\}}{\#\{\gamma\text{ non-separating curve with }\iota(\gamma,\gamma)=K\text{ and }\ell_X(\gamma)\le L\}}=
\frac{\FC_g(\iota=K,\sep)}{\FC_g(\iota=K,\ns)},$$
where $\FC_g(\iota=K,\sep)$ (resp. $\FC_g(\iota=K,\ns)$) is the sum of $\FC_g(\gamma)$ over all types of separating (resp. non-separating) curves $\gamma$ with $\iota(\gamma,\gamma)=K$ in a closed surface of genus $g$. Analogous to Delecroix-Goujard-Zograf-Zorich \cite{DGZZ, DGZZ2}, we describe for $K\ge 1$ fixed, how the ratio between $\FC_g(\iota=K,\sep)$ and $\FC_g(\iota=K, \ns)$ behaves when $g$ grows:

\begin{named}{Theorem \ref{thm:main type fixed int number}}
    For all $K\ge 1$ we have
\[\frac{\FC_g(\iota=K,\sep)}{\FC_g(\iota=K,\ns)}=\mathrm{O}\left( \frac{1}{g}  \right)\]
as $g\to\infty$.
\end{named}

As it was the case for simple curves, we can summarize the statement of Theorem \ref{thm:main type fixed int number} in the following way:
\begin{quote}
    \emph{In large genus, the overwhelming majority of curves with $K$ self-intersections are non-separating.}
\end{quote}
There is however a difference between the simple and the non-simple cases. While Delecroix-Goujard-Zograf-Zorich get that for simple curves the probability of being separating decays exponentially as the genus grows, we only get polynomial decay from Theorem \ref{thm:main type fixed int number}. We do not know if the decay rate $O(g^{-1})$ is optimal, but one can definitively not do any better than $O(g^{-2})$--compare with Theorem \ref{kor asymp fix int} below and the curves $\gamma^K_{\max-1}$ in Table \ref{tab:4int}. Indeed, in the setting of non-simple curves, polynomial factors are a feature and not a bug. Our life would be easier if this were not the case.
\medskip

\noindent \textbf{Total frequency--large genus asymptotics.}
Returning now for a moment to the setting of simple curves, one also gets from \cite{DGZZ, DGZZ2} asymptotics for the {\em total simple frequency } $\FC_g(\text{simple})=\FC_g(\text{simple},\sep)+\FC_g(\text{simple, }\ns)$ in genus $g$, that is the sum of the frequencies of all types of curves:
\begin{equation}\label{eq DGZZ asymptotic2}
\FC_g(\text{simple}) \asymp \frac{1}{2^{2g}}\left( \frac{e}{3g} \right)^{4g}g^{2}
\end{equation}
where the symbol $\asymp$ means that the ratio tends to a positive constant when $g\to\infty$. When we consider curves with self-intersections we get the following asymptotics for the {\em total frequency
$$\FC_g(\iota=K)=\FC_g(\iota=K,\sep)+\FC_g(\iota=K,\ns)$$
of curves with $K$ self-intersections in a surface of genus $g$}:

\begin{named}{Theorem \ref{kor asymp fix int}} For any $K\ge0$, we have
    \[   \FC_g(\iota=K) \asymp   \dfrac{1}{2^{2g}}\left( \frac{e}{3g}\right)^{4g} g^{K+2} \]
    as $g\to\infty$. 
\end{named}

As we see, both the total frequency of curves with $K$ self-intersections and the total simple frequency differ, asymptotically, by a constant multiple of the polynomial factor $g^K$.

\begin{rmk*}
    In the course of the proof of Theorem \ref{thm:main type fixed int number} we actually get that if $g$ is large and $\gamma$ is a non-separating curve with $\iota(\gamma,\gamma)=K$ in a closed surface of genus $g$, then the frequency $\FC_g(\gamma)$ is of the order of $\FC_g(\iota=K)$. However, we do not know if for any sequence $(\gamma_g)$ of such curves we have
    $$\FC_g(\gamma_g)\asymp \dfrac{1}{2^{2g}}\left( \frac{e}{3g} \right)^{4g}g^{K+2}$$
    or not. In other words, there could be $c>1$ and two sequences $(\gamma_g)$ and $(\gamma_g')$ with both $\gamma_g$ and $\gamma_g'$ non-separating in the surface of genus $g$, both with $K$ self-intersections, and with $\FC_g(\gamma_g)>c\cdot\FC_g(\gamma_g')$ for all $g$.
\end{rmk*}

\medskip

\noindent \textbf{Frequencies of local types--large genus asymptotics.}
In order to obtain such results as Theorem \ref{thm:main type fixed int number} and Theorem \ref{kor asymp fix int}, we need to understand how the frequencies $\FC_g(\cdot)$ behave when $g\to\infty$. It makes however no sense to say that curves in surfaces of different genus are of the same type. Motivated by work of 
Anantharaman--Monk \cite{AM}, and by conversations with them, we are led to the concepts of {\em local type}  and of {\em realization} of a local type. We will give a precise definition in Section \ref{sec:local type}, but for now we can think of a {\em local type} as a pair $(\Sigma,\gamma_0)$ consisting of a compact surface $\Sigma$ and of a filling curve $\gamma_0\subset\Sigma$, and of a {\em realization} of the local type $(\Sigma,\gamma_0)$ as a $\pi_1$-injective embedding $\phi:\Sigma\to X$. A curve $\gamma$ in surface $X$ is {\em of local type} $(\Sigma,\gamma_0)$ if there is a realization $\phi:\Sigma\to X$ with $\gamma$ homotopic to $\phi(\gamma_0)$. The curve $\gamma$ of local type $(\Sigma,\gamma_0)$ is {\em of non-separating local type} if $X\setminus\phi(\Sigma)$ is connected\footnote{The difference between what we are calling here {\em local type}, {\em realization}, and {\em non-separating local type}, and what we will formally define in Section \ref{sec:local type} is the following. Later on we will decorate the pair $(\Sigma,\gamma_0)$ with an involution $\flip_\Sigma$ of the set of connected components of $\D\Sigma$ and realizations $\phi:\Sigma\to X$ will be assumed to map curves in a $\flip_\Sigma$-orbit to homotopic curves. If $\flip_\Sigma\neq\Id$, this forces the presence of annuli in $X\setminus\phi(\Sigma)$. We will thus say that a curve will the be of non-separating local type when, other than those annuli, $X\setminus\phi(\Sigma)$ is connected. All of this will be important for the bookkeeping in the later parts of the paper, but the reader can safely ignore it for now.}.

\begin{ex}\label{bei figure-8 intro}
The pair $(Y,{\bf 8})$ consisting of a pair of pants and of a figure-8 in there, is a local type. Any curve $\gamma\subset X$ with self-intersection 1 is of local type $(Y,{\bf 8})$--it is of non-separating local type $(Y,{\bf 8})$ if and only if $X\setminus\gamma$ is connected.
\end{ex}

\begin{rmk*}
    In Example \ref{bei figure-8 intro} we have that the notions of being "non-separating" and "of non-separating local type" agree for curves of local type $(Y,{\bf 8})$. In general this is not the case: think of curves of local type $(Y,\gamma_0)$ where $Y$ is still a pair of pants and $\gamma_0$ is anything other than a figure-8.
\end{rmk*}

Given a local type $(\Sigma,\gamma_0)$, we denote by $\FC_g(\Sigma,\gamma_0)$ the sum of $\FC_g(\gamma)$ over all types of curves $\gamma$ of local type $(\Sigma,\gamma_0)$ in a surface of genus $g$. If we take the sum over all types of curves $\gamma$ of non-separating local type $(\Sigma,\gamma_0)$ we get $\FC(\Sigma,\gamma_0,\nsl)$.

In some sense, our main result gives, for every local type $(\Sigma,\gamma_0)$, the asymptotic behavior of $\FC_g(\Sigma,\gamma_0,\nsl)$:

\begin{named}{Theorem \ref{thm:final asym}}
     Let $(\Sigma,\gamma_0)$ be a local type then
    $$\FC_g(\Sigma,\gamma_0,\nsl)\asymp \dfrac{1}{2^{2g}} \left(\dfrac e{3g}\right)^{4g}\cdot \dfrac{g^{\chi(\Sigma)+\mu_0+2}}{\iota_0^{6g}}$$
    where 
    \begin{align*}
        \iota_0&=\min\{\iota(\alpha,\gamma_0)\vert\ \alpha\subset\Sigma\text{ essential simple arc}\},\text{ and}\\
        \mu_0&=\max\{k\vert\exists \text{ disjoint non-parallel simple arcs }\alpha_1,\dots,\alpha_k\subset\Sigma\text{ with }\iota(\alpha_i,\gamma_0)=\iota_0\}.
    \end{align*}
\end{named}

\begin{rmk*}
In the course of the proof of Theorem \ref{thm:final asym} we do not extract an expression for the constant hidden in the symbol $\asymp$. In specific cases such as that of the figure-8 it is possible to get explicit values for that constant, but doing that in general seems much harder, if possible at all.
\end{rmk*}

It turns out that curves of non-separating local type $(\Sigma,\gamma_0)$ have maximal contribution among all curves of local type $(\Sigma,\gamma_0)$, see Theorem \ref{sat dominant type} and Theorem \ref{essential is dominant}. Hence, the above theorem produces an asymptotic for $\FC_g(\Sigma,\gamma_0)$:

\begin{named}{Corollary \ref{kor final asyp loc type}}
    If $(\Sigma,\gamma_0)$ is a local type then we have
    $$\FC_g(\Sigma,\gamma_0)\asymp \dfrac{1}{2^{2g}} \left(\dfrac e{3g}\right)^{4g}\cdot \dfrac{g^{\chi(\Sigma)+\mu_0+2}}{\iota_0^{6g}}$$
    where $\iota_0$ and $\mu_0$ are as in Theorem \ref{thm:final asym}.
\end{named}

Recalling that the total simple frequency $\FC_g(\text{simple})$ in genus $g$ is given by \eqref{eq DGZZ asymptotic2}, we can rewrite the statement of Corollary \ref{kor final asyp loc type} as
    $$\FC_g(\Sigma,\gamma_0)\asymp \FC_g(\text{simple})\cdot \dfrac{g^{\chi(\Sigma)+\mu_0}}{\iota_0^{6g}}.$$
We can thus think of the total simple frequency as the baseline, and of the second term as a correction factor  determined by the topology of the local type. The reader can amuse themselves to calculate these correction factors for the examples of local types given in Figure~\ref{fig:examples of local types intro}. 

\begin{figure}[h]
    \centering
   \includegraphics[width=0.8\linewidth]{ 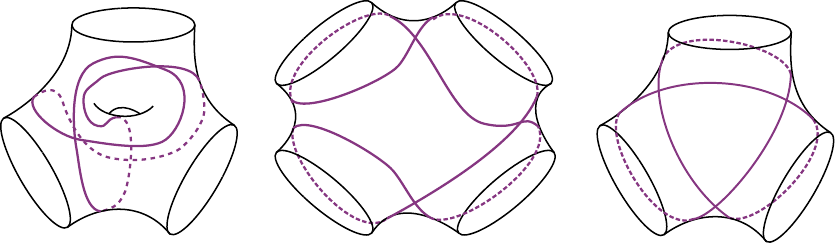}
    \caption{Examples of local types.}
    \label{fig:examples of local types intro}
\end{figure}
\medskip

\noindent \textbf{An expression for the frequency.}
For simple curves $\gamma_0$, Mirzakhani expressed in \cite{Maryam simple} the frequency $\FC_g(\gamma_0)$ in \eqref{eq absolute frequency exists} in terms of the Kontsevich polynomial
$V_{X\setminus\gamma_0}(b_1,b_2)$ of the complement of $\gamma_0$, namely
\begin{equation}\label{eq maryam frequency kontsevich}
\FC_g(\gamma_0)=\frac {1}{\sym(\gamma_0)}\int\limits_{0\le b \le 1}V_{X\setminus\gamma_0}\left( b, b\right)\cdot b\ d b
\end{equation}
where $\sym(\gamma_0)\in\{1,2\}$, depending on whether there are self-homeomorphisms of $X\setminus\gamma_0$ exchanging both boundary components. 

Having an expression for the frequency for simple curves in terms of the Kontsevich polynomial is the starting point of the work of Delecroix-Goujard-Zograf-Zorich. Getting such an expression for the frequency of non-simple curves is the starting point of our work here. 

The frequency $\FC(\gamma_0)$ can be expressed, for both simple and non-simple curves $\gamma_0$, as the Thurston measure of a suitable subset of the space $\CM\CL(X)$ of measured laminations---see \cite{Kasra-Juan} for the case that $\gamma_0$ is filling, and \cite{book} for the general case. In Theorem \ref{thm Kasra formula} below we give an explicit description of this set in terms of train tracks. Anyways, the Thurston measure is by construction the scaling limit of the counting measure over the set $\CM\CL_\BZ(X)$ of integral multicurves. It follows that we can obtain $\FC_g(\gamma_0)$ by counting certain integral simple multicurves. To facilitate this counting, we observe that, up to a twist, every simple multicurve can be recovered from the integral arc systems obtained by intersecting it with the subsurface $\Sigma=\Sigma(\gamma_0)$ filled by $\gamma_0$, and with its complement $X\setminus\Sigma$. When one writes things out, one obtains an expression for $\FC_g(\gamma_0)$ as a limit of weighted finite sums over the set of integral arc systems in $\Sigma$--see Theorem \ref{thm frequency counting} below. Arc systems are exactly the same as metric ribbon graphs, and what the Kontsevich polynomial actually does is to calculate the number of integral ribbon graphs with fixed boundary lengths \cite{Kontsevich,Norbury}. The following is the expression we obtain for $\FC_g(\gamma_0)$ in terms of the Kontsevich polynomial:

\begin{named}{Theorem \ref{thm constant c in terms of intersection numbers}}
    Let $X$ be a closed hyperbolic surface of genus $g\ge 3$ and $\gamma_0\subset X$ a non-simple curve filling a subsurface $\Sigma$ such that $Z=X\setminus\Sigma$ contains no annuli. Denote by $\BA(\Sigma)$ the arc-complex of $\Sigma$ and let ${\bf I}_{\gamma_0}:\BA(\Sigma)\to\BR_{\ge 0}$ be the intersection function with $\gamma_0$. Setting
$$\Delta_\Sigma(\gamma_0) :=\{\bar a \in \BA(\Sigma) : {\bf I}_{\gamma_0}(\bar a)\le 1 \}$$
we have
\begin{equation}  \label{eq:freq int intro}
     \FC_g(\gamma_0)=\frac {1}{\sym(\gamma_0)}\bigintsss\limits_{\Delta_\Sigma(\gamma_0)}V_Z\left({\bf b}(\bar a)\right)\cdot \left(\prod_{x\in{\bf b}(\bar a)}x\right)  d{\FM_{\BA(\Sigma)}}
\end{equation}
where $V_Z(\cdot)$ is the Kontsevich volume polynomial associated to the surface $Z$, where $\FM_{\BA(\Sigma)}$ is the canonical Lebesgue measure on $\BA(\Sigma)$, and where ${\bf b}:\BA(\Sigma)\to\BR_{\ge 0}^{\D\Sigma}$ is the map which sends arc systems to the vector whose entries are the weights the system puts on each individual boundary component.
\end{named}

Before moving on, let us add a few comments on Theorem \ref{thm constant c in terms of intersection numbers}:
\smallskip

\noindent{\bf (1.)} The actual statement of Theorem \ref{thm constant c in terms of intersection numbers} in Section \ref{sec:kontsevich} is more involved than the one we presented here. For starters, we drop there the assumption that $X\setminus\Sigma$ contains no annuli. We also allow $\gamma_0$ to be a multicurve which in particular might have isolated simple components. This is also the reason why the statements of Theorem \ref{thm:final asym} and Corollary \ref{kor final asyp loc type} in Section \ref{sec asymptotic} are more involved than the ones we gave above.\smallskip

\noindent{\bf (2.)} The assumption that $X$ in Theorem \ref{thm constant c in terms of intersection numbers} has at least genus $3$ is there to avoid the annoyance arising from the fact that genus $2$ surfaces are hyper-elliptic. This poses no problem for us because our aim is to study relative frequencies when $g\to\infty$. Still, since everything in \cite{book} is done allowing the genus to be 2, all what we do can be adapted to that case. The interested reader can find remarks on the places where modifications are needed in the document. In the Appendix we show how Mirzakhani's $\frac 1{48}$ follows from the general form of Theorem \ref{thm constant c in terms of intersection numbers}.\smallskip

\noindent{\bf (3.)} In \cite{Maryam general curves} Mirzakhani proves--at least for filling $\gamma_0$--that $\FC_g(\gamma_0)$ is a rational number. Unfortunately, we are unable to get the rationality of $\FC(\gamma_0)$ from our work. Or maybe we would just need to work some more. Indeed, \eqref{eq:freq int intro} suggests that $\FC_g$ is rational, as the integral over a simplex with rational coordinates of a polynomial with rational coefficients is rational. However, $\Delta_\Sigma$ is not properly a simplex but rather an infinite union of simplexes. It seems therefore unclear whether one can use this simple minded argument to ensure the rationality of $\FC_g$. The only case where it is clear that one can deduce the rationality of $\FC_g$ from Theorem \ref{thm constant c in terms of intersection numbers} is when $\Delta_\Sigma$ decomposes as a finite union, that is when $\Sigma=\Sigma(\gamma_0)$ is a pair of pants. Indeed, we compute an example in the Appendix where $\Sigma$ is a pair of pants. Other examples, also with $\Sigma$ a pair of pants have been computed in \cite{Victor} using a different approach.
\medskip

\noindent \textbf{Extracting asymptotics.}
As can be seen in \eqref{eq:freq int intro}, the Kontsevich polynomial is central to understanding the frequency $\mathfrak{c}_g(\gamma_0)$.
Kontsevich \cite{Kontsevich} showed that these coefficients admit an algebro-geometric interpretation,
as intersection numbers of certain cohomology classes of the Deligne--Mumford compactification of the moduli space of Riemann surfaces.
It was conjectured by Witten \cite{Witten}, and shortly thereafter proved by Kontsevich \cite{Kontsevich}, that these intersection numbers can be computed recursively.
Other recursive relations have since been discovered.
Nevertheless, due to the complexity of these recursive relations,
little can be said about these numbers for general $g$ and $n$.
However, these numbers become more tractable in the large genus limit.
An asymptotic formula for these number as $g \to \infty$ was conjectured by Delecroix--Goujard--Zograf--Zorich \cite{DGZZ} based on numerical data,
and was proved soon after by Aggarwal \cite{Aggarwal} (see also the remark preceeding Theorem \ref{Aggarwal}).
This result provides a means to understand the frequency in the large genus regime.

Anyways, plugging this asymptotic formula into \eqref{eq:freq int intro}, we get a more friendly version of the frequency $\mathfrak{c}_g(\gamma_0)$ in the large genus regime.
We encode these frequencies (indexed by $g$) into a carefully chosen generating function,
and apply standard singularity analysis.
The key idea is that the asymptotic behavior of the generating function's coefficients is determined by the analytic properties of the generating function's singularities
(see \cite[Chapter~VI]{Flajolet-Sedgewick} for an introduction of this method).
In particular, the constants $\iota_0$ and $\mu_0$ in Theorem~\ref{thm:final asym} appear, respectively, as the modulus and the multiplicities of the singularities closest to the origin.
\medskip

\noindent \textbf{Plan of the paper.} In Section \ref{sec arcs} we recall the terminology and background associated to arc systems, measured laminations and train tracks. Starting with the expression for the frequency $\FC_g(\gamma_0)$ as the Thurston measure of a fundamental domain for the action of $\Stab_{\Map(X)}(\gamma_0)$ on the set $\{\lambda\in\CM\CL(X),\ \iota(\lambda,\gamma_0)\le 1\}$, we derive in Section \ref{sec frequency} an expression for $\FC(\gamma_0)$ in terms of counting some explicit sets of simple multicurves, the stress being here on the word "explicit". This allows us, in Section \ref{sec:FrequencyCounting}, to express $\FC(\gamma_0)$ as a limit of counts of weighted arc systems in the surface $\Sigma=\Sigma(\gamma_0)$ filled by $\gamma_0$. In Section \ref{sec:kontsevich} we replace this limit of weighted counts by an integral. Since arc systems and ribbon graphs are essentially the same thing, it is here were the Konstsevich polynomial appears, leading to Theorem \ref{thm constant c in terms of intersection numbers}.

In Section \ref{sec:local type} we introduce formally the notions of local type and realization. We also introduce the dual graph to a realization and rephrase the statement of Theorem \ref{thm constant c in terms of intersection numbers} in those terms. While the notation in terms of the dual graph is less intuitive from the point of view of topology, it comes very handy when investigating large genus asymptotics. This is what we do in the pretty long and intense Section \ref{sec asymptotic}, whose content we just summarized under the rubric {\em Extracting asymptotics}.

Finally, in a much more relaxed Section \ref{sec int K}, we get Theorem \ref{thm:main type fixed int number} and Theorem \ref{kor asymp fix int} as corollaries of the results of the previous section. Finally, in the appendix we present some explicit calculations for simple curves and curves with one self intersection. 
\medskip

\noindent \textbf{Acknowledgments.} The authors are particularly indebted to Linxiao~Chen; the proof of Lemma~\ref{lem:Linxiao} is due to him. They are also very grateful to Elise Goujard for recognising the frequency of a simple curve as the common factor in the frequency of any curve in Theorem~\ref{thm:final asym}.
They thank A.~Alharbi, C.~Banderier, S.~Bronstein, Y.~Chaubet, A.~Contat, T.~Monteil, B.~Petri, M.~\"Unel and A.~Zorich for useful discussions, and are grateful to the Fields Institute, the Institut Henri Poinacr\'{e}, the Max Planck Institute in Leipzig, the University of Luxembourg, the Universit\'{e} de Rennes, the Universit\'{e} Sorbonne Paris Nord and the University of Toronto which gave them the opportunity to visit each other. 

This work is funded by the Luxembourg National Research Fund OPEN grant O19/13865598 through the first author, NSERC Discovery grant, RGPIN-05507 through the second author, the France 2030 program, Centre Henri Lebesgue ANR-11-LABX-0020-01 through the third author and the MPI MiS through the last author.

\tableofcontents

\section{Background}\label{sec arcs}
Through out this section, let $\Sigma$ be an oriented, compact, possibly disconnected surface all of whose connected components have boundary and are either annuli or have negative Euler characteristic. We denote by $\Sigma_{\ann}$ the union of the annular components of $\Sigma$ and its complement by $\Sigma_{\hyp}=\Sigma\setminus\Sigma_{\ann}$. For reference, endow $\Sigma_{\ann}$ with a flat metric with totally geodesic boundary and $\Sigma_{\hyp}$ with a hyperbolic metric, again with totally geodesic boundary. 

Under the mapping class group we understand the group
$$\Map(\Sigma)= {\Rfaktor{ \Diff^+(\Sigma)}{  \Diff^+_0(\Sigma)}},$$
where we do not assume that neither elements in $\Diff^+(\Sigma)$ nor in $\Diff^+_0(\Sigma)$ fix the boundary point-wise. In other words, $\Map(\Sigma)$ does not include Dehn-twists along the boundary components of $\Sigma$. 

Elements in $\Map(\Sigma)$ might permute the connected components of $\D \Sigma$. The pure mapping class group $\PMap(\Sigma)$ is the finite index subgroup of $\Map(\Sigma)$ preserving all individual boundary components. Since we are assuming that all components of $\Sigma$ have boundary, we automatically get that the pure mapping class group $\PMap(\Sigma)$ acts trivially on the set $\pi_0(\Sigma)$ of connected components of $\Sigma$. 

As customary in the field, we identify curves with their free homotopy classes. We allow curves to be boundary parallel. We stress this unusual convention:

\begin{defi} \label{def: Curves}
    Let $\Sigma$ be a surface with boundary, a \emph{curve} is a free-homotopy class of primitive and non-homotopically trivial immersed closed loops. A multicurve $\Gamma$ is a finite family of distinct curves.  Hence, for $\Gamma$ a multi-curve, $\gamma\in\Gamma$ is a curve.
\end{defi}

As always, the mapping class group acts on the sets of curves and multicurves, preserving the geometric intersection number $\iota(\cdot,\cdot)$. 

\subsection{The (flipped-) arc complex} \label{sec:arc complex}
A {\em topological arc system} is a collection of simple, disjoint, essential and non-parallel proper arcs in $\Sigma$ where proper means that the endpoints of the arcs lie in $\D \Sigma$---properly isotopic arc systems will be identified. In particular, each annular component of $\Sigma$ has a single arc system made of one arc. Every arc system in $\Sigma_{\hyp}$ has a unique representative which is (ortho)geodesic with respect to the fixed hyperbolic metric. A topological arc system is {\em maximal}  if it meets every annular component and  cuts the components of $\Sigma_{\hyp}$ into a collection of hexagons. It is {\em filling} if it cuts $\Sigma$ into a collection of polygons. Every maximal arc system is filling and has exactly $3\cdot\vert\chi(\Sigma)\vert+\vert\pi_0(\Sigma_{\ann})\vert$ components, where $\chi(\Sigma)$ is the Euler characteristic of $\Sigma$. We record this fact for future reference. 

\begin{fact} \label{fact: number of arcs}
    Any maximal arc system in a surface with boundary $\Sigma$ consists of $3\cdot\vert\chi(\Sigma)\vert+\vert\pi_0(\Sigma_{\ann})\vert$ distinct arcs.
\end{fact}

We will denote by $\CA(\Sigma)$ the countable set of (isotopy classes of) maximal arc systems in $\Sigma$. Note that
$$\CA(\Sigma)=\CA(\Sigma_1)\times\dots\times\CA(\Sigma_s)$$
where $\Sigma_1,\dots,\Sigma_s$ are the connected components of $\Sigma$. Since there is only one maximal arc system in each annular component we get that, as long as $\Sigma_{\hyp}\neq\emptyset$, the projection map $\CA(\Sigma)\to\CA(\Sigma_{\hyp})$ is a bijection. For later on use, let us introduce notation for its inverse: 
\begin{equation} \label{eq:completAS}
     \CA(\Sigma_{\hyp})\to\CA(\Sigma),\ \sigma \mapsto \tilde \sigma. 
\end{equation}

\subsubsection{The arc-complex} \label{}

A {\em weighted arc system} $\bar a$ is a topological arc system $\sigma$ together with a non-negative weight associated to each one of its components. The {\em support} of a weighted arc system is the set of arcs of positive weight in the underlying topological arc system. A weighted arc system is {\em filling} if its support is. Equivalently, $\bar a$ is filling if we have $\iota(\bar a,\gamma)>0$ for every curve $\gamma$ in $\Sigma$.

\begin{rmk*}
    We can associate to every weighted arc system a new weighted arc system by forgetting the components of the underlying topological arc system with $0$ weight. Slightly abusing terminology, we will not differentiate between these two arc systems. The reason is that we will be mostly interested in the intersections between arc systems and curves, and from that point of view both arc systems are identical.
\end{rmk*}

We will denote by $\bar\BA(\Sigma)$ the space of all weighted arc systems and by 
$$\BA(\Sigma)=\{\bar a\in\bar\BA(\Sigma)\text{ filling}\}$$
the subset consisting of filling weighted arc systems. Although this terminology might differ from what is sometimes used in the literature, we will refer to $\BA(\Sigma)$ as the {\em arc complex} of $\Sigma$.

We endow $\bar\BA(\Sigma)$ with the finest topology with respect to which the map
\begin{center}
  \begin{tabular}{ccl}
    $\BR_{\ge 0}^\sigma$ &$\to $&$\bar\BA(\Sigma)$\\
    $\bar w=(w_a)_{a\in\sigma}$&$\mapsto$&$\sum_{a\in\sigma}w_a\cdot a$
\end{tabular}  
\end{center}
is a homeomorphism onto its image $\bar\BA(\sigma)$ for every maximal topological arc system $\sigma\in\CA(\Sigma)$. It follows from Fact \ref{fact: number of arcs} that $\bar\BA(\Sigma)$ is the union of countably many quadrants 

\begin{equation} \label{eq dim no flip}
    \bar\BA(\sigma)\simeq\BR_{\ge 0}^\sigma\simeq\BR_{\ge 0}^{3\cdot \vert\chi(\Sigma)\vert+\vert\pi_0(\Sigma_{\ann})\vert},
\end{equation} glued together via linear maps along faces of lower dimension. The arc complex $\BA(\Sigma)$ is an open full dimensional subset of $\bar\BA(\Sigma)$. 

If $\sigma\in\CA(\Sigma)$ is a maximal arc system, then we denote by $\BA(\sigma)=\bar\BA(\sigma)\cap\BA(\Sigma)$ the set of those filling weighted arc systems whose support is contained in $\sigma$. Since the cover 
$$\BA(\Sigma)=\cup_{\sigma\in\CA(\Sigma)}\BA(\sigma)$$
is locally finite, it follows that
\begin{equation}
    \label{DimArcComplex}
    \dim(\BA(\Sigma))=3\cdot\vert\chi(\Sigma)\vert+\vert\pi_0(\Sigma_{\ann})\vert.
\end{equation}
It also follows that we have 
\begin{equation}\label{eq product decompostion of arc complex}
\BA(\Sigma)=\BA(\Sigma_1)\times\dots\times\BA(\Sigma_s)
\end{equation}
where again $\Sigma_1,\dots,\Sigma_s$ are the connected components of $\Sigma$. 

The mapping class group $\Map(\Sigma)$ acts on the set of topological arc systems. The action $\Map(\Sigma)\actson\CA(\Sigma)$ on the set of maximal arc systems has finitely many orbits and their stabilizers are finite. The mapping class group also acts via homeomorphisms on $\bar\BA(\Sigma)$, and this action preserves $\BA(\Sigma)$. In fact, the induced action $\Map(\Sigma)\actson\BA(\Sigma)$ on the arc complex is proper. In particular the stabilizer $\Stab_{\Map(\Sigma)}(\bar a)$ of every $\bar a\in\BA(\Sigma)$ is finite.

\subsubsection{Boundary map}
We have a well-defined map, the {\em boundary map},  
\begin{equation}\label{eq boundary map}
{\bf b}:\bar\BA(\Sigma)\to\BR_{\ge 0}^{\D\Sigma}
\end{equation}
associating to each weighted arc system and to each component of $\D S$ the sum, with multiplicity, of the weights of the arcs incident to the said component. Here, ``with multiplicity'' means that we count with weight 2 if both ends of the arc lie on our component, with weight 1 if only one of them does, and finally with weight 0 otherwise. 

The boundary map ${\bf b}$ is continuous and invariant under the pure mapping class group $\PMap(\Sigma)$. Moreover, it is linear in the sense that if $\sigma\in\CA(S)$ is a maximal arc system, then the composition
$$\BR_{\ge 0}^\sigma\simeq\bar\BA(\sigma)\subset\bar\BA(\Sigma)\stackrel{{\bf b}}\longrightarrow\BR_{\ge 0}^{\D\Sigma}$$
is the restriction of a linear map $\BR^\sigma\to\BR^{\D\Sigma}$. If $\Sigma=\Sigma_{\hyp}$ does not have annular components, then it has maximal rank. We record this fact:

\begin{lem}\label{lem epimorphism}
Let $\sigma\in\CA(\Sigma_{\hyp})$ be a maximal arc system in $\Sigma_{\hyp}$. The restriction to $\bar\BA(\sigma)\simeq\BR_{\ge 0}^\sigma$ of the boundary map ${\bf b}:\bar\BA(\Sigma_\hyp)\to\BR_{\ge 0}^{\D\Sigma_\hyp}$ is the restriction of a linear map of maximal rank.\qed
\end{lem}

Lemma \ref{lem epimorphism} implies that for $\sigma\in\CA(\Sigma_{hyp})$ and $\bar b\in\BR_{\ge 0}^{\D \Sigma_\hyp}$, the set 
$$\BA(\sigma,\bar b)\stackrel{\text{def}}=\{\bar a\in\BA(\sigma)\text{ with }{\bf b}(\bar a)=\bar b\}$$
is the intersection of the quadrant $\BA(\sigma)$ with an affine subspace of codimension $\vert\D \Sigma_\hyp\vert$. This means that 
$$\dim(\BA(\sigma,\bar b))\le 3\cdot\vert\chi(\Sigma_{\hyp})\vert-\vert\D \Sigma_\hyp\vert.$$

One should point out that, since we are only intersecting a quadrant, the inequality can be strict, even if $\sigma$ is maximal. For example, suppose that $\gamma_1$ and $\gamma_2$ are two boundary components of $\Sigma_{\hyp}$ and that $\sigma$ is such that every arc with an end in $\gamma_1$ has its second end in $\gamma_2$. Then we have $\BA(\sigma,\bar b)=\emptyset$ for any $\bar b\in\BR_{\ge 0}^{\D\Sigma_{\hyp}}$ with $\bar b(\gamma_1)>\bar b(\gamma_2)$. There is also $\bar b'$ for which $\BA(\sigma,\bar b')\neq\emptyset$ but has smaller dimension than expected.  \medskip

A weighted arc system is {\em integral} if all the coefficients are not only non-negative but also integral. We denote by $\BA_\BZ(\Sigma)\subset\BA(\Sigma)$ the discrete subset of integral filling weighted arc systems. Accordingly, we denote by $\BA_\BZ(\Sigma,\bar b)=\BA_\BZ(\Sigma)\cap\BA(\Sigma,\bar b)$ the set of integral weighted arc systems $\bar a$ with ${\bf b}(\bar a)=\bar b\in\BR_{\ge 0}^{\D\Sigma}$. 

There are simple conditions on $\bar b$ which ensure that the set $\BA_\BZ(\Sigma,\bar b)$ is empty:
\begin{enumerate}
    \item If $\bar b\notin\BZ_{\ge 0}^{\D \Sigma}$, then $\BA_\BZ(\Sigma,\bar b)=\emptyset$.
    \item If there is a connected component $\Sigma'$ of $\Sigma$ such that $\Vert \bar b_{\vert\D \Sigma'}\Vert$ is odd, then $\BA_\BZ(\Sigma,\bar b)=\emptyset$. Here $\bar b_{\vert\D \Sigma'}$ is the vector in $\BR_{\ge 0}^{\D \Sigma'}$ whose entries are the entries of $\bar b$ corresponding to $\D \Sigma'\subset\D\Sigma$.
    \item If there are an annular component $\Sigma'\subset \Sigma_{\ann}$ of $\Sigma$ with $b_{\vert\D \Sigma'}=(a,a')$ with $a\neq a'$, then $\BA_\BZ(\Sigma,\bar b)=\emptyset$.
\end{enumerate}

\begin{defi}\label{def admissible}
We will say that $\bar b\in\BZ_{\ge 0}^{\D \Sigma}$ is {\em admissible} if it does not satisfy (1), (2), and (3) above.
\end{defi}

Note that if $\Sigma$ is the disjoint union of $\Sigma_1$ and $\Sigma_2$ then $\bar b =(\bar b_1,\bar b_2)\in\BZ_{\ge 0}^{\D\Sigma_1}\times\BZ_{\ge 0}^{\D\Sigma_2}=\BZ_{\ge 0}^{\D\Sigma}$ is admissible if and only if $\bar b_1$ and $\bar b_2$ are.
\medskip

Still with the same notation, observe that for any $\bar a\in\BA(\Sigma,\bar b)$, we have $\Vert\bar b\Vert=2\cdot\Vert\bar a\Vert$ where $\Vert\cdot\Vert$ stands for the $\ell_1$-norm. It follows that for all $\bar b\in\BR_{\ge 0}^{\D\Sigma}$ and any maximal arc system $\sigma\in\CA(\Sigma)$, there are at most finitely many elements in $\BA_\BZ(\sigma,\bar b)$. Denote by 
\[ N_{\sigma}(\bar b) = | \BA_\BZ(\sigma,\bar b )  |    \]
the number of such elements. Noting now that the whole arc complex is covered by finitely many pure mapping class groups orbits of sets of the form $\BA(\sigma)$ with $\sigma\in\CA(\Sigma)$, define
\begin{equation}\label{eq number of integral arc systems with given boundary values}
N_{\Sigma}(\bar b)= 
\sum_{\sigma \in\Lfaktor{\CA(\Sigma,\bar b)}{\PMap(\Sigma)}}
\frac {N_\sigma(\bar b)}{\vert\Stab_{\PMap(\Sigma)}(\sigma)\vert}.
\end{equation}  
This quantity will play a key role in this paper. We record two facts which one should keep in mind:
\begin{itemize}
    \item $\bar b$ is admissible if and only if $N_\Sigma(\bar b)\neq 0$.
    \item If $\Sigma$ is the disjoint union of $\Sigma_1$ and $\Sigma_2$ then 
    $N_\Sigma(\bar b)=N_{\Sigma_1}(\bar b_1)\cdot N_{\Sigma_2}(\bar b_2)$
    for any $\bar b=(\bar b_1,\bar b_2)\in\BZ_{\ge 0}^{\D\Sigma_1}\times\BZ_{\ge 0}^{\D\Sigma_2}=\BZ_{\ge 0}^{\D\Sigma}$.
\end{itemize}

\subsubsection{Flips} \label{sec:flips}
Suppose that we have an involution $\flip:\D\Sigma\to\D\Sigma$ which exchanges both components of the boundary of every annular component of $\Sigma$. We impose no further conditions on $\flip$. In particular, $\flip$ might have fixed points, as it will in all cases of interest. Note however that we have
$$\flip(\D \Sigma_\ann)=\D \Sigma_\ann\text{ and hence }\flip(\D \Sigma_\hyp)=\D \Sigma_\hyp.$$

The involution $\flip$ of $\D\Sigma$ induces a linear involution of $\BR_{\ge 0}^{\D\Sigma}$ which we will denote by $\flip^\BR$. We will denote by 
$$\bar\BA(\Sigma,\flip)=\{\bar a\in\bar\BA(\Sigma)\text{ with }\flip^\BR({\bf b}(\bar a))={\bf b}(\bar a)\}$$
the set consisting of those weighted arc systems whose image under the boundary map {\bf b} as in \eqref{eq boundary map} is fixed by the involution $\flip^\BR$. Said differently
$$\bar\BA(\Sigma,\flip)=\text{ preimage under }{\bf b}\text{ of the fixed point set of }\flip^\BR\text{ acting on }\BR^{\D\Sigma}.$$
As we pointed out above, $\flip$ preserves the set of boundary components of the hyperbolic part of $\Sigma$, inducing thus an involution $\flip_{\hyp}$ of $\D\Sigma_{\hyp}$ is an involution of the latter set. With this notation we have the decomposition
\begin{equation}\label{eq I have a lot of nice cats} 
\bar\BA(\Sigma,\flip)=\bar\BA(\Sigma_{\hyp},\flip_{\hyp})\times\bar\BA(\Sigma_{\ann}) \equiv \bar\BA(\Sigma_{\hyp},\flip_{\hyp})\times \BR_{\ge 0}^{|\pi_0(\Sigma_{\ann})|}.
\end{equation}
Consistently with our earlier notation, set 
$$\BA(\Sigma,\flip)=\BA(\Sigma)\cap\bar\BA(\Sigma,\flip)\text{ and }\BA_\BZ(\Sigma,\flip)=\BA_\BZ(\Sigma)\cap\bar\BA(\Sigma,\flip).$$ 
Along the same lines set
$$\BA(\sigma,\flip)=\BA(\sigma)\cap\BA(\Sigma,\flip)\text{ and }\BA_{\BZ}(\sigma,\flip)=\BA(\sigma)\cap\BA_{\BZ}(\Sigma,\flip)$$
for any $\sigma\in\CA(\Sigma)$. The subsets $\bar\BA(\Sigma,\flip)$, $\BA(\Sigma,\flip)$, and $\BA_{\BZ}(\Sigma,\flip)$ of $\bar\BA(\Sigma)$ are invariant under the pure mapping class group $\PMap(\Sigma)$.

Lemma \ref{lem epimorphism} implies that for every maximal arc system $\sigma\in\CA(\Sigma_{\hyp})$, the set $\bar\BA(\sigma,\flip)=\bar\BA(\sigma)\cap\bar\BA(\Sigma_{\hyp},\flip)$ is the intersection of the positive quadrant $\bar\BA(\Sigma_{\hyp},\flip)\simeq\BR_{\ge 0}^\sigma$ with a linear subspace of codimension $\vert\D \Sigma_{\hyp}\vert-\vert\Lfaktor{\D \Sigma_{\hyp}}{\flip}\vert$. This means that the codimension of $\bar\BA(\sigma,\flip)$ in $\bar\BA(\sigma)$ is at least $\vert\D\Sigma_\hyp\vert-\vert\Lfaktor{\D\Sigma_\hyp}{\flip}\vert$. For certain maximal arc systems $\sigma$, we have equality. This is for example the case if $\sigma$ is such that for every component of $\D \Sigma_{\hyp}$ there is a component of $\sigma$ which has both endpoints there in. In light of \eqref{eq I have a lot of nice cats} we get 
\begin{align} \label{eq defi dim flip} 
\dim(\bar\BA(\Sigma,\flip))&= 3\cdot\vert\chi(\Sigma_{\hyp})\vert+\vert\pi_0(\Sigma_{\ann})\vert-(\vert\D\Sigma_\hyp\vert-\vert\Lfaktor{\D\Sigma_\hyp}{\flip}\vert )\\
&=3\cdot\vert\chi(\Sigma)\vert-\vert\D\Sigma_\hyp\vert+\vert \Lfaktor{\D\Sigma}{\flip}\vert. \notag
\end{align}

\subsubsection{Measures on the arc complex}  \label{sec:measure in arcs}

The space $\bar\BA(\Sigma,\flip)$ is endowed with a canonical $\PMap(\Sigma)$-invariant Radon measure of full measure. It is obtained as a weak-$*$-limit. Let us recall what that means. First, a Borel measure is Radon if it is locally finite, and both inner and outer regular. Second, the weak-$*$-topology on the space of Radon measures on a locally compact space is such that a sequence $(\FM_n)_{n\in\BN}$ of Radon measures converges to another such measure $\FM$ if and only if we have
$$\int f\ d\FM=\lim_{n\to\infty}\int f\ d\FM_n$$
for every compactly supported continuous function $f$. The convergence is expressed in terms of compactly supported continuous functions, but one gets from the Portmanteau theorem that $\FM_n\to\FM$ if and only if we have $\FM(V)=\lim_n\FM_n(V)$ for any measurable subset $V$ satisfying $\FM(\overline V \setminus \mathring{V}) = 0$, where $\overline V$ an $\mathring{V}$ are respectively the closure and the interior of $V$. A last general remark about measures: independently of the ambient space we are in, we will always denote by $\delta_x$ the Dirac probability measure centered at $x$. 
\medskip

To get measure on $\bar\BA(\Sigma,\flip)$ we proceed as in the construction of the Thurston measure on the space of measured laminations---see for example \cite[Theorem 4.16]{book}.
Suppose that $\sigma\in\CA(\Sigma)$ is a maximal arc system and identify, as above, $\bar\BA(\sigma,\flip)$ with the intersection of the positive quadrant $\BR_{\ge 0}^\sigma$ with a linear subspace $V(\sigma,\flip)$ of $\BR^\sigma$. The vector space $V(\sigma,\flip)$ is rational, in the sense that it is given by a collection of equations with rational (in fact, integral) coefficients. It follows that the weak-$*$-limit 
\begin{equation*} 
    \FM_{V(\sigma,\flip)}=\lim_{L\to\infty}\frac 1{L^{\dim(\BA(\Sigma,\flip))}}\sum_{\bar a\in V_\BZ(\sigma,\flip)} \delta_{\bar a}
\end{equation*}
where $V_\BZ(\sigma,\flip)$ is the set of integral points in $V(\sigma,\flip)$, exists in the space of Radon measures on $V(\sigma,\flip)$. It is non-zero as soon as $\dim(\BA(\Sigma,\flip))=\dim(V(\sigma,\flip))$ and the restriction $\FM_{\BA(\sigma,\flip)}$ of $\FM_{V(\sigma,\flip)}$ to $\BA(\Sigma,\flip)$ is the Radon measure given by
$$\FM_{\BA(\sigma,\flip)}=\lim_{L\to\infty}\frac 1{L^{\dim(\BA(\Sigma,\flip))}}\sum_{\bar a\in \BA_\BZ(\sigma,\flip)} \delta_{\bar a}.$$
Finally, since the sets $\BA(\sigma,\flip)$ with $\sigma\in\CA(\Sigma)$ form a locally finite cover of $\BA(\Sigma,\flip)$ and they intersect along subspaces of strictly lower dimension, we get that also the weak-$*$-limit
    \begin{equation}\label{eq measure simplex arc}
    \FM_{\BA(\Sigma,\flip)}=\lim_{L\to\infty}\frac 1{L^{\dim(\BA(\Sigma,\flip))}}\sum_{\bar a\in\BA_\BZ(\Sigma,\flip)} \delta_{\bar a}
    \end{equation}
exists in the space of Radon measures on $\BA(\Sigma,\flip)$. The limit measure $\FM_{\BA(\Sigma,\flip)}$ is $\Map(\Sigma,\flip)$-invariant because all members of the sequence are, and it is non-trivial because already its restriction $\FM_{\BA(\sigma,\flip)}$ to $\BA(\sigma,\flip)$ is non-trivial for any $\sigma\in\CA(\Sigma)$ with $\dim(\BA(\sigma,\flip))=\dim(\BA(\Sigma,\flip))$. Summing up, we have the following:

\begin{prop}\label{prop existence of thurston like measure on flip-invariant arc complex}
    The limit \eqref{eq measure simplex arc} exists with respect to the weak-$*$-topology on the space of Radon measures on $\BA(\Sigma,\flip)$. Moreover, the limiting measure $\FM_{\BA(\Sigma,\flip)}$ is non-trivial and invariant under $\PMap(\Sigma)$.
\end{prop}

Note at this point that above we were careful enough to consider all measures as defined on $\BA(\Sigma,\flip)$ instead of $\bar\BA(\Sigma,\flip)$. The reason is that, since $\bar\BA(\Sigma,\flip)$ is not locally compact, it does not make much sense to take weak-$*$-limits of measures there. On the other hand, the subset $\bar\BA(\sigma,\flip)$ of $\bar\BA(\Sigma,\flip)$ is locally compact for any maximal arc system $\sigma\in\CA(\Sigma)$. Since $\bar\BA(\sigma,\flip)\setminus\BA(\sigma,\flip)$ is just a union of finitely many lower-dimensional subspaces we have the following:

\begin{lem}
    For any $\sigma\in\CA(\Sigma)$ the limit
    $$\FM_{\bar\BA(\sigma,\flip)}=\lim_{L\to\infty}\frac 1{L^{\dim(\bar\BA(\Sigma,\flip))}}\sum_{\bar a\in \bar\BA_\BZ(\sigma,\flip)} \delta_{\bar a}$$
    exists in the space of Radon measures on $\bar\BA(\sigma,\flip)$ and $\FM_{\bar\BA(\sigma,\flip)}(U)=\FM_{\BA(\sigma,\flip)}(U\cap\BA(\sigma,\flip))$ for all $U\subset\bar\BA(\sigma,\flip)$.\qed
\end{lem}

To conclude this discussion, note that via the decomposition \eqref{eq I have a lot of nice cats} of $\BA(\Sigma,\flip)$ we get a decomposition of the measure $\FM_{\BA(\Sigma,\flip)}$ as the product measure of $\FM_{\BA(\Sigma_{\hyp},\flip_{\hyp})}$ and Lebesgue measure:

\begin{lem} \label{lem: measure desintegration} 
    Denote by $\flip_{\hyp}$ the restriction of $\flip$ to $\D\Sigma_{\hyp}$ and as in \eqref{eq I have a lot of nice cats} consider the decomposition $\bar\BA(\Sigma,\flip)=\bar\BA(\Sigma_{\hyp},\flip_{\hyp})\times \BR_{\ge 0}^{|\pi_0(\Sigma_{\ann})|}$. Then 
    \[  \FM_{\BA(\Sigma,\flip)}=\FM_{\BA(\Sigma_{\hyp},\flip_{\hyp})}\otimes\Leb^{|\pi_0(\Sigma_{\ann})|}  \]
    where $\Leb^n$ stands for the standard Lebesgue measure on $\BR^n$.
\end{lem}

Along the same lines, note that as soon as the flip map does not exchange any boundary component of $\Sigma_{\hyp}$ we can identify $\bar\BA(\sigma,\flip)=\bar\BA(\sigma)\equiv \BR_{\ge0}^{3|\chi(\Sigma)|+|\pi_0(\Sigma_{\ann})|}$ for any filling arc system $\sigma\in\CA(\Sigma)$. After this identification we have
\begin{equation} \label{eq: measure=lebesgue}
\FM_{\BA(\sigma)}=\Leb^{3|\chi(\Sigma)|+|\pi_0(\Sigma_{\ann})|}.
\end{equation}

\subsection{Measured laminations and train-tracks} \label{sec:ML}

We discuss here the few facts we will need about measures laminations and train tracks. The interested reader can find more on these topics in \cite{Penner-Harer, Hatcher, book}.
\medskip

If $X$ is a closed hyperbolic surface then a \emph{measured lamination} on $X$ is a closed subset $\lambda$ of $X$, the {\em support }of the measure lamination, foliated by disjoint simple complete geodesics  called \emph{leaves} together with a transverse measure. Here, a transverse measure is family $(\lambda_t)_{t\in\mathfrak{T}}$ of Radon measures indexed by the set $\mathfrak{T}$ of smooth segments in general position with respect to $\lambda$ and with the following properties. First, the measure $\lambda_t$ is positive if and only if $\lambda\cap t\neq \emptyset$. Second, if $t'\subset t$, then $\lambda_{t'}=\lambda_t\vert_{ t'}$. Finally and most importantly, if $t$ and $t'$ are homotopic to each other relative to $\lambda$, then $\lambda_{t'}$ is the push forward of $\lambda_{t}$ under the final map of this homotopy. We denote by $\CM\CL(X)$ the space of measured laminations endowed by the topology induced by the weak-$*$ topology on the spaces of Radon measures on the individual arcs. 

\subsubsection{Laminations carried by train tracks}\label{sec:Lam carried by TT} Recall that a \textit{train-track} on $X$ is a smoothly embedded 1-dimensional complex whose complement contains neither disks, once punctured disks, monogons, or bigons--see \cite{Penner-Harer} for details. 

\begin{figure}[ht!]
    \centering
    \includegraphics[width=0.6\linewidth]{ 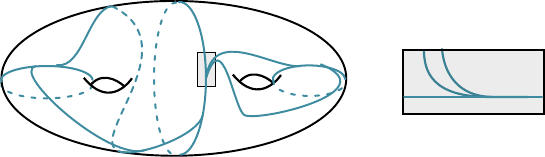}
    \caption{Example of train-track on a surface}
    \label{fig:TT}
\end{figure}

Train-tracks are a very important tool to study measured laminations. A measured lamination is {\em carried} by a train track $\tau$ if its support can be smoothly isotoped to $\tau$. Let us briefly recall the structure of the set $\CM\CL(\tau)$ of all measured laminations carried by $\tau$.

First note that we can decompose the set of half-edges of $\tau$ adjacent to a vertex $v$ into two disjoint sets with the property that two half-edges belong to the same one if and only if they are tangent at $v$. We can thus speak of incoming and outgoing edges $v$. Taking the difference between the sum of the entries corresponding to the incoming half edges and that of the entries corresponding to the outgoing half edges, we get a linear map $\BR^{E(\tau)}\to\BR$ where $E(\tau)$ is the of edges of $\tau$. We say that $w\in\BR^{E(\tau)}$ is a solution of the {\em switch equations} if it belongs to the kernel of all those maps. Denote by 
$$W(\tau)=\{w\in \BR_{\ge 0}^{E(\tau)}\text{ solution of the switch equations of }\tau\}$$ 
the quadrant of non-negative solutions of the switch equations. Given now a point $w$ in $W(\tau)$ take for each edge $e$ of $\tau$ a horizontally foliated Euclidean rectangle of height $w_e$ and glue them (via partially defined isometries) following the pattern determined by the train track. In this way we get a partial foliation of $X$ endowed with a transverse measure. To obtain a lamination, pull first tight the non-singular horizontal leaves of this foliation, and take then the closure of the so obtained subset of $X$. The push forward of the transverse measure on the foliation is a transverse measure on the lamination. See Figure \ref{fig:TTtoML} for an illustration of this process.

\begin{figure}[ht!]
    \centering
    \includegraphics[width=0.4\linewidth]{ 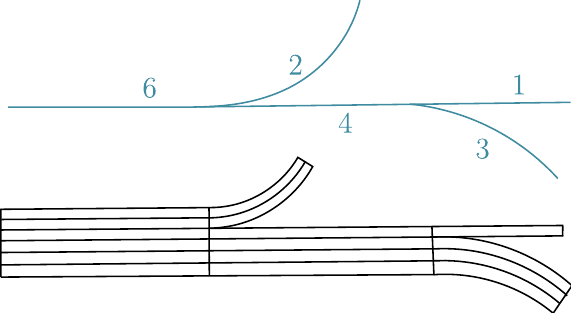}
    \caption{Getting a measured lamination from a solution of the switch equations.}
    \label{fig:TTtoML}
\end{figure}

Denote by $\Phi_\tau(w)\in\CM\CL(X)$ the measured lamination associated to $w\in W(\tau)$ and note that, very much by construction, $\Phi_\tau(w)$ is carried by $\tau$. Indeed, the map $\Phi_\tau:W(\tau)\to\CM\CL(X)$ is a homeomorphism onto the set $\CM\CL(\tau)\subset\CM\CL(X)$ of measured laminations carried by $\tau$. If $\tau$ is maximal, that is if it is not properly contained in any other train track, then $W(\tau)$ has dimension $6g-6$ where $g$ is the genus of the closed surface $X$. Moreover, $\CM\CL(\tau)$ is the closure of its interior in $\CM\CL(X)$. Via this construction one obtains a mapping class group invariant piecewise linear structure on $\CM\CL(X)$.

\begin{sat}[Thurston] \label{thm:thurston ML}Let $X$ be a closed surface of genus $g\ge 2$. The space $\CM\CL(X)$ is homeomorphic to $\BR^{6g-6}$ and for every $\lambda\in\CM\CL(X)$, there is a maximal train track $\tau_\lambda$ such that $\lambda$ belongs to the interior of $\CM\CL(\tau_\lambda)$. Moreover, there are finitely many maximal train tracks $\tau_1....\tau_n$ with the following properties: $\CM\CL(X)=\bigcup\limits_i \CM\CL(\tau_i)$ and $\Phi_i^{-1}(\CM\CL(\tau_i)\cap\CM\CL(\tau_j))$ is a lower dimensional face of the cone $W(\tau_i)$ for all $i\neq j$.  \hfill $\blacksquare$
\end{sat}

\subsubsection{Measure on the space of measured laminations} \label{sec:Thurston Measure}
The space $\CM\CL(X)$ is endowed with a natural Radon measure, the {\em Thurston measure} $\FM_{\Thu}$. If $\tau$ is a maximal train track then $\FM_{\Thu}$ corresponds under the identifications $\CM\CL(\tau)\simeq W(\tau)$ to Lebesgue measure on $W(\tau)$. Intrinsically, the Thurston measure can be described as follows: for $U\subset\CM\CL(X)$ piece-wise linear (or more generally, with $\FM_{\Thu}(\D U)=0$) we have
\begin{equation}\label{eq def thurston measure}
\FM_{\Thu}(U)=\lim_{L\to\infty}\frac {\vert L\cdot U\cap\CM\CL_\BZ(X)\vert}{L^{6g-6}}
\end{equation}
where $\CM\CL_\BZ(X)$ stands for those measured laminations whose support is a submanifold of $X$ and which give a weight in $\BZ_{\ge 0}$ to every component. We refer to elements in $\CM\CL_\BZ(X)$ as simple integral weighted multicurves. See \cite[Chapter 4]{book} for the proof of the existence of the Thurston measure, and for some of its properties, such as the fact that it is mapping class group invariant. One can also refers to \cite{MT2} for more details about the different definitions and the link between them.

\subsubsection{Adapted train tracks}\label{sec: adapted TT}
Suppose now that $\Gamma\subset X$ is a simple multicurve and let $\CN(\Gamma)$ be a small open regular neighborhood of $\Gamma$. In particular, $\CN(\Gamma)$ is the union of open annuli. Consider  $\alpha\in\CA(X\setminus\CN(\Gamma))$ a maximal arc system in $X\setminus\CN(\Gamma)$. In \cite[Chapter 4]{book} the authors explained how to associate to the pair $(\Gamma,\alpha)$ a specific train track $\tau_{\alpha}\subset X$: the {\em $\Gamma$-adapted train track associated to $\alpha$}. Let us recall now this construction.

We start by fixing a collection of points $p_1,\dots,p_{2\vert\Gamma\vert}$, one in each connected components of the boundary $\D\CN(\Gamma)$ of the regular neighborhood $\CN(\Gamma)$. Then the train track $\tau_\alpha$ is such that:
\begin{itemize}
\item[(i)] the set $\{p_1,\dots,p_{2\vert\Gamma\vert}\}$ is the set of switches of $\tau_{\alpha}$,
\item[(ii)] the intersection of $\tau_{\alpha}$ with each one of the components of $\CN(\Gamma)$ is as in Figure \ref{ch4-fig4}, and
\item[(iii)] the intersection of $\tau_{\alpha}$ with each component of $X\setminus\CN(\Gamma)$ is a collection of simple arcs representing the arc system $\alpha$.
\end{itemize}
\begin{figure}[h!]
\begin{tikzpicture}
\path (0,0) node {\includegraphics[width=4cm]{ 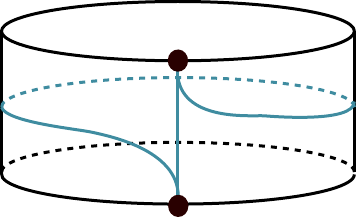}} ;
\draw(0,8mm) node {$p_i$}
(0mm,-14mm) node {$p_{i+1}$}; 
\end{tikzpicture} 
\caption{Local model for adapted train tracks in an annulus}
\label{ch4-fig4}
\end{figure}
See Figure \ref{ch4-fig5} for an example of a train track adapted to a separating curve in a surface of genus $2$.
\begin{figure}[h!]
\includegraphics[width=7cm]{  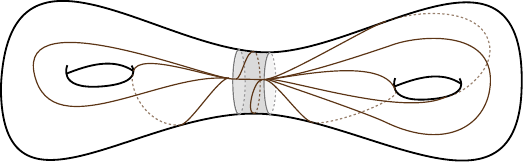}
\caption{Example of a train track adapted to $\Gamma$, a.k.a. $\Gamma$--adapted train track, where $\Gamma$ is the core curve of the shaded annulus. If we remove this annulus, then what it left of the train-track is a collection of arcs representing the given arc system.}
\label{ch4-fig5}
\end{figure}
We record here some important fact about adapted train tracks.
\begin{fact} \label{fact:Adapted TT are incompatible}
    Let $\alpha,\beta$ be distinct maximal arc systems in $X\setminus\CN(\Gamma)$. The train track $\CM\CL(\tau_\alpha)$ is maximal. The intersection $\CM\CL(\tau_\alpha)\cap\CM\CL(\tau_\beta)$ occurs along faces of lower dimension, in particular it has vanishing Thurston measure. 
\end{fact}

\section{The frequency \texorpdfstring{$\mathfrak{c}_g(\gamma_0)$}{c_X}}\label{sec frequency}
We are interested in non-simple curves in closed surfaces. The case of genus 2 will ask a specific treatment, we will keep track of the needed modifications in remarks along the document, for a first lecture we suggest to reader to skip those technicalities and focus on the case of higher genus. 

In this section let $X$ be a closed orientable genus $g\ge 3$ hyperbolic surface, let $\gamma_0$ be a multicurve in $X$, and denote by $\Sigma\subset X$ be the smallest $\pi_1$-injective subsurface containing $\gamma_0$ with the property that no annular component of $\Sigma$ is adjacent to an annular component of $X\setminus\Sigma$. The surface $\Sigma$ is unique up to isotopy, see an example in figure \ref{fig:exampleSigma}.  

\begin{figure}[!h]
\centering
\begin{tikzpicture}
\path (0,0) node {\includegraphics[width=100mm]{ 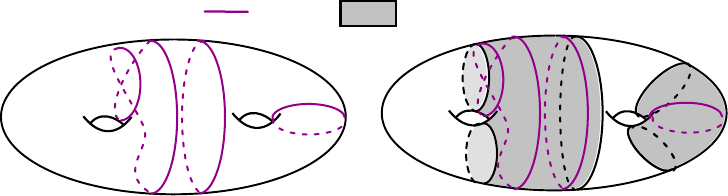}} ;
\draw
(9mm,11.5mm) node {$\Sigma$} 
(-12mm,11.5mm) node {$\gamma_0$};
\end{tikzpicture} 
\caption{Subsurface associated to a multicurve}
    \label{fig:exampleSigma}
\end{figure}

Note that $\gamma_0$ fills $\Sigma$ and that both $\Sigma$ and $X\setminus\Sigma$ are allowed to have annular components. Letting $\CN(\Sigma)$ be an open regular neighborhood of $\Sigma$, we will denote by
\begin{equation}\label{eq:defZ}
    Z=(X\setminus\CN(\Sigma))_{\hyp}
\end{equation}
the union of all the hyperbolic components of $X\setminus\CN(\Sigma)$. Note that $\CN(\Gamma):=X\setminus(\Sigma_{\hyp}\cup Z)$ is a regular neighborhood of the multicurve
\begin{equation}\label{eq:defGamma}
    \Gamma=\text{multicurve consisting of those simple curves in }X\text{ homotopic into }\D\Sigma.
\end{equation}
Note that if $\gamma_0$ has an isolated simple components $\gamma$ then $\Sigma$ will contain annular components and $\gamma$ will also be a part of $\Gamma$. See Figure \ref{fig:full description} for a full illustration of the decomposition of $\Sigma$ induced by $\gamma_0$ we just described.
It will be key that $\Sigma,Z$ and $\Gamma$ are uniquely determined by $\gamma_0$. In particular they are invariant under the stabilizer $\Stab_{\Map(X)}(\gamma_0)$ in the mapping class group of $X$.

\begin{figure}[!h]
\centering
\begin{tikzpicture}
\path (0,0) node {\includegraphics[width=100mm]{  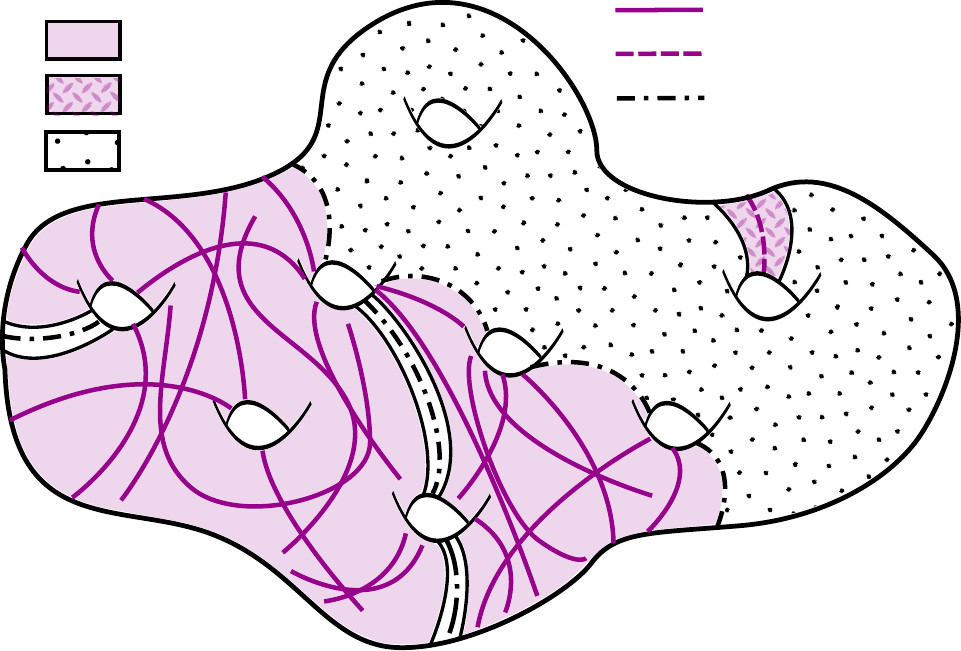}} ;
\draw
(29mm,23mm) node {$\Gamma \setminus \gamma_0$} 
(31mm,28mm) node {$\Gamma\cap\gamma_0^\ann$}
(28mm,33mm) node {$\gamma_0$}
(-32mm,29mm) node {$\Sigma_{\hyp}$}
(-32mm,23.5mm) node {$\Sigma_{\ann}$}
(-35mm,18mm) node {$Z$};
\end{tikzpicture} 
\caption{A non-filling curve $\gamma_0$, the associated surfaces $\Sigma$ and $Z$, the corresponding multicurve $\Gamma$ with $\gamma_0^\ann=\gamma_0\cap\Sigma_\ann$}\label{fig:full description}
\end{figure}

\subsection{Counting curves}
Recall that the mapping class group $\Map(X)$ acts on the set of multicurves in $X$. Two multicurves are {\em of the same type} if they belong to the same mapping class group orbit. In other words, $\Map(X)\cdot\gamma_0$ is the set of multicurves of type $\gamma_0$. 13

The starting point of our discussion is the following counting results for the number of multicurves $\gamma\in\Map(X)\cdot\gamma_0$ of length $\ell_X(\gamma)\le L$. The length of a curve is the hyperbolic length of its geodesic representative in $X$ and the length function extends to multicurves by linearity.

\begin{sat}[Mirzakhani \cite{Maryam simple, Maryam general curves}, Erlandsson-Souto \cite{book}]\label{thm counting}
For every $g\ge 2$ there is a positive constant $\FB_g$ such that for every genus $g$ hyperbolic surface $X$  and every multicurve $\gamma_0$ in $X$ there are positive constants $\FC(X)$ and $\FC_g(\gamma_0)$ with
$$\lim_{L\to\infty}\frac 1{L^{6g-6}}\vert\{\gamma\in\Map(X)\cdot\gamma_0,\ \ell_X(\gamma)\le L\}\vert=\frac{\FC_g(\gamma_0)\cdot\FC(X)}{\FB_g}.$$
Note that $\FC_g(\gamma_0)$ depends only on $\gamma_0$ and the genus of $X$. \hfill $\blacksquare$
\end{sat}

Let us discuss for a moment the history of this theorem. Mirzakhani considered first, in \cite{Maryam simple}, the case where $\gamma_0$ is simple. She then considered the general case in \cite{Maryam general curves}, proving only the existence of the limit.
A value for the limit was obtained in \cite{Kasra-Juan} for $\gamma_0$ filling. A complete proof of the theorem above appears in \cite{book}.

\begin{rmk*}
Since the quantities $\FB_g$, $\FC(X)$, and $\FC_g(\gamma_0)$ in Theorem \ref{thm counting} can be normalized in different ways, we stress that we will follow the conventions in \cite{book}. 
\end{rmk*}

\begin{defi} \label{def:frequency}
    Let $\gamma_0$ be a multicurve of a closed hyperbolic surface $X$ of genus $g\ge 2$. The \emph{frequency of $\gamma_0$} is the quantity $\FC_g(\gamma_0)$ given by Theorem \ref{thm counting}.
\end{defi}

The goal of this section is to explain how to give an expression for the constant $\FC_g(\gamma_0)$. In fact, loosely speaking, the goal of this paper is to understand how $\FC_g(\gamma_0)$ depends on $g$.  Recall that the geometric intersection number $\iota(\cdot,\cdot)$ between multicurves extends continuously to $\CM\CL(X)$, and this extension is mapping class group invariant. We will say that a measured lamination $\lambda\in\CM\CL(X)$ {\em fills with $\gamma_0$} if the function
$$\CM\CL(X)\setminus\{0\}\to\BR_{\ge 0},\ \eta\mapsto\iota(\eta,\lambda)+\iota(\eta,\gamma_0)$$
is positive. Equivalently, but using a language we will not need here, $\lambda$ fills with $\gamma_0$ if $\lambda+\gamma_0$ is a filling current (see \cite[Chapter 3]{book} for definitions and basic facts about currents). 

The set $\{\lambda\in\CM\CL(X)|\lambda\text{ fills with }\gamma_0\}$ of those measured laminations which fill with $\gamma_0$ has full Thurston measure and is invariant under the stabilizer $\Stab_{\Map(X)}(\gamma_0)$ of $\gamma_0$ in the mapping class group. The action of $\Stab_{\Map(X)}(\gamma_0)$
is proper and the quotient has finite Thurston measure.

The reason why all of this is relevant for us is that we have 
\begin{equation}\label{eq formula for c(gamma) in terms of thurston measure 0}
\FC_g(\gamma_0)=\FM_{\Thu}\left(\Lfaktor{\{\lambda\in\CM\CL(X)\text{ fills with }\gamma_0\text{ and }\iota(\lambda,\gamma_0)\le 1\}}{\Stab_{\Map(X)}(\gamma_0)}\right).
\end{equation}
This formula is due to Mirzakhani \cite{Maryam simple} when $\gamma_0$ is simple, to the second and third authors \cite{Kasra-Juan} when it is filling, and to Erlandsson and the third author \cite{book} in general.

\begin{rmk*}
    As the reader might remember, we are assuming that $X$ has genus at least $3$. In genus $2$ the presence of the hyper-elliptic involution forces, for certain multicurves $\gamma_0$, the presence of a additional factor in \eqref{eq formula for c(gamma) in terms of thurston measure 0} given by:
    \begin{equation} \label{eq:CstGenus2}
        \frac{\left|\Ker(\Map(X)\actson\CM\CL(X)) \right|}{\left|\Ker(\Map(X)\actson\CM\CL(X)) \bigcap \Stab_{\Map(X)}(\gamma_0)  \right|}.
    \end{equation}
    See \cite[pp.146-148]{book} for more details.
\end{rmk*}

Instead of working with the set of measured laminations which fill with $\gamma_0$ we will consider the slightly smaller set
$$\CM\CL_{\gamma_0}(X)=\{\lambda\in\CM\CL(X)\text{ fills with }\gamma_0\text{ and has no component contained in }\Sigma\text{ or }Z\}$$
of those laminations which fill with $\gamma_0$ and which have the property that all the leaves in their support meet $\Gamma$. Since the set $\CM\CL_{\gamma_0}(X)$ has full Thurston measure and is $\Stab_{\Map(X)}(\gamma_0)$-invariant, we get directly from \eqref{eq formula for c(gamma) in terms of thurston measure 0} the following:

\begin{sat}[\cite{Maryam simple,Kasra-Juan,book}]\label{thm first formula for c}
Let $X$ be a closed surface of genus greater that $3$ and $\gamma_0$ a multicurve of $X$. The quantity $\FC_g(\gamma_0)$ in Theorem \ref{thm counting} is given by 
$$\FC_g(\gamma_0)=\FM_{\Thu}\left(\Lfaktor{\{\lambda\in\CM\CL_{\gamma_0}(X),\iota(\lambda,\gamma_0)\le 1\}}{\Stab_{\Map(X)}(\gamma_0)}\right)$$
where $\Stab_{\Map(X)}(\gamma_0)$ is the stabilizer of $\gamma_0$ in the mapping class group, and where $\iota(\cdot,\cdot)$ is the geometric intersection form. The quantity $\FC_g(\gamma_0)$ is finite.  \hfill $\blacksquare$
\end{sat}

To evaluate the constant $\FC_g(\gamma_0)$ using the expression in Theorem \ref{thm first formula for c}, one needs to be able to give a fundamental domain for the action of $\Stab_{\Map(X)}(\gamma_0)$ on $\CM\CL_{\gamma_0}(X)$. 

\subsection{Replacing \texorpdfstring{$\mathrm{Stab}_{\mathrm{Map}(X)}(\gamma_0)$}{Stab_{Map(X)}(phi(gamma_0))} by one of its subgroups} \label{sec : subgroup}

Recall that our current aim is to give a fundamental domain for the action of $\Stab_{\Map(X)}(\gamma_0)$ on $\CM\CL_{\gamma_0}(X)$. This becomes easier if one replaces $\Stab_{\Map(X)}(\gamma_0)$ by one of its subgroups $G$. To describe $G$, note that since both $\Sigma_{\hyp}$ and $Z$ are $\Stab_{\Map(X)}(\gamma_0)$-invariant, we get a homomorphism
\begin{equation} \label{eq homomorphism}
    \pi_*:\Stab_{\Map(X)}(\gamma_0)\to\Map(\Sigma_\hyp)\times\Map(Z)
\end{equation}
whose kernel consists of the group $\BT$ of multi-twists along the components of $\Gamma$.

The image of the composition of $\pi_*$ with the projection of $\Map(\Sigma_\hyp)\times\Map(Z)$ to the first factor is finite because $\gamma_0\cap\Sigma_\hyp$ fills. On the other hand the image of the composition of $\pi_*$ with the projection onto the second factor is very much not finite because it contains the pure mapping class group $\PMap(Z)$ of $Z$. The subgroup we care about is 
\begin{equation}\label{eq convenient subgroup}
G=\pi_*^{-1}\big(\{\Id_{\Sigma_\hyp}\}\times \PMap(Z)\big)\subset\Stab_{\Map(X)}(\gamma_0).
\end{equation}
By construction, $G$ is a normal subgroup of $\Stab_{\Map(X)}(\gamma_0)$ of finite index
\begin{equation}\label{eq constant sym}
\sym(\gamma_0)=\left[\Stab_{\Map(X)}(\gamma_0):G\right].
\end{equation}
It follows that 
$$\Lfaktor{\{\lambda\in\CM\CL_{\gamma_0}(X),\iota(\lambda,\gamma_0)\le 1\}}{\Stab_{\Map(X)}(\gamma_0)}=\Lfaktor{\left(\Lfaktor{\{\lambda\in\CM\CL_{\gamma_0}(X),\iota(\lambda,\gamma_0)\le 1\}}{G}\right)}{\left(\Lfaktor{\Stab_{\Map(X)}(\gamma_0)}{G}\right)}.$$
It follows from the finiteness of $\FC_g(\gamma_0)$ that the quantity
\begin{equation}\label{eq original c}
\FC^G(\gamma_0):=\FM_{\Thu}\left(\Lfaktor{\{\lambda\in\CM\CL_{\gamma_0}(X),\iota(\lambda,\gamma_0)\le 1\}}{G}\right)
\end{equation}
is also finite and can be related to the constant $\FC_g(\gamma_0)$ by the following lemma.

\begin{lem}  \label{lem:measure of a fund dom}
    Let $H$ be a finite group acting on a measured space $(\CX,\FM)$ such that the measure is $H$-invariant and gives no weight to the boundary of any closed set. If $\CD$ is a fundamental domain for $H\actson\CX$ then
    \[ \FM\left( \CD \right) = \frac{\FM(\CX)}{\left| \Lfaktor{H}{\Ker(H\actson \CX)} \right|}.  \]
\end{lem}

We apply this lemma with lemma with $\CX = \Lfaktor{\{\lambda\in\CM\CL_{\gamma_0}(X),\iota(\lambda,\gamma_0)\le 1\}}{G} $ and $H= \Lfaktor{\Stab_{\Map(X)}(\gamma_0)}{G}$. In that case, assuming the genus is greater that $2$, the kernel of the action is trivial and we get
\begin{equation}\label{eq0}
\FC_g(\gamma_0)=\frac 1{\sym(\gamma_0)}\cdot\FC^G(\gamma_0).
\end{equation}
The reason why it is easier to calculate $\FC^G(\gamma_0)$ than $\FC_g(\gamma_0)$ is that it is much easier to give a fundamental domain for the action of $G$ on $\CM\CL_{\gamma_0}(X)$ than for the action of $\Stab_{\Map(X)}(\gamma_0)$. We discuss next the main tool we will use to give that fundamental domain.

\begin{rmk*}
    As mentioned above the kernel is trivial for $g\geq 3$. In genus $2$ the hyper-elliptic involution has again to be taken into account and the size of the kernel appears as an supplementary factor in \eqref{eq0}. The kernel is of cardinality
    \begin{equation} \label{eq:CstGenus2bis}
        \frac{\left|\Ker(\Map(X)\actson\CM\CL(X) \right|}{\left|\Ker(\Map(X)\actson\CM\CL(X)) \bigcap G  \right|}.
    \end{equation}
\end{rmk*}

\subsection{$\FC_g(\gamma_0)$ in terms of train-track charts}\label{sec Kasra}
Continuing with the same notation, recall that our aim is to build a fundamental domain for the action of the group $G\subset\Map(X)$ from \eqref{eq convenient subgroup} on the set $\CM\CL_{\gamma_0}(X)$ of measured laminations which fill with $\gamma_0$ and have no components contained in either $\Sigma$ or $Z$.

The key idea is to consider the map 
$$\pi:\CM\CL_{\gamma_0}(X)\to\BA(\Sigma_{\hyp})\times\BA(Z)=\BA(X\setminus\CN(\Gamma))$$
sending the lamination $\lambda\in\CM\CL_{\gamma_0}(X)$ to its restriction to the subsurface $\Sigma_{\hyp}\cup Z=X\setminus\CN(\Gamma)$. More precisely note that, for every $\lambda\in\CM\CL_{\gamma_0}(X)$, when we cut the support of $\lambda$ along $\Gamma$, we obtain a bunch of simple and disjoint arcs in $X\setminus \CN(\Gamma)\simeq\Sigma_{\hyp}\cup Z$. In general we get uncountably many arcs, but they all fall into finitely homotopy classes $\kappa_1,\dots,\kappa_k$ which altogether forms a filling arc system $\kappa$. Moreover, for each $i$ there is a transverse arc $I_i\subset X\setminus\CN(\Gamma)$ which at the same time 
\begin{enumerate}
    \item meets all arcs of type $\kappa_i$ exactly once and
    \item does not meet any arc homotopic to $\kappa_j$ for $j\neq i$. 
\end{enumerate}

\begin{figure}[!ht]
\centering
\begin{tikzpicture}
\path (0,0) node {\includegraphics[width=70mm]{ 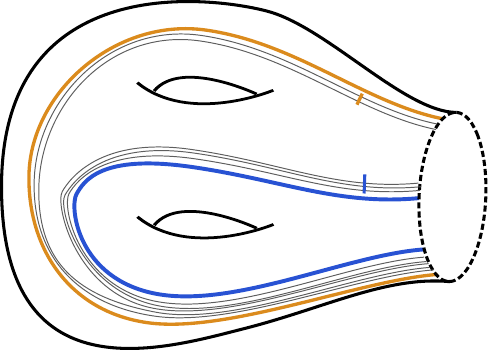}} ;
\draw(15mm,8mm) node {$I_1$}
(2mm,21mm) node {$\kappa_1$} 
(-35mm,18mm) node {$\Sigma_{\hyp}$} 
(20mm,1mm) node {$I_2$}
(2mm,-3mm) node {$\kappa_2$} 
;
\end{tikzpicture} 
\caption{Arcs and associated transverse arcs in $\Sigma_\hyp$ for a truncated measured lamination of $X$}
    \label{fig:From lam to arcs}
\end{figure}

The weighted arc system $\pi(\lambda)$ is the defined to be
$$\pi(\lambda)=\sum_i\iota(I_i,\lambda)\cdot\kappa_i$$
where $\iota(I_i,\lambda)$ is the total measure of the measure on $I_i$ induced by $\lambda$ is positive.
 
A key fact is that $\pi$ is equivariant under the homomorphism $\pi_*$ from \eqref{eq homomorphism},
meaning that
\[
\pi(\phi(\lambda))=\pi_*(\phi)(\pi(\lambda)),
\quad
\text{for every }
\Stab_{\Map(X)}(\gamma_0).
\]
See \cite[Chapter 11]{book} for details. Moreover, when we identify $\PMap(Z)\simeq \{\Id_{\Sigma_\hyp}\}\times \PMap(Z)$, $G$ fits in the exact sequence
\begin{equation}\label{eq exact sequence for subgroup}
0\to\BT\to G\stackrel{\pi_*}\to \PMap_Z\to 0.
\end{equation}

Now, as $\BA(X\setminus\CN(\Gamma))$ decomposes as $\bigcup \limits_{\alpha\in\CA(X\setminus\CN(\Gamma))} \BA(\alpha)$ we can write 
$$ \CM\CL_{\gamma_0}(X) =     \bigcup \limits_{\alpha\in\CA(X\setminus\CN(\Gamma))} \pi^{-1}( \BA(\alpha) ).  $$
Choosing a representative of each orbit of $\pi_*(G)\actson \CA(X\setminus\CN(\Gamma))$, we get an identification:
$$\Lfaktor{\CA(X\setminus\CN(\Gamma))}{\pi_*(G)} =\begin{array}{l}\text{set of representatives of maximal arc systems in }X\setminus\CN(\Gamma)\\
\text{under the action }\pi_*(G)\actson\CA(X\setminus\CN(\Gamma)).\end{array}$$
Observe that, unless $\Sigma_{hyp}$ is a finite union of pairs of pants, the group $\pi_*(G)$ acts with infinitely many orbits on $\CA(X\setminus\CN(\Gamma))$. Anyways, keeping in mind that some of these arc systems have non-trivial stabilizer we get that if $K_\alpha \equiv \Lfaktor{\BA(\alpha)}{\Stab_{\pi*(G)}(\alpha)}$ is a fundamental domain for $\Stab_{\pi*(G)}(\alpha)$ acting on $\BA(\alpha)$ then
 \[   \bigcup\limits_{\alpha \in \Lfaktor{\CA(X\setminus\CN(\Gamma))}{\pi_*(G)}}
 K_\alpha
 \]
is a fundamental domain for the action of $\pi_*(G)$ on $\BA(X\setminus\CN(\Gamma))$.

%%%%%%%%% Action of twst %%%%%%%%%%%

Let's now look at the action of $\BT$ on $$K = \pi^{-1}(\bigcup\limits_{\alpha \in \Lfaktor{\CA(X\setminus\CN(\Gamma))}{\pi_*(G)}}
 K_\alpha) = \bigcup\limits_{\alpha \in \Lfaktor{\CA(X\setminus\CN(\Gamma))}{\pi_*(G)}}
 \pi^{-1}(K_\alpha).$$

For every $\bar a \in \BA(\alpha)$ the action of $\BT$ on $\CM\CL(X)$ preserves $\pi^{-1}(\bar a)$ and by \cite[Exercise 11.3]{book}, the subset of $\CM\CL(\tau_\alpha)$ consisting of those measured laminations carried by the $\Gamma$-adapted train track $\tau_{\alpha}$ associated to the arc system $\alpha$ is a fundamental domain for the action of $\BT\actson\pi^{-1}(\BA(\alpha))$. Hence, the following is a fundamental domain for the action of $\BT$ on $K$:
\[   \bigcup\limits_{\alpha \in \Lfaktor{\CA(X\setminus\CN(\Gamma))}{\pi_*(G)}}
 \pi^{-1}(K_\alpha)\cap \CM\CL(\tau_\alpha).   \]
Now, given that $\BT$, $G$ and $\PMap(Z)$ fits in the exact sequence \eqref{eq exact sequence for subgroup} it is also domain for the action of $G$ on $\CM\CL_{\gamma_0}(X)$. It follows that a fundamental domain for the action of $G$ on $\{ \lambda \in \CM\CL(\gamma_0) | i(\lambda,\gamma_0)\le 1\}$ is given by  
\begin{equation} \label{eq: expression of a fundamental domain}
    \bigcup\limits_{\alpha \in \Lfaktor{\CA(X\setminus\CN(\Gamma))}{\pi_*(G)}}
 \pi^{-1}(K_\alpha)\cap \{ \lambda \in \CM\CL(\tau_\alpha) | i(\lambda,\gamma_0)\le 1\}.
\end{equation}
We are now ready to prove the main theorem of this section.

\begin{sat}\label{thm Kasra formula}
Let $X$ be a closed surface of genus $g\ge 3$ and $\gamma_0\subset X$ a multicurve, and let $\Sigma$ and $Z$ be as introduced in the opening of Section \ref{sec frequency}. We have
\[
    \FC_g(\gamma_0)
    =
    \frac{1}{\sym(\gamma_0)}
    \sum_{\sigma \in \CA(\Sigma_{\hyp})}
    \sum_{\zeta \in \Lfaktor{\CA(Z)}{P \Map(Z)}}
    \frac{\FM_{\Thu}(\{ \mu \in \mathcal{ML}(\tau_{\sigma \cup \zeta}) : \iota(\mu, \gamma_0) \leq 1 \})}{\left| \Stab_{\PMap(Z)}(\zeta) \right|}
\]
where $\tau_{\sigma\cup\zeta}$ is the $\Gamma$-adapted train track associated to the arc system $\sigma\cup\zeta$,  and $\sym(\gamma_0)$ is as in \eqref{eq constant sym}.
\end{sat}

\begin{proof}
    To begin with, observe that it follows directly from \eqref{eq original c}, \eqref{eq: expression of a fundamental domain} and Fact \ref{fact:Adapted TT are incompatible} that  
    \[   \FC^G(\gamma_0)
=\sum_{\alpha \in \Lfaktor{\CA(X\setminus\CN(\Gamma))}{\pi_*(G)} }
\FM_{\Thu}\left(\pi^{-1}(K_\alpha)\cap\{ \lambda \in  \CM\CL(\tau_\alpha) | i(\lambda,\gamma_0)\le 1 \}  \right). \]
Recall also that the decomposition $X\setminus\CN(\Gamma)\simeq \Sigma_{\hyp}\cup Z$ implies the following decomposition for $\CA(X\setminus\CN(\Gamma))$:
$$\CA(X\setminus\CN(\Gamma))=\CA(\Sigma_{\hyp})\times\CA(Z).$$
Given that $\pi_*(G)=\{\Id_{\Sigma_\hyp}\}\times \PMap(Z)$, we have that the action of $\pi_*(G)$ on $\CA(X\setminus\CN(\Gamma))$ is trivial on the factor $\CA(\Sigma_\hyp)$, and via $\PMap(Z)\actson \CA(Z)$ on the other factor. We can thus rewrite our previous expression for $\FC^G(\gamma_0)$ as follows:
\[   \FC^G(\gamma_0)
=\sum_{\sigma \in \CA(\Sigma_{\hyp})}
    \sum_{\zeta \in \Lfaktor{\CA(Z)}{P \Map(Z)}}
\FM_{\Thu}\left(\pi^{-1}(K_{\sigma\cup\zeta})\cap\{ \lambda \in  \CM\CL(\tau_{\sigma\cup\zeta}) | i(\lambda,\gamma_0)\le 1 \}  \right). \]
Recalling now that $K_{\sigma\cup\zeta} \equiv \Lfaktor{\BA({\sigma\cup\zeta})}{\Stab_{\pi_*(G)}({\sigma\cup\zeta})}$ is a fundamental domain for $\Stab_{\pi*(G)}(\alpha)$ acting on $\BA(\alpha)$, we have an identification 
\[  \pi^{-1}(K_{\sigma\cup\zeta})\cap\{ \lambda \in  \CM\CL(\tau_{\sigma\cup\zeta}) | i(\lambda,\gamma_0)\le 1 \}   \equiv \Lfaktor{\{ \lambda \in  \CM\CL(\tau_{\sigma\cup\zeta}) | i(\lambda,\gamma_0)\le 1 \}}{\left(\Lfaktor{\pi_*^{-1}(\Stab_{\pi_*(G)}({\sigma\cup\zeta}))}{\BT}\right)}.     \]
The action of $\pi_*^{-1}(\Stab_{\pi_*(G)}({\sigma\cup\zeta}))$ on $\CM\CL(X)$ has, for all $\sigma,\zeta$, trivial kernel. Moreover, 
$$\left|\Lfaktor{\pi_*^{-1}(\Stab_{\pi_*(G)}(\sigma\cup\zeta))}{\BT}\right|= |\Stab_{\pi_*(G)}(\sigma\cup\zeta)|= |\Stab_{\PMap(Z)(\zeta)}|$$ 
and Lemma \ref{lem:measure of a fund dom} ends the proof.
\end{proof}

\section{$\FC_g(\gamma_0)$ in terms of counting} \label{sec:FrequencyCounting}
Continuing with the same notation, suppose that $\gamma_0$ is a multicurve in $X$, and let $\Gamma$, $\Sigma$, $Z$ and $\CN(\Gamma)$ be as described at the beginning of Section \ref{sec frequency}. Consider
$$\flip_{\Gamma}:\D\CN(\Gamma)\to\D\CN(\Gamma)$$
switching the two boundary components of each component of $\CN(\Gamma)$ as described in the left part of Figure~\ref{fig:flip gamma}.

\begin{figure}[!ht]
\centering
\begin{tikzpicture}
\path (0,0) node {\includegraphics[width=\linewidth]{  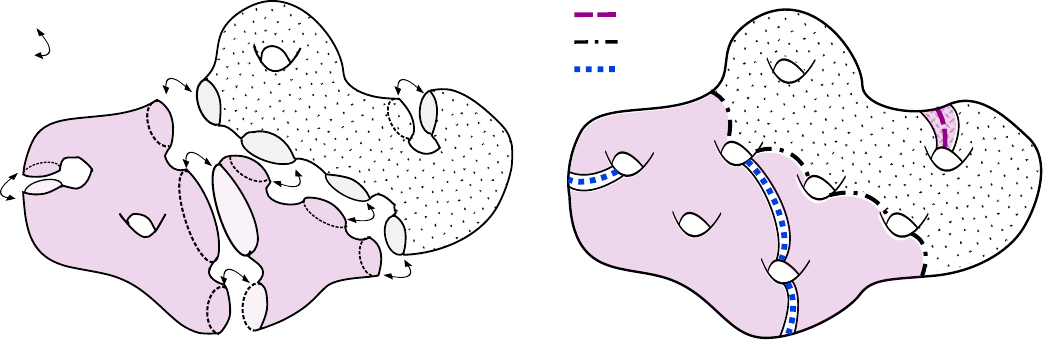}} ;
\draw
(-58mm,18mm) node {action of $\flip_\Gamma$} 
(19.5mm,18mm) node {$\Gamma_{\fix}$} 
(20mm,23mm) node {$\Gamma_{\ann}$}
(20mm,13mm) node {$\Gamma_{\flip}$}
;
\end{tikzpicture}
\caption{Description of the action of $\flip_\Gamma$ for the setting corresponding to Figure \ref{fig:full description} and the corresponding decomposition of $\Gamma$}
    \label{fig:flip gamma}
\end{figure}

The involution $\flip_\gamma$ induces the involution \eqref{eq involution induced for Sigma}, described in Figure \ref{fig:flipSigma} on $\D\Sigma$.
\begin{figure} [ht!]
\begin{minipage}{0.6\textwidth}
\begin{equation} \label{eq involution induced for Sigma}
\begin{array}{crcl}
     \flip_{\Sigma}:& \D\Sigma &\longrightarrow &\D\Sigma \\
     &  c &\longmapsto&
    \begin{cases}
        \flip_\Gamma(c) & \text{if } \flip_\Gamma(c) \in \D\Sigma, \\
        c & \text{otherwise.}
    \end{cases}
\end{array}
\end{equation} 
\end{minipage}
\hfill
\begin{minipage}{0.35\textwidth}
    \centering
    \includegraphics[width=\linewidth]{  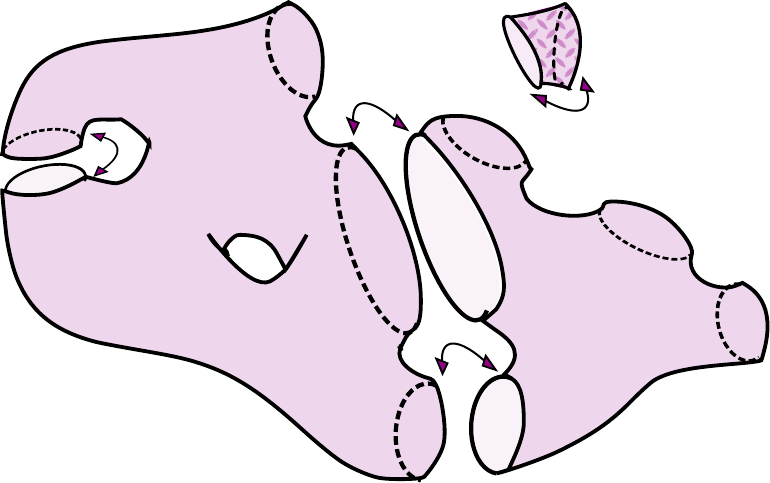}
\end{minipage}
    \caption{Description of $\flip_\Sigma$}
    \label{fig:flipSigma}
\end{figure}

Note that we can identify the multicurve $\Gamma$ with both the set of orbits of $\flip_\Gamma$ and the set of orbits of $\flip_\Sigma$. From this second view point the curve $\Gamma$ can be divided into 3 distinct multicurves (see the decomposition in Figure \ref{fig:flip gamma}):
\begin{enumerate}
    \item $\Gamma_\ann$ the curve homotopic to the two boundary components of an annuli component of~$\Sigma$, i.e. the orbits for the action of $\flip_\Sigma$ restricted to $\Sigma_{\ann}$,
    \item $\Gamma_{\flip}$ the curves homotopic to two distinct boundaries of $\Sigma_\hyp$, i.e. the $\flip_\Sigma$ orbits of size two in $\D\Sigma_{\hyp}$, and
    \item $\Gamma_{\fix}$ the curves homotopic to a unique boundary component of $\Sigma$, i.e. the $\flip_\Sigma$ orbits of size one in $\D\Sigma$.
\end{enumerate}

Recall that $\flip_\Sigma^\BR:\BR^{\D\Sigma}\to \BR^{\D\Sigma}$ denote the linear map induced by $\flip_\Sigma$ on $\BR^{|\D\Sigma|}$, and that $\bar\BA(\Sigma,\flip_\Sigma)$ is the subset of $\bar\BA(\Sigma)$ consisting of points whose image under the boundary map ${\bf b}:\bar\BA(\Sigma)\to\BR_{\ge 0}^{\D\Sigma}$ is $\flip_\Sigma^\BR$-invariant. Then, there is a unique map $${\bf w}_\Sigma:\bar\BA(\Sigma,\flip_\Sigma)\to \BR^\Gamma$$ making the following diagram commute:
\begin{equation} \label{eq:defWSigma}
    \xymatrix{\bar\BA(\Sigma,\flip_\Sigma)\ar[d]_{\bf b}\ar[dr]^{{\bf w}_\Sigma} & \\ \Fix(\flip^\BR_\Sigma) \ar[r]^{\sim} & \BR^\Gamma.}
\end{equation}
Noting that all components of $\D Z$  can be seen as a subset of $\D \Sigma$, we can define a canonical injection $\D Z\hookrightarrow\D\Sigma$. We denote by 
\begin{equation}\label{eq definition of boundary map restricted to Z}
{\bf b}\vert_{\D Z}:\bar\BA(\Sigma,\flip_\Sigma)\to\BR_{\ge 0}^{\D Z}    
\end{equation}
the composition of the boundary map ${\bf b}$ and the projection $\BR_{\ge 0}^{\D\Sigma}\to\BR_{\ge 0}^{\D Z}$ induced by the inclusion.

Before finally being able to state the theorem we are going to prove in this section, we need to introduce the linear form
\begin{equation}\label{eq intersection with gamma0}
\begin{array}{ccrcl}
     {\bf I}_{\gamma_0}& :& \bar\BA(\Sigma, \flip_\Sigma) &\longrightarrow &\BR \\
     & & \sum_i x_i\cdot\sigma_i &\longmapsto & \sum_i \iota(\sigma_i,\gamma_0) \cdot x_i.
\end{array}
\end{equation}
and a piece of notation: given a vector $\bar u=(u_i)\in\BR^n$ we set $\product_{+1}(u)=\prod_i(u_i+1)$.

The following result is the goal of this section:

\begin{sat}\label{thm frequency counting}
Let $\gamma_0$ be a multicurve in a closed surface $X$ of genus $g\ge 3$. We have 
$$\FC_g(\gamma_0)=
\frac 1{\sym(\gamma_0)}
\lim_{L\to\infty}
\frac 1{L^{6g-6}}
\sum_{\substack{\bar x\in\BA_\BZ(\Sigma,\flip_\Sigma)\\ {\bf I}_{\gamma_0}(\bar x)\le L}}
N_Z({\bf b}\vert_{\D Z}(\bar x))\cdot
\product_{+1}({\bf w}_\Sigma(\bar x)),$$
where $\Sigma\subset X$ is the smallest subsurface containing $\gamma_0$, where $Z$ is the union of the hyperbolic components of $X\setminus\Sigma$, where $\flip_\Sigma$ is the involution \eqref{eq involution induced for Sigma}, where ${\bf I}_{\gamma_0}$ is given by \eqref{eq intersection with gamma0}, and where we are identifying $ {\D Z}$ with a subset of $\D\Sigma$.
\end{sat}

Theorem \ref{thm frequency counting} follows from Theorem \ref{thm Kasra formula} and the idea that Thurston measures can be calculated by counting points and taking a limit as described in \eqref{eq def thurston measure}. In that light, it is highly unsurprising. We suggest the reader to, in a first reading, skip its proof, that is the rest of the section.

\subsection{Measured laminations carried by a $\Gamma$-adapted train track}
Fixing maximal arc systems $\sigma\in\CA(\Sigma_{\hyp})$ and $\zeta\in\CA(Z)$, let $\tau_{\sigma\cup\zeta}\subset X$ be the $\Gamma$-adapted train track associated to $(\sigma,\zeta)\in\CA(\Sigma_{hyp})\times\CA(Z)=\CA(\Sigma_{\hyp}\cup Z)$ as described in Subsection \ref{sec: adapted TT}.
As for every train track, we can identify the set $\CM\CL(\tau_{\sigma\cup\zeta})$ of measured laminations it carries, with the set $W(\tau_{\sigma\cup\zeta})\subset\BR_{\ge 0}^{E(\tau_{\sigma\cup\zeta})}$ of solutions of the switch equations. 
In $E(\tau_{\sigma\cup\zeta})$ there is one edge per arcs in $\sigma\cup\zeta$ and two edges for each $\gamma\in\Gamma$, one going strait and one twisting, see Figure \ref{fig:SwitchAdaptedTT} or Figure \ref{ch4-fig4}. In particular $R^{E(\tau_{\sigma\cup\zeta})}$ identifies with $\bar\BA({\sigma\cup\zeta})\times\BR_{\ge 0}^{\Gamma}\times\BR_{\ge 0}^{\Gamma}$ and writing $\bar\BA(\sigma\cup\zeta)=\bar\BA(\sigma)\times\bar\BA(\zeta)$, we write points in $\bar\BA({\sigma\cup\zeta})\times\BR_{\ge 0}^{\Gamma}\times\BR_{\ge 0}^{\Gamma}$ as $((\bar x,\bar z),\bar s,\bar t)$.  Here for $\gamma\in\Gamma$, we think of the $\gamma$-th component $s_\gamma$ of $\bar s$ (resp. $t_\gamma$ of $\bar t$) as giving the weight for the \underline{s}traight (resp. \underline{t}wisty) edge in the component of $\CN(\Gamma)$ corresponding to $\gamma$ as illustrated in Figure~\ref{fig:SwitchAdaptedTT}.
\begin{figure}[!ht]
\centering
\begin{tikzpicture}  
\path (0,0) node {\includegraphics[width=0.25\linewidth]{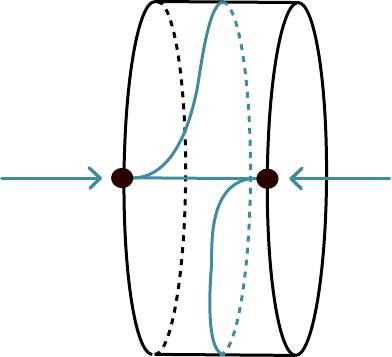}} ;
\draw
(-28mm,1.5mm) node {${\bf w}_\Gamma(\bar x,\bar z)_\gamma$}
(26mm,1.5mm) node {${\bf w}_\Gamma(\bar x,\bar z)_\gamma$}
(0mm,20mm) node {$\CN(\gamma)$}
(1mm,2mm) node {$s_\gamma$}
(4mm,-10mm) node {$t_\gamma$}
;
\end{tikzpicture}
\caption{Switch equations around $\gamma\in\Gamma$ for a $\Gamma-$adapted train-track}
    \label{fig:SwitchAdaptedTT}
\end{figure}
In these coordinates, the switch equations are given by $\mathbf{w}_\Gamma(\bar{x}, \bar{z}) = \bar{s} + \bar{t}$ and then
\begin{equation}\label{eq total form of W(tau)}
W(\tau_{\sigma \cup \zeta})
=
\{ ((\bar{x}, \bar{z}), \bar{s}, \bar{t}) \in \mathbb{A}(\alpha \cup \zeta, \flip_\Gamma) \times \mathbb{R}_{\geq 0}^\Gamma \times \mathbb{R}_{\geq 0}^\Gamma : \mathbf{w}_\Gamma(\bar{x}, \bar{z}) = \bar{s} + \bar{t} \}
\end{equation}
where we are thinking of $\flip_\Gamma$ as acting on $\D\Sigma_{\hyp}\sqcup\D Z=\D\CN(\Gamma)$, and where $$\mathbf{w}_\Gamma:\bar\BA(\sigma\cup\alpha,\flip_\Gamma)\to\BR^\Gamma$$  
is the map making commute the diagram 
$$\xymatrix{\bar\BA(\sigma\cup\alpha,\flip_\Gamma)\ar[d]_{\bf b}\ar[dr]^{{\bf w}_\Gamma} & \\ \Fix(\flip^\BR_\Gamma) \ar[r]^{\sim} & \BR^\Gamma,}$$
where $\flip^\BR_\Gamma:\BR^{\D(\Sigma\cup Z)}\to \BR^{\D(\Sigma\cup Z)}$ is the linear involution induced by $\flip_\Gamma$. We record this fact.

\begin{lem}\label{lem solutions switch equation adapted train track}
For $(\sigma,\zeta)\in\CA(\Sigma_{\hyp})\times\CA(Z)$ there is a bijection
\begin{equation}\label{eq sick of this}
\CM\CL(\tau_{{\sigma\cup\zeta}})\simeq
\left\{((\bar x,\bar z),\bar s,\bar t)
\in\bar\BA({\sigma\cup\zeta},\flip_{\Gamma})\times\BR_{\ge 0}^{\Gamma} \times\BR_{\ge 0}^{\Gamma} \ 
\middle\vert \ 
{\bf w}_\Gamma(\bar x,\bar z)=\bar s+\bar t\right\}.
\end{equation} 
This bijection restricts to a bijection between $\CM\CL_\BZ(\tau_{\sigma\cup\alpha})$ and the set of integral points on the right.\qed
\end{lem}

\subsection{Intersection form}
Fix some $(\sigma,\zeta)\in\CA(\Sigma_{\hyp})\times\CA(Z)$, a key observation is that under the bijection \eqref{eq sick of this},
the function
\begin{center}
$\begin{array}{rcl}
     \CM\CL(\tau_{\sigma\cup\zeta})& \longrightarrow &\BR_{\ge 0}  \\
     \lambda& \longmapsto& \iota(\lambda,\gamma_0)
\end{array}$
\end{center}
is the restriction of the composition of the projection
\[
\CM\CL(\tau_{\sigma\cup\zeta})
\simeq
W(\tau_{\sigma\cup\zeta}) \longrightarrow \bar\BA({\sigma\cup\zeta},\flip_{\Gamma})
\]
with the a linear form ${\bf J}_{\gamma_0}$ (that we describe bellow). Before describing this form note that the multicurve $\Gamma_{\ann}$ introduced in the introduction to this section agrees with the set of $\flip_\Gamma$-orbits contained in the invariant subset $\D\Sigma_{\ann}$.

Denoting by $y_\gamma$ the entry of $\bar y\in\BR^\Gamma$ corresponding to $\gamma\in\Gamma$, and writing $\bar x\in\BA(\sigma)$ as
$$\bar x=\sum_{\sigma_i\in\sigma}x_i\cdot\sigma_i$$
where the $\sigma_i$'s are the components of $\sigma$, consider the linear form
\begin{equation}\label{eq intersection with gamma0 linear form}
\begin{array}{cccl}
     {\bf J}_{\gamma_0} :&\bar\BA(\sigma\cup\zeta,\flip_\Gamma)& \longrightarrow &\BR_{\ge 0}  \\
     & (\bar x,\bar a) &\longmapsto &\sum_{\sigma_i\in\sigma}\iota(\sigma_i,\gamma_0)\cdot x_i+\sum_{\gamma\in\Gamma_{\ann}}{\bf w}_\Gamma(\bar x,\bar a)_\gamma.
\end{array}
\end{equation}
With this notation we have:

\begin{lem}\label{lem intersection given by linear form}
For every $((\bar x,\bar z),\bar s,\bar t)\in\bar\BA(\sigma\cup\zeta,\flip_\Gamma)\times\BR_{\ge 0}^{\Gamma}\times\BR_{\ge 0}^{\Gamma}$ with $\bar s+\bar t={\bf w}_\Gamma(\bar x,\bar z)$, we have
\[
\iota(\lambda_{(\bar x,\bar z),\bar s,\bar t},\gamma_0)={\bf J}_{\gamma_0}(\bar x,\bar z)
\]
Here $\lambda_{(\bar x,\bar z),\bar s,\bar t}$ is the measured lamination corresponding to $((\bar x,\bar z),\bar s,\bar t)$ under the bijection~\eqref{eq sick of this}.
\end{lem}
\begin{proof}
To begin with note that up to isotopy we might assume that $\gamma_0\cap\CN(\Gamma)=\emptyset$ and that the representatives of the arcs in $\sigma\cup\zeta$ used to build $\tau_{\sigma\cup\zeta}$ are transversal to $\gamma_0$, with minimal possible number of intersections. It follows that $\tau_{\sigma\cup\zeta}$ is transversal to $\gamma_0$ and that no component of $X\setminus(\tau_{\sigma\cup\zeta}\cup\gamma_0)$ is a bigon.

Note now that, since both sides of the equation we want to prove are continuous and homogeneous, it suffices to prove the claim when both $\bar x,\bar z$ and $\bar s$ are integers, meaning that $\lambda_{(\bar x,\bar z),\bar s,\bar t}$ is an integral multicurve. In fact, since both sides are additive, we can assume that $\lambda_{(\bar x,\bar z),\bar s,\bar t}$ is a single curve of weight $1$. Now, since $\lambda_{(\bar x,\bar z),\bar s,\bar t}\in\CM\CL(\tau_{\sigma\cup\alpha})$, we can isotope it to an immersion, which we still denote by $\lambda_{(\bar x,\bar z),\bar s,\bar t}$, of $\BS^1$ into $\tau_{\sigma\cup\zeta}$. Direct inspection shows that this immersion meets $\gamma_0$ transversely in exactly ${\bf J}_{\gamma_0}(\bar x,\bar z)$ points. Since $\tau_{\sigma\cup\zeta}$ and $\gamma_0$ form no bigons, we get that neither do $\lambda_{(\bar x,\bar z),\bar s,\bar t}$ and $\gamma_0$. It follows that all the intersections are essential, and the claim follows.
\end{proof}

\subsection{Thurston measure on $\tau_{\sigma\cup\zeta}$}
Lemma \ref{lem intersection given by linear form} is the first ingredient needed to be able to evaluate the right side of Theorem \ref{thm Kasra formula}. What is left is to understand how to calculate the Thurston measure of the set 
\begin{equation}\label{eq I am tired of this shit}
U_{\sigma\cup\zeta}=\{ \mu \in \mathcal{ML}(\tau_{\sigma \cup \zeta}) : \iota(\mu, \gamma_0) \leq 1 \}.
\end{equation}
Note that we get for example from Lemma \ref{lem intersection given by linear form} that this set is piece-wise linear, in particular its boundary has vanishing Thurston measure.
We deduce thus from \eqref{eq def thurston measure} that
\begin{equation} \label{measureToCounting}
    \FM_{\Thu}(U_{\sigma\cup\zeta})=\lim_{L\to\infty}\frac{\vert\{\lambda\in \mathcal{ML}_\BZ(\tau_{\sigma \cup \zeta})\text{ with }\iota(\lambda,\gamma_0)\le L\}\vert}{L^{6g-6}}.
\end{equation}
Since the bijection provided by Lemma \ref{lem solutions switch equation adapted train track} induces a bijection between the corresponding sets of integer points, we get that
$$\FM_{\Thu}(U_{\sigma\cup\zeta})
=\lim_{L\to\infty}\frac{1}{L^{6g-6}}
\left\vert\left\{\begin{array}{l}
    ((\bar x,\bar z),\bar s,\bar t)\in\bar\BA_\BZ({\sigma\cup\zeta},\flip_{\Gamma})\times\BZ_{\ge 0}^{\Gamma} \times\BZ_{\ge 0}^{\Gamma}\\
    \text{with }{\bf w}_\Gamma(\bar x,\bar z)=\bar s+\bar t\text{ and }\iota(\lambda_{(\bar x,\bar z),\bar s,\bar t},\gamma_0)\le L
\end{array}\right\}\right\vert.$$
Invoking Lemma \ref{lem intersection given by linear form}, and noting that $\bar\BA(\sigma\cup\zeta,\flip_\Gamma)\setminus\BA(\sigma\cup\zeta,\flip_\Gamma)$ has lower dimension, we can rewrite this as
$$\FM_{\Thu}(U_{\sigma\cup\zeta})=
\lim_{L\to\infty}\frac{1}{L^{6g-6}}
\left\vert\left\{\begin{array}{l}
((\bar x,\bar z),\bar s,\bar t)\in\BA_\BZ({\sigma\cup\zeta},\flip_{\Gamma})\times\BZ_{\ge 0}^{\Gamma} \times\BZ_{\ge 0}^{\Gamma}\\ 
\text{with }{\bf w}_\Gamma(\bar x,\bar z)=\bar s+\bar t\text{ and }{\bf J}_{\gamma_0}(\bar x,\bar z)\le L\end{array}
\right\}\right\vert.$$
Observing that for every $(\bar x,\bar z)\in\BA_\BZ({\sigma\cup\zeta})$ there are precisely 
$$\product_{+1}({\bf w}_\Gamma(\bar x,\bar z))=
    \prod_{\gamma\in\Gamma}({\bf w}_\Gamma(\bar x,\bar z)_\gamma+1)$$
pairs $(\bar s,\bar t)\in \BZ_{\ge 0}^{\Gamma} \times\BZ_{\ge 0}^{\Gamma}$ with ${\bf w}_\Gamma(\bar x,\bar z)=\bar s+\bar t$, we can rewrite the last expression for $\FM_{\Thu}(U)$ as follows:
\begin{multline}\label{eq almost there}
\FM_{\Thu}(\{ \mu \in \mathcal{ML}(\tau_{\sigma \cup \zeta}) : \iota(\mu, \gamma_0) \leq 1 \})=\\ =\lim_{L\to\infty}
\frac 1{L^{6g-6}}
\sum_{\substack{(\bar x,\bar z)\in\BA_\BZ(\sigma\cup\alpha,\flip_{\Gamma})\\ {\bf J}_{\gamma_0}(\bar x,\bar z)\le L}}
\product_{+1}({\bf w}_\Gamma(\bar x,\bar z)).
\end{multline}
The main issue with \eqref{eq almost there} is that we are adding over a pretty complicated set.

\subsection{Decomposing $(\BA(\sigma\cup\alpha),\flip_\Gamma)$} 
The involution $\flip_\Gamma$ does not preserve the decomposition $\D(\Sigma_{hyp}\sqcup Z)=\D\Sigma_{\hyp}\sqcup \D Z$. It still induces the following two involutions on $\D\Sigma_\hyp$ and $\D Z$ as described in Figure \ref{fig:flip hyp and Z}:
\begin{align*}
\flip_{\Sigma_{\hyp}} : \D\Sigma_\hyp &\longrightarrow \D\Sigma_\hyp\\
c &\longmapsto  
\begin{cases}
   \flip_\Gamma(c), & \text{if }\flip_\Gamma(c)\in\D\Sigma_\hyp,\\[2pt]
   c, & \text{otherwise},
\end{cases}
\end{align*}

\begin{align*}
\flip_Z : \D Z &\longrightarrow \D Z\\
c &\longmapsto 
\begin{cases}
   \flip_\Gamma(c), & \text{if }\flip_\Gamma(c)\in\D Z,\\[2pt]
   c, & \text{otherwise}.
\end{cases}
\end{align*}

\begin{figure}[!ht]
\centering
\begin{tikzpicture}
\path (0,0) node {\includegraphics[width=100mm]{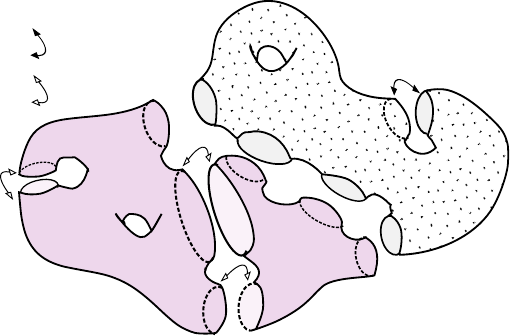}} ;
\draw
(-26mm,16mm) node {action of $\flip_{\Sigma_\hyp}$} 
(-28mm,25mm) node {action of $\flip_{Z}$}
;
\end{tikzpicture} 
\caption{Description of the action of $\flip_{\Sigma_\hyp}$ and $\flip_Z$ for the setting corresponding to Figure \ref{fig:full description}}
    \label{fig:flip hyp and Z}
\end{figure}

We will denote by $\D_{\fix}\Sigma_{\hyp}\subset\D\Sigma_{\hyp}$ and $\D_{\fix} Z \subset\D Z$ union of connected components fixed by $\flip_{\Sigma_{\hyp}}$ and $\flip_Z$ respectively: $\D_{\fix}\Sigma_{\hyp}=\Fix(\flip_{\Sigma_{\hyp}})$ and $\D_{\fix}Z=\Fix(\flip_Z)$. We also decompose
$$\D Z\setminus\D_{\fix}Z=\D_LZ\sqcup\D_RZ$$
where $\flip_Z(\D_LZ)=\D_RZ$. We can summarize this as follows:
$$\D\CN(\Gamma)=\overbrace{\supp(\flip_{\Sigma_{\hyp}})\sqcup\D_{\fix}\Sigma_{\hyp}}^{\D\Sigma_{\hyp}}\sqcup\overbrace{\D_{\fix}Z\sqcup\underbrace{\D_LZ\sqcup\D_RZ}_{\supp(\flip_Z)}}^{\D Z}.$$
Both multicurves $\D_{\fix}\Sigma_{hyp}$ and $\D_{\fix}Z$ are isotopic in $X$ to the same submulticurve $\Gamma_{\fix}$ of $\Gamma$. Similarly, the two multicurves $\D_LZ$ and $\D_RZ$ are isotopic in $X$ to one and the same submulticurve $\Gamma_{\ann}$ of $\Gamma$. We thus get an identification
\begin{equation}\label{eq juan needs badly a beer}
\begin{split}
\BR^{\D(X\setminus\CN(\Gamma))}
&=\overbrace{\BR^{\supp(\flip_{\Sigma_{\hyp}})}}^{\text{(1)}}\times\overbrace{\BR^{\D_{\fix}\Sigma_{\hyp}}\times\BR^{\D_{\fix}Z}}^{\text{(2)}}\times\overbrace{\BR^{\supp(\flip_{Z})}}^{\text{(3)}}\\
&=\overbrace{\BR^{\supp(\flip_{\Sigma_{\hyp}})}}^{\text{(1)}}\times\overbrace{\BR^{\Gamma_{\fix}}\times\BR^{\Gamma_{\fix}}}^{\text{(2)}}\times\overbrace{\BR^{\Gamma_{\ann}}\times\BR^{\Gamma_{\ann}}}^{\text{(3)}}.
\end{split}
\end{equation}

The linear map $\flip_\Gamma^{\BR}$ induced by the involution $\flip_\Gamma$ preserves each one of three overbraced factors in \eqref{eq juan needs badly a beer}, acting like $\flip_{\Sigma_{\hyp}}^{\BR}$ on (1), and exchanging both factors within (2) and (3) (see Figure \ref{fig:flip gamma}).

Recall now that $\BA(\Sigma_{\hyp},\flip_{\Sigma_\hyp})$ and  $\BA(Z,\flip_Z)$ are the preimages under the corresponding boundary map
\begin{align*}
{\bf b}_{\Sigma_{\hyp}}&:\BA(\Sigma_{\hyp})\to\BR_{\ge 0}^{\D\Sigma_{\hyp}}\simeq\BR^{\supp(\flip)}\times\BR^{\Gamma_{\fix}}\text{ and }\\
{\bf b}_Z&:\BA(Z)\to\BR_{\ge 0}^{\D Z}\simeq\BR^{\Gamma_{\fix}}\times\BR^{\supp(\flip_Z)}
\end{align*}
of the set of fixed points of the involutions of the target induced by $\flip_{\Sigma_{\hyp}}$ and $\flip_Z$ respectively. Denote by
$${\bf b}_{\Sigma_{\hyp}}\vert_{\fix}:\BA(\Sigma_{\hyp})\to\BR^{\D_{\fix}\Sigma_{\hyp}}=\BR^{\Gamma_{\fix}}\text{ and }{\bf b}_Z\vert_{\fix}:\BA(Z)\to\BR^{\D_{\fix}Z}=\BR^{\Gamma_{\fix}}$$
obtained by composing the respective boundary map with the projection onto the $\BR^{\Gamma_{\fix}}$-factors. With this notation we have that the image of
$$(\bar x,\bar z)\in\BA(\sigma)\times\BA(\zeta)=\BA(\sigma\cup\zeta)$$
under the boundary map ${\bf b}_{X\setminus\CN(\Gamma)}:\BA(\Sigma_{\hyp}\sqcup Z)\to\BR^{\D(\Sigma_{\hyp}\sqcup Z)}=\BR^{\D\CN(\Gamma)}$ is invariant under $\flip_\Gamma$ if and only if $\bar x\in\BA(\Sigma_{\hyp},\flip_{\Sigma_{\hyp}})$, $\bar z\in\BA(Z,\flip_Z)$, and ${\bf b}_{\Sigma_{\hyp}}\vert_{\fix}(\bar x)={\bf b}_{Z}\vert_{\fix}(\bar z)$.

Everything we have done maps integer points to integer points. In particular we get:

\begin{lem}\label{lem bijection arc complexes new}
We have a bijection
$$\BA_{\BZ}(\sigma\cup\alpha,\flip_\Gamma)\simeq\{(\bar x,\bar z)\in\BA_{\BZ}(\sigma,\flip_{\Sigma_{\hyp}})\times\BA_{\BZ}(\zeta,\flip_Z)\text{ with }{\bf b}_{\Sigma_{hyp}}\vert_{\fix}(\bar x)={\bf b}_Z\vert_{\fix}(\bar a)\}.$$
\qed
\end{lem}

\subsection{Separation of variables}
Lemma \ref{lem bijection arc complexes new} allows us to express a sum over $(\bar x,\bar z)\in\BA_{\BZ}(\sigma\cup\zeta,\flip_\Gamma)$ as two nested sums: a first one over $\bar x$, and a second one over those $\bar z$ satisfying ${\bf b}_{\Sigma_{\hyp}}\vert_{\fix}(\bar x)={\bf b}_Z\vert_{\fix}(\bar z)$. When we apply this idea and let $U_{\sigma\cup\zeta}$ be as in \eqref{eq I am tired of this shit}, we can rewrite \eqref{eq almost there} as follows:
\begin{equation}\label{eq just got a printer}
    \FM_{\Thu}(U_{\sigma \cup \zeta})
    =
    \lim_{L \to \infty} \frac{1}{L^{6g-6}}
    \sum_{(\bar{x}, \bar{z})}
    \product_{+1}(\mathbf{w}_\Gamma(\bar{x}, \bar{z})),
\end{equation}
where $((\bar{x}, \bar{z}))$ runs over the set
\[
    \{ (\bar{x}, \bar{z}) \in \BA_\BZ(\sigma,\flip_{\Sigma_{\hyp}}) \times \BA_\BZ(\zeta,\flip_Z) : {\bf b}_{\Sigma_{\hyp}}\vert_{\fix}(\bar x) = {\bf b}_Z\vert_{\fix}(\bar z),\ {\bf J}_{\gamma_0}(\bar{x}, \bar{z}) \leq L\}.
\]
To rewrite this in a more useful way recall that $\D Z$ is the disjoint union of $\D_{\fix}Z\simeq\Gamma_{\fix}$, $\D_LZ\simeq\Gamma_{\ann}$ and $\D_RZ\simeq\Gamma_{\ann}$, and that $\flip_Z$ fixes the first set pointwise and exchanges the two others. It follows that the image under the boundary map ${\bf b}_Z$ of any $\bar z\in\BA(Z,\flip_Z)$ is for the form $(\bar u,\bar v,\bar v)\in \BR^{\Gamma_{\fix}}\times\BR^{\Gamma_{\ann}}\times\BR^{\Gamma_{\ann}}=\BR^{\D Z}$. There are thus unique maps $${\bf w}_{\fix}:\BA(Z,\flip_Z)\to\BR^{\Gamma_{\fix}}\quad \text{and} \quad {\bf w}_{\ann}:\BA(Z,\flip_Z)\to\BR^{\Gamma_{\ann}}$$ making the following diagram commute
$$\xymatrix{
& \BA(Z,\flip_Z)\ar[d]_{{\bf b}_Z}\ar[dr]^{{\bf w}_{{\ann}}} \ar[dl]_{{\bf w}_{{\fix}}}& \\ 
\BR^{\Gamma_{\fix}} & \BR^{\Gamma_{\fix}}\times\BR^{\Gamma_{\ann}}\times\BR^{\Gamma_{\ann}} \ar[r]\ar[l]& \BR^{\Gamma_{\ann}}}$$
where the horizontal arrows are the projections onto the first and last factors respectively. 

Recall that $\|{\bf w}_{{\ann}}(\bar z)\|$ is the $\ell^1$ norm ${\bf w}_{{\ann}}(\bar z)$ so its the sum of all its entries in the vector, and for $\sigma\in\CA(\Sigma_{\hyp})$ we write $\bar x\in\BA(\sigma)$ as $\bar x=\sum_{\sigma_i\in\sigma}x_i\cdot\sigma_i$, then we have
$${\bf J}_{\gamma_0}(\bar x,\bar z)=\sum_{\sigma_i\in\sigma}\iota(\sigma_i,\gamma_0)\cdot x_i+\|{\bf w}_{{\ann}}(\bar z)\|.$$
In other words, if we denote
\begin{equation}\label{eq intersection with gamma0 linear form1}
\begin{array}{cccc}
    {\bf J}_{\gamma_0}': & \BA(\sigma,\flip_{\Sigma_{\hyp}})\times\BR^{\Gamma_{\ann}} & \to &\BR \\
     & \left(\sum_{\sigma_i\in\sigma}x_i\cdot\sigma_i,\bar v\right) & \mapsto & \sum_{\sigma_i\in\sigma}\iota(\sigma_i,\gamma_0)\cdot x_i+\|\bar v\|
\end{array}
\end{equation}
we have that the commutative diagram
$$\xymatrix{\BA(\Sigma_{\hyp}\cup Z,\flip_\Gamma)\ar[d]_{\mathrm{proj}\times{\bf w}_{{\ann}}}\ar[drr]^{{\bf J}_{\gamma_0}} & & \\ \BA(\Sigma_{\hyp},\flip_{\Sigma_{\hyp}})\times\BR^{\Gamma_{\ann}} \ar[rr]_{\hspace{1.2cm} {\bf J}'_{\gamma_0}}& & \BR}$$
where $\mathrm{proj}$ is the map induced by the projection of $\BA(\Sigma_{\hyp}\cup Z)=\BA(\Sigma_{\hyp})\times\BA(Z)$ to the first factor. With this notation we can rewrite \eqref{eq just got a printer} as
\begin{equation}\label{eq just got a printer2}
\FM_{\Thu}(U_{\sigma\cup\zeta})
=\lim_{L\to\infty}\frac 1{L^{6g-6}}\sum
\left\vert\left\{\substack{\bar z\in\BA_\BZ(\zeta,\flip_Z)\\ \text{with }{\bf b}_Z(\bar z)=({\bf b}_{\Sigma_\hyp|\fix}(\bar x),\bar v,\bar v)}\right\}\right\vert\cdot\product_{+1}({\bf w}_{\Sigma_{\hyp}}(\bar x),\bar v),
\end{equation}
where the sum is over all pairs $(\bar x,\bar v)\in\BA_\BZ(\sigma,\flip_{\Sigma_{\hyp}})\times \BZ_{\ge 0}^{\Gamma_{\ann}}$ with ${\bf J}'_{\gamma_0}(\bar x,\bar v)\le L$.

To rewrite this in a more compact way, recall that to the maximal arc system $\sigma\in\CA(\Sigma_{\hyp})$ we can associate a unique maximal arc system $\tilde\sigma\in\CA(\Sigma)$: the arc system $\tilde\sigma$ is the union of $\sigma$ and of a collection of arcs, one for each connected component of $\Sigma_{\ann}$ (see  \eqref{eq:completAS}). Identifying $\pi_0(\Sigma_{\ann})$ with $\Gamma_{\ann}$ we get the bijection
\[    \begin{array}{ccc}
     \BA(\sigma,\flip_{\Sigma_{\hyp}})\times\BR_{\ge 0}^{\Gamma_{\ann}} &  \to  &\BA(\tilde\sigma,\flip_\Sigma) \\
     \bar x, \bar v& \mapsto & (\bar x,\bar v).
\end{array}
\]

Under this identification, the linear map ${\bf J}'_{\gamma_0}$ from \eqref{eq intersection with gamma0 linear form1} becomes the form ${\bf I}_{\gamma_0}$ from \eqref{eq intersection with gamma0}. Moreover, the map 
\[ {\bf b}= {\bf b}_{\Sigma} : \BA(\tilde\sigma,\flip_\Sigma)  \to   \BR_{\ge}^{\D\Sigma}\]
corresponds to the map 
\[ \begin{array}{ccc}
    \BA(\sigma,\flip_{\Sigma_{\hyp}})\times\BR_{\ge 0}^{\Gamma_{\ann}} & \to &  \BR^{\Rfaktor{\D\Sigma_{\hyp}}{\flip_{\Sigma_{\hyp}}}} \times \BR^{\Gamma_{\ann}} \times \BR^{\Gamma_{\ann}} \\
    (\bar x, \bar v) & \mapsto & ({\bf b}_{\Sigma_\hyp}(\bar x),\bar v,\bar v), 
\end{array}    \]
then  $({\bf b}_{\Sigma_\hyp|\fix}(\bar x),\bar v,\bar v)$ becomes ${\bf b}\vert_{\D z}(\bar x,\bar v)$ as defined in in \eqref{eq definition of boundary map restricted to Z}, and $({\bf w}_{\Sigma_{\hyp}}(\bar x),\bar v)$ becomes ${\bf w}_\Sigma(\bar x,\bar v)$. We thus can rewrite \eqref{eq just got a printer2} as follows:

\begin{lem}\label{lem calculating thurston measure intersection adapted train track}
For $(\sigma,\zeta)\in\CA(\Sigma_{\hyp})\times\CA(Z)$ we have
\begin{multline*}
\FM_{\Thu}(\{ \mu \in \mathcal{ML}(\tau_{\sigma \cup \zeta}) : \iota(\mu, \gamma_0) \leq 1 \})=\\
=\lim_{L\to\infty}\frac 1{L^{6g-6}}\sum_{\substack{\bar x\in\BA_\BZ(\tilde\sigma,\flip_{\Sigma})
\\ {\bf I}_{\gamma_0}(\bar x)\le L}}\left\vert\left\{\substack{\bar z\in\BA_\BZ(\zeta,\flip_Z) \text{ with}\\{\bf b}_Z(\bar z)={\bf b}\vert_{\D Z}(\bar x)}\right\}\right\vert\cdot\product_{+1}({\bf w}_\Sigma(\bar x)),
\end{multline*}
where $\tau_{\sigma\cup\zeta}$ is the $\Gamma$-adapted train track associated to ${\sigma\cup\zeta}$, and where $\tilde\sigma\in\CA(\Sigma)$ is the  unique maximal arc system in $\Sigma$ with $\sigma=\tilde\sigma\cap\Sigma_{\hyp}$.\qed
\end{lem}

\subsection{Proof of Theorem \ref{thm frequency counting}}
Let us now see how we can rewrite the right side of Theorem \ref{thm Kasra formula} in terms of Lemma \ref{lem calculating thurston measure intersection adapted train track}. Recalling the definition of $N_Z(\cdot)$ in \eqref{eq number of integral arc systems with given boundary values} we get  
$$N_Z(\bar u,\bar v,\bar v)=\sum_{\zeta \in \Lfaktor{\CA(Z)}{P \Map(Z)}}\frac{\left\vert\left\{\substack{\bar z\in\BA_\BZ(\zeta,\flip_Z) \text{ with}\\{\bf b}_Z(\bar z)=(\bar u,\bar v,\bar v)}\right\}\right\vert}{|\Stab_{P \Map(Z)}(\zeta)|}$$
for $\bar u\in\BR^{\Gamma_{\fix}}$ and $\bar v\in\BR^{\Gamma_{\ann}}$.
Fixing $\sigma\in\CA(\Sigma_{hyp})$, we then get from Lemma \ref{lem calculating thurston measure intersection adapted train track} that 
\begin{multline*}
\sum_{\zeta \in \Lfaktor{\CA_{Z}}{P \Map(Z)}}\frac{\FM_{\Thu}(\{ \mu \in \mathcal{ML}(\tau_{\sigma \cup \zeta}) : \iota(\mu, \gamma_0) \leq 1 \})}{|\Stab_{P \Map(Z)}(\zeta)|}=\\
=\lim_{L\to\infty}\frac 1{L^{6g-6}}\sum_{\substack{\bar x\in\BA_\BZ(\tilde\sigma,\flip_{\Sigma})\\ {\bf I}_{\gamma_0}(\bar x)\le L}}N_Z({\bf b}\vert_{\D Z}(\bar x))\cdot\product_{+1}({\bf w}_\Sigma(\bar x)).
\end{multline*}
When we now add over all $\sigma\in\CA(\Sigma_{\hyp})$ and invoking  Theorem \ref{thm Kasra formula} we get
$$\FC_g(\gamma_0)=\frac 1{\sym(\gamma_0)}\lim_{L\to\infty}\frac 1{L^{6g-6}}\sum_{\substack{\bar x\in\BA_\BZ(\Sigma,\flip_\Sigma)\\ {\bf I}_{\gamma_0}(\bar x)\le L}}N_Z({\bf b}\vert_{\D Z}(\bar x))\cdot\product_{+1}({\bf w}_\Sigma(\bar x)),$$
as we needed to prove.\qed

\section{The frequency $\FC(\gamma_0)$ in terms of the Kontsevich polynomial}\label{sec:kontsevich}
The title is self-explanatory, but what we might need to explain before stating the main result of this section is what is the Kontsevich polynomial.

Suppose that $Z$ is a compact orientable surface, all of whose connected components are hyperbolic. We suppose that the boundaries of $Z$ have been labeled and denote by $g_Z$ and $n_Z$ the total genus, that is the sum of the genera of all connected components, and the number of boundary components of $Z$. Now, let $\hat Z$ be the closed surface obtained by closing up each components $\D_iZ$ of the boundary to a point $x_i\in\hat Z$, and denote by
$$\CM_{g_Z,n_Z}=\CM(\hat Z,x_1,\dots,x_{n_Z})$$
the moduli space of Riemann surface structures on $\hat Z$ with the points $x_1,\dots,x_{n_Z}$ marked. The moduli space $\CM$ is neither a manifold nor compact, but it can be compactified to a good complex orbifold of complex dimension $3g_Z-3+n_Z$, and hence of real dimension $6g_Z-6+2n_Z$ (see \cite{Deligne-Mumford,Looijenga - Intersection theoryonDeligne-Mumfordcompactifications}). 
Points in $\overline{\CM}\smallsetminus \CM$ are so-called stable nodal curves. What is important is that in these nodal curves we still have our labeled points, and that they are smooth points of the given nodal curve. Now, each one of our labeled points $x_i$ determines a holomorphic line bundle $L_i\to\overline\CM$: the fiber of $L_i$ over $X\in\overline\CM$ is the cotangent space of $X$ at the marked point $x_i$. Denote by $c_1(L_i)\in H^2_{\mathrm{dR}}(\overline\CM)$ the first Chern class of the line bundle $L_i$. 

Suppose now that $d_1,\dots,d_{n_Z}\in\BZ_{\ge 0}$ are such that $d_1+\dots+d_{n_Z}=3g_Z-3+n_Z=\dim_{\BC}\overline\CM$. For mystifying reasons, it is standard to denote by $\langle \tau_{d_1} \cdots \tau_{d_{n_Z}} \rangle$ the intersection number of the classes $c_1(L_1)^{d_1},\dots,c_1(L_{n_Z})^{d_{n_Z}}$:
\begin{equation}\label{eq intersection number chern classes}
\langle \tau_{d_1} \cdots \tau_{d_n} \rangle =\int_{\overline\CM_{g_Z,n_Z}}c_1(L_1)^{d_1}\dots c_1(L_{n_Z})^{d_{n_Z}}.
\end{equation}
The {\em Kontsevich polynomial} is the degree $6g_Z-6+2n_Z$ homogeneous polynomial in $n_Z$ variables $b_1,\dots,b_{n_Z}$ given by 
\begin{equation}\label{eq Konsevich polynomial}
    V_Z(b_1, \dots, b_n)
    \coloneqq
    \sum_{\substack{(d_1, \dots, d_{n_Z}) \in \mathbb{Z}_{\geq 0}^{n_Z} \\ d_1 + \cdots + d_{n_Z} = 3g_Z-3+n_Z}}
    \langle \tau_{d_1} \cdots \tau_{d_{n_Z}} \rangle
    \prod_{i=1}^{n_Z} \frac{b_i^{2d_i}}{2^{d_i} d_i!}
\end{equation}
The reason why one denotes the variables as $b_1,\dots,b_n$ is that one wants to think of them as boundary lengths. 

\begin{rmk*}
There are different normalizations for the Kontsevich polynomial. We follow the convention of \cite{ABCGLW} because this makes the Kontsevich polynomial multiplicative: if $Z$ is the union of $Z_1$ and $Z_2$ then $V_Z=V_{Z_1}\cdot V_{Z_2}$. Indeed, the extra factors vanish because in a given surface, as soon as $(d_1, \dots, d_{n_Z}) \in \mathbb{Z}_{\geq 0}^{n_Z}$ with $ d_1 + \cdots + d_{n_Z} > 3g_Z-3+n_Z$ then $\langle \tau_{d_1} \cdots \tau_{d_{n_Z}} \rangle=0$. 
\end{rmk*}

We are now ready to state the main result of this section:

\begin{sat}\label{thm constant c in terms of intersection numbers}
Let $\gamma_0$ be a multicurve in a closed surface $X$ of genus $g\ge 3$, let $\Sigma\subset X$ be the smallest subsurface containing $\gamma_0$, and $Z$ be the union of the hyperbolic components of $X\setminus\Sigma$. Let $\flip_\Sigma$ also be the involution \eqref{eq involution induced for Sigma}, ${\bf I}_{\gamma_0}$ be the linear form \eqref{eq intersection with gamma0}, and ${\bf b}$ the boundary map on $\Sigma$ and ${\bf w}_\Sigma$ its quotient by the flip as defined in \eqref{eq:defWSigma}. Finally, set
$$\Delta_\Sigma(\gamma_0) :=\{\bar a \in \BA(\Sigma,\flip_\Sigma) : {\bf I}_{\gamma_0}(\bar a)\le 1 \}.$$
With this notation we have
$$ \FC_g(\gamma_0)=
\frac {2^{\chi(Z)+|\pi_0(Z)|}}{\sym(\gamma_0)}
\int\limits_{\Delta_\Sigma(\gamma_0)}
V_Z\left({\bf b}(\bar a)_{|\D Z}\right) \product({\bf w}(\bar a)) \, d{\FM_{\BA(\Sigma,\flip_\Sigma)}}(\bar a),$$
where $V_Z(\cdot)$ is the Kontsevich volume polynomial associated to the surface $Z$, and where $\FM_{\BA(\Sigma,\flip_\Sigma)}$ is the measure provided by Proposition \ref{prop existence of thurston like measure on flip-invariant arc complex}. 
\end{sat}

Evidently, the difference between Theorem \ref{thm frequency counting} and Theorem \ref{thm constant c in terms of intersection numbers} is that we passed from a limit of sums to an integral. In some sense, we will be distributing the exponent of $L$ in Theorem \ref{thm frequency counting} as follows
$$6g-6=3\cdot|\chi(X)|= \dim(\BA(\Sigma,\flip_\Sigma))+\deg(V_Z)+|\Gamma|.$$

Then, the main ingredient to the proof of Theorem \ref{thm constant c in terms of intersection numbers} is the duality between arcs systems and ribbon graphs. This will allow us to use work of Norbury \cite{Norbury} to make the link between the quantities $N_Z$ and $V_Z$. 

\subsection{Ribbon graphs}  \label{sec:ribbon graphs}
There is a perfect dictionary between arc systems in hyperbolic surfaces and ribbon graphs. The reason we discuss the latter here is that the work of Kontsevich and Norbury below is framed in the latter setting.
\medskip

A {\em ribbon graph} $R=(E(R),V(R))$ is a (possibly disconnected) finite trivalent graph with set of vertices $V(R)$ 
and such that for every vertex $v\in V(R)$ the set $\hal_v$ of half-edges adjacent to $v$ is endowed with a cyclic ordering. 

The etymology of the name ``ribbon graph'' is not transparent from the definition. The reason they are called like that is that {\em ribbon graphs} are basically graphs where the edges have been thickened to ribbons. This is why they are also known, less politely but maybe more suggestively, as {\em fat graphs}, such graphs are also called \emph{maps} in the literature.

Let us recall how to thicken a ribbon graph to a surface. For every vertex $v\in V(R)$, take an oriented solid polygon $P_v$ with $2\cdot\deg(v)$ sides. Color the sides of each $P_v$ alternatively in white and black. The orientation of $P_v$ induces an orientation of its boundary, and hence a cyclic order on the set of black sides of $P_v$. Fix a bijection, preserving cyclic orders, between $\hal_v$ and the set of black sides of $P_v$. Now, for every edge $e\in E(R)$ of $R$ let $v_1,v_2$ be the two (possibly equal) vertices at which the two half-edges $\vec e_1\in\hal_{v_1}$ and $\vec e_2\in\hal_{v_2}$ corresponding to $e$ are based. We identify in an orientation reversing way the $\vec e_1$-edge of $\D P_{v_1}$ with the $\vec e_2$-edge of $\D P_{v_2}$. Proceeding like this for all edges of $R$, we end up with a compact oriented surface, the {\em thickening} $\neigh(R)$ of $R$. Note that the boundary of the thickening $\neigh(R)$ is white and that each identified pair of black edges yields an arc in $\neigh(R)$. When taken together, these arcs form a maximal arc system $\zeta_R$ in $\neigh(R)$. It follows directly from the construction of $\neigh(R)$ that the graph $R$ can be embedded in $\neigh(R)$ in such a way that the image, which we still call $R$, and the arc system $\zeta_R$ are dual to each other. See Figure \ref{fig:ribToSurface} for an example.

\begin{figure}[!h] 
    \centering
        \includegraphics[scale=0.8]{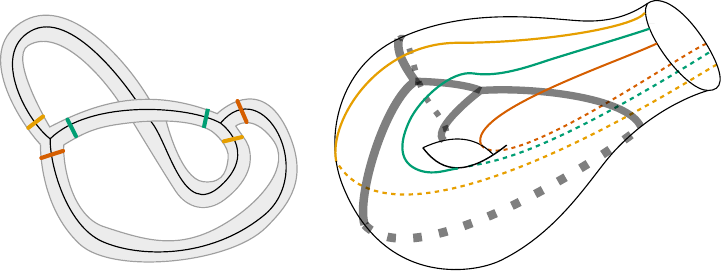}
    \caption{A ribbon graph of type $(1,1)$ with vertices oriented clockwise in the plane together with its thickening and an embedding in $S_{1,1}$}
    \label{fig:ribToSurface}
\end{figure}

By construction, the inclusion $R\to\neigh(R)$ of the ribbon graph $R$ into its thickening is a homotopy equivalence. It follows that the components of $\neigh(R)$ has negative Euler characteristic. Moreover,  we can identify free homotopy classes of curves in $R$ and in $\neigh(R)$. In particular, we can think of the boundary components of $\neigh(R)$ as curves in $R$ that we will call the {\em boundary components} of $R$. Consistently with what we did in earlier sections, we identify curves and their free homotopy classes.

\begin{quote}
{\bf Convention:} From now on, we assume that ribbon graphs come with a labeling of their boundary components. Equivalently, the boundary components of their thickenings are labeled. 
\end{quote}

Suppose now that $Z$ is a compact, oriented surface, and label its boundary components. We say that a ribbon graph $R$ {\em is of type }$Z$ if there is an orientation preserving homeomorphism $\phi:Z\to\neigh(R)$ which also preserves the labeling of the boundary components. When we pull back to $Z$ the preferred topological arc system $\zeta_R$ in $\neigh(R)$, we get an arc system on $Z$. If we pull $\zeta_R$ back under another homeomorphism $\psi:Z\to\neigh(R)$ then the two topological arc systems $\phi^{-1}(\zeta_R)$ and $\psi^{-1}(\zeta_R)$ differ by a self-homeomorphism of $Z$. In particular, the pure mapping class orbit $\PMap(Z)\cdot\phi^{-1}(\zeta_R)$ is well-defined.

Conversely, suppose that $\zeta$ is a maximal arc system in $Z$ and let $R_\zeta$ be its dual graph. We upgrade $R_\zeta$ into a ribbon graph by endowing $\hal_v$ with the cyclic ordering induced by the orientation of $Z$. We can thus see $Z$ as being itself $\neigh(R_\zeta)$, and $\zeta$ being the arc system associated to $R_\zeta$. If we replace $\zeta$ by $\phi(\zeta)$ for some self-homeomorphism $\phi$ of $Z$, then $R_\zeta$ and $R_{\phi(\zeta)}$ are isomorphic ribbon graphs. Finally, it is clear that the map sending the pure mapping class orbit of $\zeta$ to the isomorphism class of $R_\zeta$ is the inverse of the map sending the isomorphism class of $R$ to the pure mapping class orbit of $\zeta_R$. Let $\Rib^Z$ be the set of isomorphism classes of ribbon graphs of type $Z$ if follows from the discussion above the following lemma.

\begin{lem} \label{lem:BijectionArcRib}There is a bijection between the set $\Rib^Z$ of ribbon graphs of type $Z$ up to isomorphism, and the set of $\PMap(Z)$-orbits of filling topological arc systems in $Z$:
\[  \Rib^Z \longleftrightarrow \Lfaktor{\CA(Z)}{\PMap(Z)}.   \]

Moreover, if $\zeta$ is a maximal arc system in $\CA(Z)$ and $R_\zeta$ the associated ribbon graph then there is the following bijection
\[ \Aut(R_\zeta) \longleftrightarrow \Stab_{\PMap(Z)}(\zeta).  \]
\hfill \qed
\end{lem}

\subsection{Metric ribbon graphs}
A {\em metric ribbon graph} $(R,m)$ is a ribbon graph $R$ endowed with a metric $m$, that is a non-negative function $m:E(R)\to\BR_{\ge 0}$ subject to the condition that for any cycle $(e_1,\dots,e_r)$ in $R$ the sum of the values corresponding to its edges is positive: $\sum m(e_i)>0$. We think of $m(e)$ as the length of $e\in E(R)$.

Now, let $R$ be a ribbon graph of type $Z$, if $m$ is a metric on $R$, then, via the duality between the graph $R\subset\neigh(R)$ and the associated maximal arc system $\zeta_R$, we get from the metric $m$ a weight for every component of $\zeta_R$ as follows: every component $\zeta_0$ of $\zeta_R$ meets a single edge $e_0$ of $R$, and we give $\zeta_0$ the weight $m(e_0)$. Let us denote by $\zeta_{(R,m)}$ the so obtained weighted arc system, the condition on the length of cycles for the metric $m$ ensures that the support of this weighted arc system is filling.

Let's denote by $\Rib^Z(R)$ the space of metrics on $R$, the same argument as above leads now to a bijection 
\begin{equation} \label{eq:duality arcs graph}
    \Rib^Z(R) \longleftrightarrow \BA(\zeta_R).
\end{equation}

\begin{rmk*}
    We allowed the length function $m$ to take $0$ values because we wanted \eqref{eq:duality arcs graph} to hold. However, the reader might safely think that $m$ takes only positive values.
\end{rmk*}

As we pointed out earlier, we think of $m(e)$ as the length of the edge $e\in E(R)$. Let us go a step further and endow $R$ with a length metric $d_m$, with respect to which $m(e)$ the actual $d_m$-length of $e\in E(R)$. The {\em length} $\ell_m(\gamma)$ of a curve $\gamma$ in $R$ is defined as the minimum over all the $d_m$-lengths of curves in $R$ freely homotopic to $\gamma$. As suggested by the notation, $\ell_m(\gamma)$ depends uniquely on $m$. In terms of the weighted arc system $\zeta_{(R,m)}$ we have 
$$\ell_m(\gamma)={\bf I}(\zeta_{(R,m)},\gamma).$$
Anyways, if $R$ is of type $Z$, then every component of $\D Z$ determines a free homotopy class of curves in $R$. We thus get a map
$$  
\begin{array}{ccccl}
     {\bf b}&:&\Rib^Z(R)&\to &\BR_{> 0}^{\D Z}  \\
     &&(R,m)&\mapsto&(\ell_m(\gamma))_{\gamma\in\D Z}.
\end{array}
$$
The reason we denote this map in the same way the boundary map on the space of weighted arc systems is that the following diagram commutes:
\[
\begin{tikzcd}
\xymatrix{\Rib^Z(R)} \arrow[rd, "\mathbf{b}"'] \arrow[rr] &                        & \BA(\zeta_R) \arrow[ld, "\mathbf{b}"] \arrow[ll] \\
                                                        & \BR_{> 0}^{\pi_0(Z)} &                                                              
\end{tikzcd}
\]
Consistently, for $\bar b\in\BR_{\ge 0}^{\D S}$ we denote by 
$$\Rib^Z(R,\bar b)=\{m\in\Rib^Z(R)\text{ with }{\bf b}(m)=\bar b\}$$
the set of metric ribbon graphs over $R$ whose image under the boundary map ${\bf b}$ is $\bar b$.

A metric ribbon graph $(R,m)$ is {\em integral} if the metric $m$ takes values in $\BZ_{\ge 0}$. Evidently, the arc system $\zeta_{(R,m)}$ associated to an integral metric ribbon graph is integral. We thus get from the bijection above that there is a bijection between the set $\Rib^Z_\BZ(R,\bar b)$ of integer points of $\Rib^Z(R,\bar b)$ and $\BA_\BZ(\zeta_R)$. It follows from this and Lemma \ref{lem:BijectionArcRib} that the quantity $N_S(\bar b)$ introduced in \eqref{eq number of integral arc systems with given boundary values}  can now be written as
\begin{equation}\label{eq number of integral arc systems with given boundary values 2}
N_Z(\bar b)=\sum_{R\in\Rib^Z}\frac {\vert \Rib^Z_\BZ(R,\bar b)  \vert }{\vert\Aut(R)\vert}.
\end{equation}

Now that we dispose of this duality between weighted arc systems and metric ribbon graphs we can state the main theorem we are interested in, which is due to Norbury \cite{Norbury} and strongly related to work of Kontsevich \cite{Kontsevich}.

\begin{sat} \label{thm KN} \emph{\cite{Norbury}} Let $Z$ be surface made of only hyperbolic components. For any $\bar b\in \BZ_{\ge 0}^{\partial Z}$ admissible,
\[ N_{Z}(\bar b) = 2^{\chi(Z)+|\pi_0(Z)|}\cdot V_{Z}(\bar b)+ O(\|\bar b\|^{\deg(V_Z)-1 }).  \]
Moreover, the term hidden in the $O$ belongs to a finite family of polynoms.  \hfill $\blacksquare$
\end{sat}

\begin{rmk*}
Both Kontsevich and Norbury considered only the case of a connected $Z$, but everything remains true for disconnected surfaces because both $N_Z(\bar b)$ and $V_Z(\bar b)$ are multiplicative in the sense that if $Z_1$ and $Z_2$ are two surfaces then
\begin{align*}
    N_{Z_1\sqcup Z_2}(\bar b_1,\bar b_2)&=N_{Z_1}(\bar b_1)\cdot N_{Z_2}(\bar b_2),\text{ and}\\
    V_{Z_1\sqcup Z_2}(\bar b_1,\bar b_2)&=V_{Z_1}(\bar b_1)\cdot V_{Z_2}(\bar b_2)
\end{align*}
for any admissible $\bar b_1\in\BR^{\D Z_1}$ and $\bar b_2\in\BR^{\D Z_2}$.
Note also that our version of Theorem \ref{thm KN} differs from the version of Norbury by a multiplicative factor $2^{\chi(Z)+1}$, this is due to the fact that we don't use the same renormalization of the Kontsevitch polynomial.
\end{rmk*}

\subsection{Proof of Theorem \ref{thm constant c in terms of intersection numbers}}\label{sec minikuns epiphany}

In the rest of this section $\gamma_0$ is a non-filling curve of a closed hyperbolic surface $X$ of genus greater that 3, and $\Sigma$ and $Z$ are subsurfaces of $X$ defined as in the previous sections as well as the multicurve $\Gamma$ and the flip-maps $\flip, \flip_\Sigma$ and $\flip_Z$.

Let us recall the definitions introduced in section 2.4 of measures on $\BA(\Sigma,\flip_\Sigma)$, for any $L>0$ and $\sigma\in\CA(\Sigma)$ we have
$$\FM^L_{\sigma}=\frac{1}{L^{\dim(\BA(\Sigma,\flip_\Sigma))}}\sum_{\bar a\in\BA_{\BZ}(\sigma,\flip_\Sigma)}\delta_{\frac 1L\bar a}$$
where, $\delta_x$ stands for the Dirac probability measure centered at $x$. 
Recall that these measures converge to $\FM_{\BA(\sigma,\flip_\Sigma)}=\FM_{\BA(\Sigma,\flip_\Sigma)|\BA(\sigma,\flip_\Sigma)}$ as $L$ goes to infinity. 

With these reminders in mind, we come to the proof of Theorem \ref{thm constant c in terms of intersection numbers}.

\begin{proof}
From Theorem \ref{thm frequency counting} and via monotone convergence we have 
\begin{align*}
    \FC_g(\gamma_0) &=\frac 1{\sym(\gamma_0)}\lim_{L\to\infty}\frac 1{L^{6g-6}}\sum_{\substack{\bar x\in\BA_\BZ(\Sigma,\flip_\Sigma)\\ {\bf I}_{\gamma_0}(\bar x)\le L}}N_Z({\bf b}(\bar x)_{|\D Z})\cdot\product_{+1}({\bf w}_\Sigma(\bar x))\\
        & =\frac 1{\sym(\gamma_0)}
        \sum\limits_{\sigma\in\CA(\Sigma)}
        \lim_{L\to\infty}\frac 1{L^{6g-6}}\sum_{\substack{\bar x\in\BA_\BZ(\sigma,\flip_\Sigma)\\ {\bf I}_{\gamma_0}(\bar x)\le L}}N_Z({\bf b}(\bar x)_{|\D Z})\cdot\product_{+1}({\bf w}_\Sigma(\bar x)).
\end{align*}
Noting now that $6g-6=|\chi(X)|$ and that by \eqref{eq defi dim flip} and \eqref{eq Konsevich polynomial} we have
\begin{align*}
\dim(\BA(\Sigma,\flip_\Sigma))&=3\vert\chi(\Sigma)\vert-\vert\D\Sigma_{\hyp}\vert+\vert\Lfaktor{\D\Sigma}{\flip_\Sigma}\vert\\
& =3|\chi(\Sigma)|+2|\Gamma|-|\D\Sigma_{\hyp}|-|\Gamma|\\
& = 3|\chi(\Sigma)|+2|\Gamma_{\ann}|+|\Gamma_{\fix}|-|\Gamma|,\text{ and}\\
\deg(V_Z)&=3\vert\chi(Z)\vert-\vert\D Z\vert\\
& = 3\vert\chi(Z)\vert-(|\Gamma_\D|+2|\Gamma_{\ann}|), 
\end{align*}
we can decompose $3|\chi(X)|$ in the following way:
\begin{align*}
    3|\chi(X)| & =3\vert\chi(X\setminus\Gamma)\vert = 3|\chi(\Sigma)| +3|\chi(Z)| \\
               & =\dim(\BA(\Sigma,\flip_\Sigma))+\deg(V_Z)+\vert\Gamma\vert.
\end{align*}
It follows that for any $\sigma\in\BA(\Sigma)$ we can write 
\begin{align*}
    \frac 1{L^{6g-6}}\sum_{\substack{\bar x\in\BA_\BZ(\sigma,\flip_\Sigma)\\ {\bf I}_{\gamma_0}(\bar x)\le L}} &
    N_Z({\bf b}(\bar x)_{|\D Z})
    \cdot\product_{+1}({\bf w}_\Sigma(\bar x))\\
    &= \frac{1}{L^{\dim(\BA(\Sigma,\flip_\Sigma))}}
      \sum_{\substack{\bar x\in\frac{1}{L}\BA_\BZ(\sigma,\flip_\Sigma)\\ {\bf I}_{\gamma_0}(\bar x)\le 1}}
      \frac{N_Z(L{\bf b}(\bar x)_{|\D Z})}{L^{\deg(V_Z)}}\cdot
      \frac{\product_{+1}({L\bf w}_\Sigma(\bar x))}{L^{|\Gamma|}}\\
    &=\int\limits_{\Delta_\sigma(\gamma_0)}
    \frac{N_Z(L{\bf b}(\bar x)_{|\D Z})}{L^{\deg(V_Z)}}\cdot
      \frac{\product_{+1}({L\bf w}_\Sigma(\bar x))}{L^{|\Gamma|}}
      d\FM^L_\sigma(\bar x),
\end{align*}
where $\Delta_\sigma(\gamma_0)$ is defined by 
\[ \Delta_\sigma(\gamma_0) =\{\bar x \in \BA(\sigma,\flip_\Sigma) : {\bf I}_{\gamma_0}(\bar x) \le 1 \}.   \]
Now, using Theorem \ref{thm KN} and the fact that ${\bf w}_\Sigma(\bar x)\in\BR^\Gamma$ for any $\bar x\in\BA(\sigma,\flip_\Sigma)$ we get
\begin{multline*}
\frac{N_Z(L{\bf b}(\bar x)_{|\D Z})}{L^{\deg(V_Z)}}\cdot\frac{\product_{+1}({L\bf w}_\Sigma(\bar x))}{L^{|\Gamma|}} = \\
      = 2^{\chi(Z)+|\pi_0(Z)|}V_Z({\bf b}(\bar x)_{|\D Z})\product({\bf w}_\Sigma(\bar x))+ \frac{O(\|  L \bar x\|^{\deg(V_Z)+|\Gamma|-2})}{L^{\deg(V_Z)+|\Gamma|}}.
\end{multline*}
Moreover, the term hidden in the $O(\|  L \bar x\|^{\deg(V_Z)+|\Gamma|-2})$ belongs to a finite family of polynomials of degree at most $\deg(V_Z)+|\Gamma|-2$ hence the error term above can be bounded above and below by terms of the form $A/L^2$ over $\Delta_\sigma(\gamma_0)$  where $A$ is a constant. It follows that
\begin{multline*}
    \lim_{L\to\infty} \frac 1{L^{6g-6}}\sum_{\substack{\bar x\in\BA_\BZ(\sigma,\flip_\Sigma)\\ {\bf I}_{\gamma_0}(\bar x)\le L}} 
    N_Z({\bf b}(\bar x)_{|\D Z})
    \cdot\product_{+1}({\bf w}_\Sigma(\bar x))\\  
    = \int\limits_{\Delta_\sigma(\gamma_0)}
    2^{\chi(Z)+|\pi_0(Z)|}
    {V_Z({\bf b}(\bar x)_{|\D Z})}\cdot
      \product({\bf w}_\Sigma(\bar x))
      d\FM_{\BA(\Sigma,\flip_\Sigma)}(\bar x).
\end{multline*}

To conclude the proof, recall that the $\BA(\sigma,\flip_\Sigma)$ intersect along sets of zero $\FM_{\BA(\Sigma,\flip_\Sigma)}$--measure. Hence, for

\[ \Delta_\Sigma(\gamma_0)=\{\bar x \in \BA(\Sigma,\flip_\Sigma): {\bf I}_{\gamma_0}(\bar x) \le 1 \} = \bigcup\limits_{\sigma\in\CA(\Sigma)} \Delta_{\sigma}(\gamma_0), \]
we have
\begin{align*}
     \FC_g(\gamma_0)&=\frac{1}{\sym(\gamma_0)} \sum\limits_{\sigma\in\CA(\Sigma)} \int\limits_{\Delta_\sigma(\gamma_0)}
    2^{\chi(Z)+|\pi_0(Z)|}
    {V_Z({\bf b}(\bar x)_{|\D Z})}\cdot
      \product({\bf w}_\Sigma(\bar x))
      d\FM_{\BA(\Sigma,\flip_\Sigma)}(\bar x) \\
      & = \frac{2^{\chi(Z)+|\pi_0(Z)|}}{\sym(\gamma_0)}  \int\limits_{\Delta_\Sigma(\gamma_0)}
    {V_Z({\bf b}(\bar x)_{|\D Z})}\cdot
      \product({\bf w}_\Sigma(\bar x))
      d\FM_{\BA(\Sigma,\flip_\Sigma)}(\bar x).
\end{align*}
\end{proof}

\subsection{Particular cases of Theorem \ref{thm constant c in terms of intersection numbers}}
There are some specific cases of Theorem \ref{thm constant c in terms of intersection numbers} one may be interested in. Namely the case where $Z$ is connected, and the case where the curve $\gamma_0$ is a simple multicurve (and then $\Sigma=\Sigma_{\ann}$).

\subsubsection{\texorpdfstring{$\gamma_0$}{gamma_0} is a simple multicurve.}
This is a very degenerated case, where $\Sigma_{\hyp}=\emptyset$, $Z=X\setminus\CN(\gamma_0)$, $\Gamma=\Gamma_{\ann}=\gamma_0$. In this setting, the map $\pi_*$ from \eqref{eq homomorphism} should be written as 
\[  \pi_* : \Stab_{\Map(X)}(\gamma_0) \to \Map(Z),  \]
then $G=\pi_*^{-1}(\PMap(Z))$ and then $\sym(\gamma_0)=[\Stab_{\Map(X)}(\gamma_0):G]$ corresponds with the definition of $\sym(\gamma_0)$ in \cite{Maryam simple}.  Then, we obtain the following expression for $\FC_g(\gamma_0)$:

$$ \FC_g(\gamma_0)=\frac{2^{\chi(Z)+|\pi_0(Z)|}}{\sym(\gamma_0)}
\bigintsss\limits_{\tiny{\begin{array}{c}\bar b\in\BR_{\ge 0}^{\gamma_0}\\ \Vert\bar b\Vert\le 1\end{array}}}
V_Z\left(\bar b,\bar b\right)\cdot
\left(\prod_{x\in\bar b}x\right)\ d\bar b.$$
As expected, this formula is the same as obtained by Mirzakhani in \cite{Maryam simple}. Our expression differs from her by powers of two coming from different renormalization of the Kontsevich polynomial.

\subsubsection{\texorpdfstring{$Z$}{Z} connected.} In this setting, the formula in Theorem \ref{thm constant c in terms of intersection numbers} reads as follows:
\begin{equation} \label{eq:c_connected}
\FC_g(\gamma_0)=
\frac {2^{\chi(z)+1}}{\sym(\gamma_0)}
\bigintsss\limits_{\Delta_\Sigma(\gamma_0)}
{V_Z({\bf b}(\bar x)_{|\D Z})}\cdot
      \product({\bf w}_\Sigma(\bar x))
      d\FM_{\BA(\Sigma,\flip_\Sigma)}(\bar x).
\end{equation}
The difference with the general case is that the polynomial $V_Z$ is not a product of different Kontsevich polynomials.

\begin{ex}
    Some computations on given examples are made in the Appendix, specifically when the curve fills a pair of pants.
\end{ex}

\subsection{Finiteness of \texorpdfstring{$\Delta_\Sigma(\gamma_0)$}{Delta_Sigma(gamma_0)}} \label{finiteness}
In later sections we will use Theorem \ref{thm constant c in terms of intersection numbers} to investigate, in a well-defined way, the behavior of $\FC_g(\gamma_0)$ when the genus of $X$ tends to infinity. A key ingredient will be the fact that the subset $\Delta_\Sigma(\gamma_0)$ of the arc complex $\BA_\Sigma$ has finite $\FM_{\BA(\Sigma,\flip_\Sigma)}$-measure. This is what we discuss in this section. 

We will see later that curves for which $\flip_{\Sigma_{\hyp}}$ is non-trivial contribute less than the one for which it is trivial, hence, we restrict our study to the case where there is no flip outside of the annuli components of $\Sigma$.

\subsubsection{When $\gamma_0$ has no isolated simple closed component}
To start with, assume that the subsurface $\Sigma$ filled by $\gamma_0$ has no annuli components \emph{ie.} $\Sigma=\Sigma_{\hyp}$. To prove the finiteness of the volume of $\Delta_\Sigma(\gamma_0)$ we will compare the measure of this set to a measure over an easier to studied space.

 First of all, let's endow $\Sigma$ with a fixed hyperbolic metric with geodesic boundary, hence, every arc is uniquely represented by its unique othogeodesic representative. If the curve $\gamma_0$ is represented by its geodesic representative then the intersection number between $\gamma_0$ and any $\bar a \in \BA(\Sigma)$ is given by the intersections between the geodesic and (weighted) orthogeodesic representatives.
 
 The surface $\Sigma$ has $n$ boundary components and genus $g$, let $D\Sigma$ be its doubled, meaning the closed surface of genus $2g+n-1$.
 This surface is naturally coming from gluing $\Sigma$ to an orientation reversed copy of itself pointwise along the boundary as illustrated in Figure~\ref{fig:doubling}.
 Let's denote by $\Sigma^+$ and $\Sigma^-$ the two embedded copies of $\Sigma$ into $D\Sigma$, $i^+$ and $i^-$ the associated embeddings. Let $s:D\Sigma\to D\Sigma$ be the involution of $D\Sigma$ exchanging $\Sigma^+$ and $\Sigma^-$. 
 We will denote by $\CM\CL^s(D\Sigma)$ the set of all measured laminations of $D\Sigma$ stable through this involution. 

\begin{figure}[ht!]
    \centering
\begin{tikzpicture}
\path (0,0) node { \includegraphics[width=0.7\linewidth]{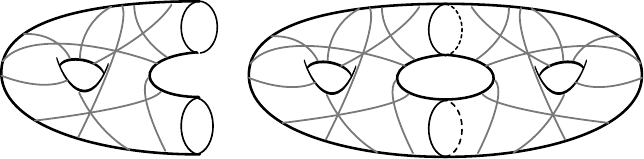}} ;
\draw
(-26mm,16mm) node {$\Sigma$} 
(-32mm,0mm) node {\textcolor{mygrey}{$\gamma_0$}}
(26mm,16mm) node {$D\Sigma$} 
(32mm,0mm) node {\textcolor{mygrey}{$\hat\gamma_0$}}
;
\end{tikzpicture} 
    \caption{Doubling of $\Sigma$ and of the filling curve $\gamma_0$.}
    \label{fig:doubling}
\end{figure}

Applying \cite[Proposition 1.5]{Hatcher}, in the same way as for $\CM\CL$ (\cite{Penner-Harer} or Theorem \ref{thm:thurston ML}) one gets that the space $\CM\CL^s(\D\Sigma)$ is a closed and locally compact space of dimension $6g-6+3n$: it is homeomorphic to $\BR^{6g-6+2n}\times\BR_{\ge 0}^n$, has a piecewise linear structure and can be seen as the union of finitely many maximal dimensional cones glued along their faces of lower dimension. Hence by proceeding as we already did previously the following measure is well-defined on $\CM\CL^s(D\Sigma)$:
\begin{equation} \label{eq:Thurston like measure}
    \FM_\sigma := \lim\limits_{L\to\infty} \frac{1}{L^{6g-6+3n}}\sum\limits_{\lambda \in \CM\CL_\BZ^{\sigma}(D\Sigma)} \delta_{\frac{1}{L}\lambda},
\end{equation}
and is the Lebesgue measure on each cone of the decomposition.

Now, given $\bar a\in \BA(\Sigma)$, by gluing $i^+(\bar a)$ together with $i^-(\bar a)$ we obtain a simple closed weighted multicurve $\hat{ a}$ of $D\Sigma$ that naturally appears as a symmetric measured lamination of $D\Sigma$. We then dispose of the doubling operator described bellow. For more details on this doubling process the reader can refer to \cite{MarieArcs}.
\begin{equation*}
    \begin{array}{ccccl}
    \hat{\cdot} &:& \BA(\Sigma)&\to&\CM\CL^s(D\Sigma) \\
    &&\alpha &\mapsto& \hat \alpha. \\
    \end{array}
\end{equation*} 
In particular, the doubling operator $\hat\cdot$ defined on $\BA(\Sigma)$ is an embedding which sends integer points to integers points. Note also that the geodesic representative of $\hat a$ corresponds exactly to the gluing of the two orthogeodesic representatives for $i^+(\bar a)$ and $i^-(\bar a)$. 

The same process also applies for multicurves, hence $\hat{\gamma}_0$ is a multicurve in $D\Sigma$. This curve is not filling but it can be completed into a filling multicurve by adding to it the equators of the genus created by the gluing, see Figure \ref{fig:make it filling}. If $\hat\gamma_e$ is the equator multicurve then the filling multicurve  $\gamma_*$ we just described decomposes as 
\[\gamma_* = \hat\gamma_0+\hat\gamma_e,\] moreover, $\hat\gamma_e$ can be seen as the double via $\hat\cdot$ of an arc system $\gamma_e\in\CA(\Sigma)$.

\begin{figure}[ht!]
    \centering
\begin{tikzpicture}
\path (0,0) node { \includegraphics[width=0.7\linewidth]{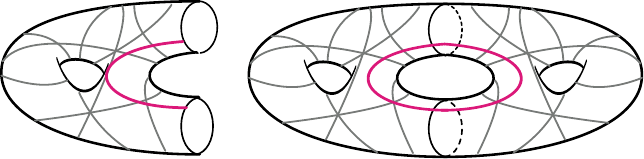}} ;
\draw
(-26mm,16mm) node {$\Sigma$} 
(-32mm,-8mm) node {\textcolor{mygrey}{$\gamma_0$}}
(-32.5mm,0mm) node {\textcolor{mypink}{$\gamma_e$}}
(26mm,16mm) node {$D\Sigma$} 
(32mm,-8mm) node {\textcolor{mygrey}{$\hat\gamma_0$}}
(35mm,-3mm) node {\textcolor{mypink}{$\hat\gamma_e$}}
;
\end{tikzpicture} 
    \caption{From $\hat \gamma_0$ to a filling multicurve $\gamma_*=\hat\gamma_0+\hat\gamma_e$}
    \label{fig:make it filling}
\end{figure} 

\begin{lem} \label{lem:change of complexity}
    With the notations above, 
    \[ \Delta_\Sigma(\gamma_0)\subset \{ \bar a \in \BA(\Sigma) | \iota(\hat a,\gamma_*)\le 2+4\iota( \gamma _0,\gamma_e) \}.\]
\end{lem}
\begin{proof}
    First of all, the decomposition $\gamma_*=\hat\gamma_0+\hat\gamma_e$ implies that for any $\bar a \in \BA(\Sigma)$, $\iota(\hat a,\gamma_*)=\iota(\hat a,\hat \gamma_0)+\iota(\hat a, \hat\gamma_e)$ where by construction $\iota(\hat a,\hat \gamma_0)=2\iota( \bar a, \gamma_0)$ and $\iota(\hat a, \hat\gamma_e)=2\iota(\bar a,\gamma_e)$. 

    All curves (resp. arcs) are represented by their geodesic (resp. orthogeodesic) representatives and are then in relative positions (they realize the minimal amount of intersection). The curve $\gamma_0$ is filling in $\Sigma$ so it cuts $\Sigma$ into crowns around $\D \Sigma$ and cells homeomorphic to disks elsewhere. Let's study the intersections between $a$ and $\gamma_e$ that happens in the crowns and the one appearing in the cells separately.
    
\begin{figure}[!h]
\begin{subfigure}[c]{0.5\textwidth}
  \centering
  \includegraphics[scale=1]{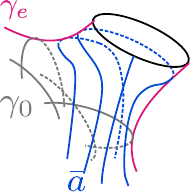}
  \caption{Intersections in the crowns}
  \label{fig:Crown Intersections}
\end{subfigure}%
\hfill
\begin{subfigure}[c]{.5\textwidth}
  \centering
  \includegraphics[scale=1]{ 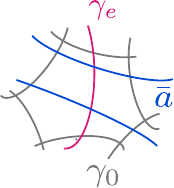}
  \caption{ Intersections in the cells}
  \label{fig:celles intersections}
\end{subfigure}
\caption{Intersections in the crowns and in the cells of $\Sigma\setminus\gamma_0$}
\label{fig:intersections}
\end{figure}

    In the crowns, since $\alpha$ is simple it cannot turn more than once around $\D\Sigma$ and each pair of intersection of $ \bar a$ with $\gamma_0$ leads to at most two intersections with $\gamma_e$ (see Figure \ref{fig:Crown Intersections}): $\iota(\bar a, \gamma_e)_{|crown}\le \iota(\bar a, \gamma_0)_{|crown}$. Outside of the crowns , each time $\bar a$ enters a cell, it crosses $\gamma_e$ at most $\frac{1}{2}\iota(\gamma_e, \gamma_0)$ times (see Figure  \ref{fig:celles intersections}), hence,  $\iota(\bar a, \gamma_e)_{|cells}\le \iota(\gamma_e, \gamma_0)\iota(\bar a,  \gamma_0)_{|cells}$ where $\iota(\gamma_e, \gamma_0)>1$.

    Moreover, $\iota(\bar a, \gamma_e)=\iota(\bar a, \gamma_e)_{|cells}+\iota(\bar a, \gamma_e)_{|crown}$ and $\iota(\bar a,  \gamma_0)_{|cells}+\iota(\bar a, \gamma_0)_{|crown}= 2\iota(\bar a, \gamma_0)$ (each intersection happens on the boundary of two different cells or crown) hence $\iota(\hat a,\gamma_*)=2\iota(\bar a, \gamma_0)+2\iota(\bar a, \gamma_e)\le 2\iota(\bar a,\gamma_0)(1+2\iota(\gamma_e, \gamma_0))$. This concludes the proof of the Lemma.
\end{proof}

We can now move to the main theorem we are interested in.

\begin{sat}\label{sat finite volume}
    Let $\Sigma$ be a compact surface with boundary such that $\Sigma=\Sigma_{\hyp}$ and let $\gamma_0\subset\Sigma$ be a filling multicurve. The set
    $$\Delta_{\Sigma}(\gamma_0)=\{\bar a\in\BA(\Sigma)\text{ with }\iota(\bar a,\gamma_0)\le 1\}$$
    has finite $\FM_{\BA(\Sigma)}$-measure.
\end{sat}

\begin{proof}
Define $\Delta_{D\Sigma}(\gamma_*)$ as $\{ \bar a \in \BA(\Sigma) | \iota(\hat a,\gamma_*)\le 2+4\iota( \gamma _0,\gamma_e)$. By definition of $\FM_{\BA(\Sigma)}$ (see \eqref{eq measure simplex arc}) and $\FM_\sigma$ (see \eqref{eq:Thurston like measure}), we have\begin{align*}
    \FM_{\BA(\Sigma)}(\Delta_{D\Sigma}(\gamma_*)) &= \lim\limits_{L\to\infty} \frac{1}{L^{6g-6+2n}}|\{ \bar a \in \BA_\BZ(\Sigma) | \iota(\hat a,\gamma_*)\le 2+4\iota(\gamma _0,\gamma_e) \}| \\
    &\le \lim\limits_{L\to\infty} \frac{1}{L^{6g-6+2n}}|\{ \lambda \in \CM\CL^s_\BZ(D\Sigma) |\iota(\lambda,\gamma_*)\le 2+4\iota( \gamma _0,\gamma_e)  \}| \\
    & = \FM_\sigma(\{ \lambda \in \CM\CL^s(D\Sigma) | \iota(\lambda,\gamma_*)\le 2+4\iota(\hat \gamma _0,\gamma_e) \}).
\end{align*}
Where the last equality holds because by homogeneity and continuity of the intersection form, $\FM_\sigma(\D\{ \lambda \in \CM\CL^s(D\Sigma) | \iota(\lambda,\gamma_*)\le 2+4\iota( \gamma _0,\gamma_e) \})=0$.

However, the intersection number with $\gamma_*$ is positive, continuous and homogeneous on $\CM\CL^s(D\Sigma)$ because $\gamma_*$ is filling so it is proper and $\{ \lambda \in \CM\CL^s(D\Sigma) |\iota(\lambda,\gamma_*)\le 2+4\iota(\hat \gamma _0,\gamma_e) \}$ is compact, hence
$$\FM_\sigma(\{ \lambda \in \CM\CL^s(D\Sigma) |\iota(\lambda,\gamma_*)\le 2+4\iota(\hat \gamma _0,\gamma_e) \})<\infty,$$ 
this together with Lemma \ref{lem:change of complexity} ends the proof of Theorem \ref{sat finite volume}. 
\end{proof}

The reason why we do not apply the argument above in $\bar\BA$ is that to obtain the properness of the intersection number one need the local compacity of $\CM\CL^s$.

\subsubsection{When $\gamma_0$ has isolated simple closed components} Let's consider now the case where $\Sigma$ decomposes as $\Sigma_{\hyp}\sqcup\Sigma_{\ann} $, we still assume that the flip acts trivially on $\D\Sigma_{\hyp}$. Then, as seen in \eqref{eq I have a lot of nice cats} we have $\bar\BA(\Sigma)=\bar\BA(\Sigma_{\hyp})\times \BR_{\ge 0}^{|\pi_0(\Sigma_{\ann})|}$ and it appears that  if $\Leb^{|\pi_0(\Sigma_{\ann})|}$ is the standard Lebesgue measure on $\BR_{\ge 0}^{|\pi_0(\Sigma_{\ann})|}$ then, as seen in Lemma \ref{lem: measure desintegration}
$$\FM_{\BA(\Sigma)}=\FM_{\BA(\Sigma_{\hyp})}\otimes\Leb^{|\pi_0(\Sigma_{\ann})|}.$$ 

For any $0\le t \le 1$, define $\Delta_\Sigma^t(\gamma_0)=\{\bar a\in\BA(\Sigma)\text{ with }\iota(\bar a,\gamma_0)\le t\}$, it follows that  

\begin{align*}
    \FM_{\BA(\Sigma)}(\Delta_\Sigma(\gamma_0))&= \int_0^1  \FM_{\BA(\Sigma_{\hyp})}(\Delta_{\Sigma_{\hyp}}^{1-t}(\gamma_0^{\hyp}))t^{|\pi_0(\Sigma_{\ann})|} dt \\
    & \le \FM_{\BA(\Sigma_{\hyp})}(\Delta_{\Sigma_{\hyp}}(\gamma_0^{\hyp}))  \int_0^1  t^{|\pi_0(\Sigma_{\ann})|} dt. \notag
\end{align*}
The following Corollary is then a direct consequence of Theorem \ref{sat finite volume}.

\begin{kor}\label{cor finite volume}
    Let $\Sigma$ be a compact surface with boundary and let $\gamma_0\subset\Sigma$ be a filling multicurve. The set
    $$\Delta_{\Sigma}(\gamma_0)=\{\bar a\in\BA(\Sigma)\text{ with }\iota(\bar a,\gamma_0)\le 1\}$$
    has finite $\FM_{\BA(\Sigma)}$-measure. \hfill \qed
\end{kor}

\section{Local types} \label{sec:local type}
Our ultimate goal is to be able to compare frequencies of curves in large genus, for this, we need to way to identify the curves we will compare in a way which is independent from the ambient surface $X_g$ which is going to be our "variable". For this we will use the notion of local type inspired from similar notions given by Anantharaman and Monk \cite{AM}. This will enable us to compare, in large genus, curves which are locally the same (\emph{ie}. of the same local type), but not globally the same (\emph{ie}. not of the same type).

\subsection{Local types and realizations}\label{sec local type def}
Let $\Sigma$ be a compact, possibly disconnected, oriented surface, all of whose components are either annuli or have negative Euler characteristic and admit boundary components. We will keep using the same decomposition of $\Sigma$ as $\Sigma_{\hyp}\sqcup\Sigma_{\ann}$.

Given $\gamma_0$ a multicurve in $\Sigma$ in the sens of Definition \ref{def: Curves}, $\gamma_0$ is said to be \emph{filling} if it cuts $\Sigma$ into topological disks and boundary parallel annuli. Equivalently, it has positive intersection number with every non boundary parallel simple closed curve of $\Sigma$ or also with every arc of~$\Sigma$.
\medskip

\begin{defi}\label{def: local type} Let $\gamma_0$ and $\Sigma$ be as above, a \emph{local type} is a triplet of the form $(\Sigma,\flip_\Sigma,\gamma_0)$ where $\flip_\Sigma$ is an involution of $\D\Sigma$ with the following properties: 
it exchanges both boundary components of every annular component of $\Sigma$ 
and some, but not all, boundary components of~$\Sigma_{\hyp}$. 

A local type is said to be \emph{essential} if the flip does not exchange any components of~$\D\Sigma_{\hyp}$. We will tend to denote such local types by $(\Sigma,\gamma_0)$.
\end{defi}

\begin{figure}[!h] 
    \centering
        \includegraphics[scale=1]{  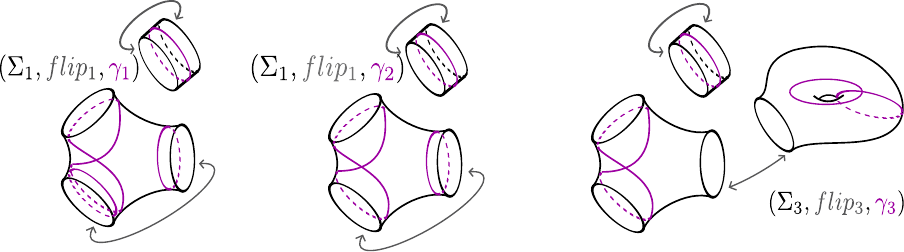}
    \caption{Examples of local types}
    \label{fig:LocalTypes}
\end{figure}

We want to see multicurves on surfaces as embeddings of local types.

\begin{defi} \label{def: Realization}
    A {\em realization} of a local type $(\Sigma,\flip_\Sigma,\gamma_0)$ in a closed surface $X$ is a class of $\pi_1$-injective  embedding
$$\phi:\Sigma\to X$$
with the property that the images of $\phi(\D_i\Sigma)$ and $\phi(\D_j\Sigma)$, two boundary components of $\Sigma$, are isotopic to each other if and only if $\flip_\Sigma(\D_i\Sigma)=\D_j\Sigma$. Two diffeomorphisms are in the same class if they differ pre or post composition by diffeomorphisms isotopic to the identity.

A curve $\gamma$ of $\Sigma$ is \emph{of local type} $(\Sigma,\flip_\Sigma,\gamma_0)$ if there is a realization $\phi$ of $(\Sigma,\flip_\Sigma,\gamma_0)$ such that $\phi(\gamma_0)=\gamma.$
\end{defi}

\begin{figure}[!h] 
    \centering
        \includegraphics[scale=0.8]{  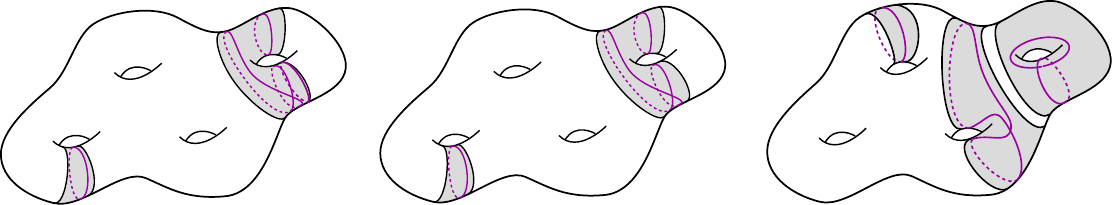}
    \caption{Realizations for the local types of Figure \ref{fig:LocalTypes}}
    \label{fig:Realisations}
\end{figure}

Let us add some comments that the reader should keep in mind.
\begin{enumerate}
\item Some local types may represent partially weighted multicurves, this will not change anything in the following, the choice to consider not weighted curves in general was made to avoid unnecessary notations. (In the first example of Figure \ref{fig:LocalTypes}, the two boundary parallels curves in the pair of pants will produce in $X$ a simple closed curve with weight $2$).
\item There might be different local types $(\Sigma,\flip_\Sigma,\gamma_0)$ and $(\Sigma',\flip_{\Sigma'},\gamma_0')$ for which there are realizations $\phi$ and $\phi'$ with $\phi(\gamma_0)=\phi'(\gamma_0')$, where equal here means that they are freely homotopic to each other. (See Figures  \ref{fig:LocalTypes} and \ref{fig:Realisations}, second and third examples.)
\item If $\Sigma',\Sigma''\in\pi_0(\Sigma)$ are distinct connected components of $\Sigma$ which have boundary components $\D_i\Sigma\in\D\Sigma'$ and $\D_j\Sigma\in\D\Sigma''$ such that $\flip_\Sigma(\D_i\Sigma)=\D_j\Sigma$, then both $\Sigma'$ and $\Sigma''$ are hyperbolic components of $\Sigma$. (See Figure \ref{fig:LocalTypes}, fourth example.)
\item The involution $\flip_\Sigma$ is uniquely determined by the realization $\phi$ of the local type $(\Sigma,\flip_\Sigma,\gamma_0)$. 
 \end{enumerate}
\medskip

Now, note that for any non filling curve $\gamma_0$ of a closed surface $X$ of genus $g\ge 2$, by taking  $\Sigma$ and $\flip_\Sigma$ as defined in the previous sections, the triplet $(\Sigma, \flip_\Sigma,\gamma_0)$ defines a local type with a natural realization in $X$ coming from the fact that $\gamma_0$ lives in $X$. We then dispose of the following proposition.

\begin{prop}
    Let $X$ be a closed surface of genus $g\ge 2$. For any non-filling multicurve $\gamma$ in $X$ there is a local type $(\Sigma,\flip_\Sigma,\gamma_0)$ and a realization $\phi:\Sigma \to X$ such that $\phi(\gamma_0)=\gamma$.
\end{prop}

Following this link with the previous sections, we associate to each realization of a local type a multicurve $\Gamma_\phi=\Gamma$ in $X$ and a subsurface $Z_\phi=Z$ of $X$. Let's describe them in term of the local type and the realization. One can ffind an example in Figure \ref{fig:FullLocalType}.

Let $(\Sigma, \flip_\Sigma,\gamma_0)$ be a local type and $\phi$ a realization in a closed surface $X$.
\begin{itemize}
    \item $Z=Z_\phi$ is the hyperbolic part of $X\setminus\phi(\Sigma\setminus\D\Sigma)$. It is the same subsurface as $Z_{\phi(\gamma_0)}$ defined in \eqref{eq:defZ}.
    \item $\Gamma=\Gamma_\phi$ is the simple multicurve of $X$ given by $\phi(\D\Sigma)$. It is the same multicurve as $\Gamma_{\phi(\gamma_0)}$ defined in \eqref{eq:defGamma}. 
\end{itemize}

\begin{figure}[!h] 
    \centering
        \includegraphics[scale=0.83]{  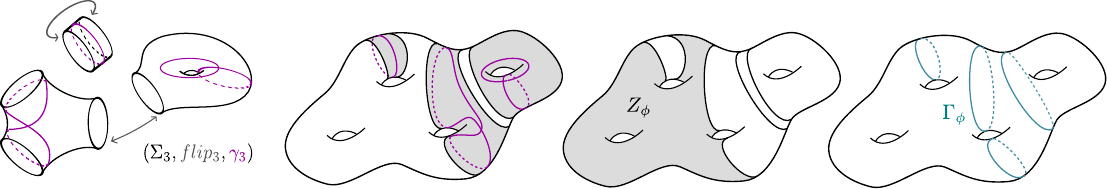}
    \caption{Local type, realization $Z_{\phi}$ and $\Gamma_{\phi}$}
    \label{fig:FullLocalType}
\end{figure}

By definition of a realization $|\Gamma_\phi|$ naturally identifies with $\Lfaktor{\D\Sigma}{\flip_\Sigma}$, hence, this curve can be seen as independent of $\phi$ (which describes the way it is embedded in $X$). Hence, we will write $\Gamma$ instead of $\Gamma_\phi$. Also, the multicurves $\Gamma_{\fix}$ and $\Gamma_{\ann}$ defined in Section \ref{sec:FrequencyCounting} for $\Gamma=\Gamma_{\phi(\gamma_0)}$ are also well defined independently of the realization. In particular, the cardinality of $\Gamma,\Gamma_{\fix}$ and $\Gamma_{\ann}$ depend only on the local type.

\subsection{Types and local types}
Local types can be seen as a way to describe the local shape of a curve in a surface. Hence, if two curves are of the same type (\emph{i.e.}\ they are in the same mapping class group orbit) in particular they are of the same local type. This is because the post-composition of a realization by a mapping class is still a realization. However. the reverse is false. For example, if $\phi_0$ and $\phi_1$ are two realizations of the same local type such that $Z_0$ and $Z_1$ don't have the same number of connected components then $\phi(\gamma_0)$ and $\phi_1(\gamma_0)$ cannot be of the same type, in the same way as separating and non-separating simple closed curves are not of the same type. 

\begin{prop} \label{prop:finitely many realization}
    Let $(\Sigma, \flip_\Sigma,\gamma_0)$ be a local type and $X$ a closed surface. There is finitely many types of curves in $X$ of local type $(\Sigma, \flip_\Sigma,\gamma_0)$.
\end{prop}
\begin{proof}
    If the genus of $X$ is too small maybe there is no realization of $(\Sigma, \flip_\Sigma,\gamma_0)$ in $X$ and the result is direct. Otherwise, all the curves which have the same local type have the same number of self intersection $\iota(\gamma_0,\gamma_0)$. Since there is finitely many different types of fixed intersection number, there is also finitely many types of local type $(\Sigma, \flip_\Sigma,\gamma_0)$.
\end{proof}

\begin{rmk*}
    Remark that this also proves that given a surface $X$, there is only finitely many local types that can be realized in $X$ and of fixed intersection number.
    Moreover, if one restricts to local types such that $\Sigma=\Sigma_{\hyp}$, for any $k>0$ there is only finitely many local types of self intersection $k$.
\end{rmk*}

Among those finitely many types of a given local type one will be of specific interest.

\begin{defi} \label{sep local type}
    Let $(\Sigma,\flip_\Sigma,\gamma_0)$ be a local type and $\phi : \Sigma \to X$ a realization into a closed surface $X$. The realization is said to be \emph{separating} if $Z_\phi$ is disconnected. Otherwise it is said to be \emph{non-separating}. See example in Figure \ref{fig:Sep adn non-sep real}.
\end{defi}

\begin{figure}[h!]
    \centering
    \includegraphics[width=\linewidth]{ 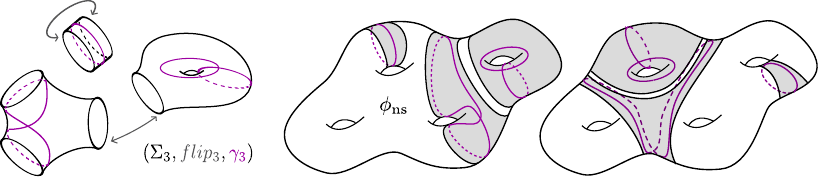}
    \caption{A non-separating and a separating realization of the same local type}
    \label{fig:Sep adn non-sep real}
\end{figure}

In the same way as there is only one type of non-separating simple closed curve there is, in some sens, a single non-separating realization a each local type.

\begin{prop}
    Given a local type $(\Sigma, \flip_\Sigma,\gamma_0)$ and a surface $X$ of genus big enough such that $(\Sigma, \flip_\Sigma,\gamma_0)$ can be realized in it. Up to post-composition by a mapping class, there is a unique non separating realization $\phi_{\ns}^g:\Sigma\to X$ of this local type into $X$. \qed
\end{prop}

Note that we will tend to consider realization up to post-composition by mapping classes and then tend to talk about \textbf{the} non-separating realization, when the context is clear we will write $\phi_{\ns}$ instead of $\phi_{\ns}^g$ for this realization. 
We will extensively study it so it is worth noticing the following fact.

\begin{prop}   \label{prop: sym independant of g} 
    Given a local type $(\Sigma, \flip_\Sigma,\gamma_0)$,  there exists $g_0$ such that for all $g\ge g_0$, $$\sym(\phi_{\ns}^g(\gamma_0))= |\Stab_{\Map(\Sigma)}(\gamma_0)|.$$
\end{prop}

\begin{proof}
    Let us recall some definitions from Section \ref{sec : subgroup}. For a curve $\gamma$ in a surface $X$ we have $\sym(\gamma)=[\Stab_{\Map(X)}(\gamma): G]$ where $G=\pi_*^{-1}(\{\Id_{\Sigma_\hyp} \}\times \PMap(Z))$ with 
    \[ \pi_* :  \Stab_{\Map(X)}(\gamma) \to \Map(\Sigma_\hyp)\times \Map(Z). \]
    One can also consider the map $\varphi : \Stab_{\Map(X)}(\gamma) \to \Stab_{\Map(\Sigma)}(\gamma)$. Hence, $G$ is the kernel of this map and which is surjective as soon as $Z$ is connected and in that case
    \[ \sym(\gamma) = |\text{Im}(\varphi)|=|\Stab_{\Map(\Sigma)}(\gamma)|.      \]

    So for a local type $(\Sigma, \flip_\Sigma,\gamma_0)$ and $g$ big enough such that $\phi_{\ns}^g(\gamma_0)$ exists we have
    \[ \sym(\phi_{\ns}^g(\gamma_0)) = |\Stab_{\Map(\Sigma)}(\gamma_0)|.  \]
    \end{proof}

In the following, given a realization $\phi$ of a local type $(\Sigma, \flip_\Sigma,\gamma_0)$ we will write $\sym(\phi,\gamma_0)$ for $\sym(\phi(\gamma_0))$.

\subsection{Frequencies for local types} 
The reason why we introduce local types is to have a clear setting to compare frequencies in large genus. Indeed, given a local type an a sequence $\phi_g$ of realization into the surface $X_g$ of genus $g$, the sequence $\FC_g(\phi_g(\gamma_0))$ is well defined and can be studied as $g$ goes to infinity. This can be seen as the asymptotic frequency of a given curve $\gamma_0$. It is this kind of asymptotic that we will study in the next sections. To do so, we need to rewrite our expressions of the frequencies in term of local types. The only one we will need is the one of Theorem \ref{thm constant c in terms of intersection numbers}.

\begin{kor} \label{kor:formulationLocalType}
Let $X$ be a closed hyperbolic surface of genus $g$ and $(\Sigma, \flip_\Sigma,\gamma_0)$ a local type. Consider a realization $\phi \colon \Sigma \to X$ of this local type into $X$. With the notation fixed in Section \ref{sec local type def} and with ${\bf I}_{\gamma_0}:\BA(\Sigma,\flip_\Sigma)\to\BR_{\ge 0}$ the intersection function with $\gamma_0$. Setting
$$\Delta_\Sigma(\gamma_0) :=\{\bar a \in \BA(\Sigma,\flip_\Sigma) : {\bf I}_{\gamma_0}(\bar a)\le 1 \}$$
we have
$$ \FC_g(\phi(\gamma_0))=
\frac {2^{\chi(Z)+|\pi_0(Z)|}}{\sym(\phi,\gamma_0)}
\bigintsss\limits_{\Delta_\Sigma(\gamma_0)}
V_Z\left({\bf b}(\bar a)_{|\D Z_\phi}\right)
\cdot \product({\bf w} (\bar a))  d{\FM_{\BA(\Sigma,\flip_\Sigma)}}(\bar a)$$
where $V_Z(\cdot)$ is the Kontsevich volume polynomial associated to the surface $Z_\phi$, and where $\FM_{\BA(\Sigma,\flip_\Sigma)}$ is the measure provided by Proposition \ref{prop existence of thurston like measure on flip-invariant arc complex}.  \qed
\end{kor}

\subsection{Graph dual to a realization}
Later on, it will be useful to rewrite Corollary \ref{kor:formulationLocalType} in terms of \emph{the graph dual to the realization $\phi:\Sigma\to X$} of the local type $(\Sigma,\flip_\Sigma,\gamma_0)$, there is an example in Figure \ref{fig:dualGraph}. This is a decorated graph given $G_\phi$ defined as follows:
\begin{itemize}
\item The vertex set $\CV$ of $G_\phi$ is the set of connected components of $X\setminus\CN(\Gamma_\phi)\equiv \Sigma_{\hyp}\sqcup Z\equiv \phi(\Sigma_{\hyp})\sqcup Z$. In particular we have
$$\CV=\CV_{\Sigma}\sqcup \CV_Z$$
where $\CV_\Sigma$ is the set of connected components of $\Sigma_{hyp}$ and $\CV_Z$ is the set of connected components of $Z$.
\item The edge set $E(G_\phi)$ is the set of components of the multicurve $\Gamma$. Recall that any component of $\Gamma$ corresponds to a pair of components of $\D(Z\sqcup\phi(\Sigma_{\hyp}))$. The endpoints of the associated edge of $G_\phi$ are the two (possibly equal) vertices corresponding to the subsurfaces containing those two components of $\D(Z\sqcup\phi(\Sigma_{\hyp}))$.
\item The vertices of $G_\phi$ are decorated with the genus $g(v)$ of the surface they represent.
\end{itemize}
Note that every connected component of $\Sigma_{\hyp}\cup Z$ has exactly $\deg(v)$ boundary components, where $v$ is the corresponding vertex of $G_\phi$. In particular, the Euler characteristic of the component of $\Sigma_{\hyp}\sqcup Z$ corresponding to $v$ is $\chi_v=(2-2g(v)-\deg(v))$. Note that
$$\chi(X)=\chi(\Sigma_{\hyp})+\chi(Z)=\chi(\Sigma)+\sum_{v\in \CV_Z}\chi_v.$$

\begin{figure}[!h] 
    \centering
        \includegraphics[scale=1]{  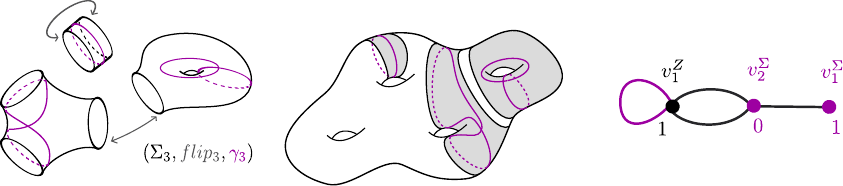}
    \caption{Example of dual graph}
    \label{fig:dualGraph}
\end{figure}

For later use we will need to bound the number of realizations of a given type having the same dual graph. Let's first make clear what we mean by \emph{the same dual graph}.

\begin{defi}
    Two realizations $\phi$ and $\phi'$ of a given local type $(\Sigma,\flip_\Sigma,\gamma_0)$ into a surface $X$ are said to have \emph{the same dual graph} if there is a graph isomorphism $I : G_\phi \to G_{\phi'}$ with
    $$I(\CV_\Sigma^\phi)=\CV_\Sigma^{\phi'}\text{ and }g'(I(v))=g(v)\text{ for all }v\in \CV_\phi.$$
    Here, $g(\cdot)$ and $g'(\cdot)$ are respectively the decorations of $G_\phi$ and $G_{\phi'}$, that is the functions sending each vertex to the genus of the associated subsurface.
\end{defi}

The following lemma asserts that realizations with the same dual graph differ by pre and post composition by diffeomorphisms of $\Sigma$ and $X$. 

\begin{lem} \label{lem:step1DualGraph}
     Let $(\Sigma,\flip_\Sigma,\gamma_0)$ be a local type and $X$ a surface of big enough genus such that $(\Sigma,\flip_\Sigma,\gamma_0)$ can be realized in it. If $\phi:\Sigma\to X$ and $\phi':\Sigma\to X$ are two realizations with the same dual graph, then there are diffeomorphisms $\psi_\Sigma:\Sigma\to\Sigma$ and $\psi_X:X\to X$ such that $\phi=\psi_X^{-1}\circ\phi'\circ\psi_\Sigma$. 
\end{lem}
\begin{proof}
    Let $\phi$ and $\phi'$ be realizations with the same dual graph. The sets $\D (X\setminus\phi(\Sigma))$ and $\D (X\setminus\phi(\Sigma))$ have the same number of connected components because both sets are in bijection with the set of half edges of $G$. Furthermore, the complementary regions of these sets have the same topology (genus and number of boundary) when they correspond to the same vertex. Hence, there is diffeomorphism $\psi_X : X \to X$ such that
    \begin{itemize}
        \item $\psi_X(\phi(\D \Sigma))=\phi'(\D \Sigma)$ and if $\gamma\in\phi(\D\Sigma)$ correspond to the half edge $e$ in $G$ and the same for $\gamma'\in\phi'(\Sigma)$ then $\psi_X(\gamma)=\gamma'$, and
        \item $\psi_X(\phi(\Sigma))=\phi'(\Sigma)$ and if $\Sigma_v$ is the connected component of $\Sigma_\hyp$ associated to the vertex $v$ for the realization $\phi$ and $\Sigma_v'$ then one for $\phi'$ the $\psi_X(\phi(\Sigma_v))=\phi'(\Sigma_v')$.
    \end{itemize}
    Note that, by the first point, images under $\phi$ of annular components are sent to images under $\phi'$ of annular components.

    Let $\psi_\Sigma:\Sigma\to\Sigma$ be the diffeomorphism defined by 
    $  (\phi'_{\vert \phi'(\Sigma)})^{-1} \circ \psi_{X\vert \phi(\Sigma)}\circ \phi.     $
    By construction $\phi=\psi_X^{-1}\circ\phi'\circ\psi_\Sigma$. 
\end{proof}

Recall that we identify isotopic realizations. In particular $\Map(X)$ acts on the set of all realizations $\phi:\Sigma\to X$. It follows directly from Lemma \ref{lem:step1DualGraph} that there are at most $[\Map(\Sigma):\PMap(\Sigma)]$ many $\Map(X)$-orbits of realizations $\phi:\Sigma\to X$ with given dual graph. Let us record this fact:

\begin{prop}\label{prop:NumberOfRealizationForAGraph} 
        Let $(\Sigma,\flip_\Sigma,\gamma_0)$ be a local type and $X$ a surface of big enough genus such that $(\Sigma,\flip_\Sigma,\gamma_0)$ can be realized in it. For a given realization $\phi:\Sigma\to X$, there is (up to composition by $\Map(X)$) at most $[\Map(\Sigma):\PMap(\Sigma)]$ different realizations with the same dual graph as $\phi$.\qed
\end{prop}

\begin{figure}[!ht]
  \centering
{\includegraphics[width=0.3\textwidth]{  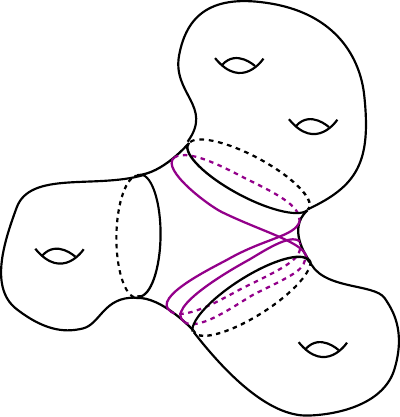}\label{fig:f1}}
  \hfill
{\includegraphics[width=0.3\textwidth]{  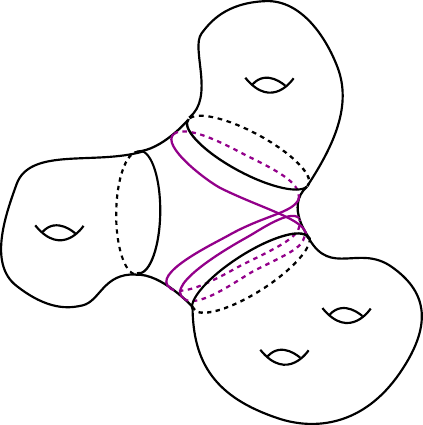}\label{fig:f2}}
 \hfill
{\includegraphics[width=0.3\textwidth]{  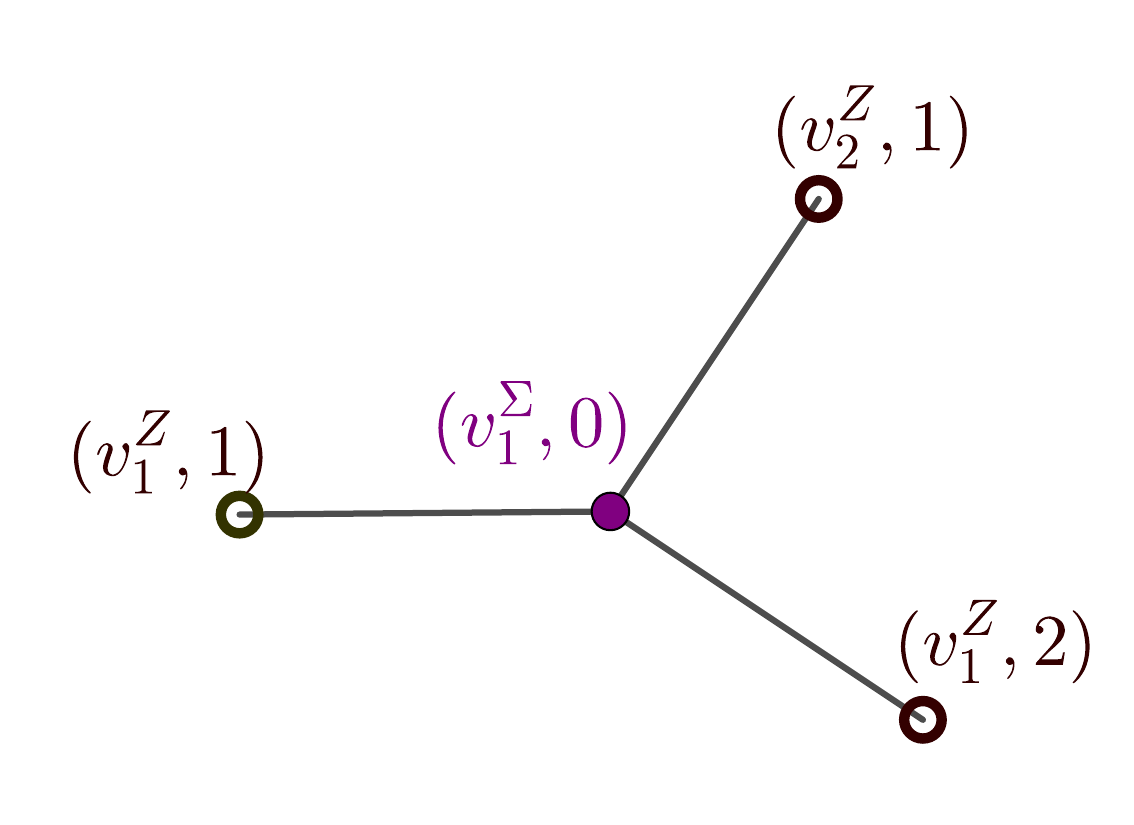}\label{fig:f3}}
  \caption{Different realizations with the same dual graph differing by a non-pure homeomorphism $\psi_\Sigma$}
\end{figure}

Now that it should be clear what is the dual graph $G_\phi$ of a realization $\phi:\Sigma\to X$, let us rewrite Corollary \ref{kor:formulationLocalType} using this notation. Via the bijection
\begin{center}
    \begin{tabular}{ccc}
$E(G_\phi)$&$\to$&$\Gamma$\\  $e$&$\mapsto$&$\gamma_e$,
\end{tabular}
\end{center}
we can identify $\BR^\Gamma$ and $\BR^{E(G_\phi)}$. In particular we can think of the map ${\bf w}$ as taking values in $\BR^{E(G_\phi)}$. To bring our notation in line with what is usual in the field, we define 
$$b_e(\bar a):={\bf w}(\bar a)_{\gamma_e}$$
for $e\in E(G_\phi)$ and $\bar a \in \BA(\Sigma,\flip_\Sigma)$. Also, noting that if $v\in \CV_Z$ is a vertex of $G_\phi$ associated to a connected component $Z_v$ of $Z$, then we have an identification between the set $\hal_v$ of half-edges emanating out of $v$ and the set of boundary components of $Z_v$. Given $\bar a\in\BA(\Sigma,\flip_\Sigma)$ set
$$\bar b _v(\bar a) := (b_e(\bar a))_{e\in \hal_v}.$$
With this notation, the following is a direct rewriting of Corollary \ref{kor:formulationLocalType}:

\begin{kor} \label{cor: cste in term of dual graph}  
Let $X$ be a closed hyperbolic surface of genus $g$ and $(\Sigma, \flip_\Sigma,\gamma_0)$ a local type. Consider a realization $\phi \colon \Sigma \to X$ of this local type into $X$. With the notation fixed above and with ${\bf I}_{\gamma_0}:\BA(\Sigma,\flip_\Sigma)\to\BR_{\ge 0}$ the intersection function with $\gamma_0$ set
$$\Delta_\Sigma(\gamma_0) :=\{\bar a \in \BA(\Sigma,\flip_\Sigma) : {\bf I}_{\gamma_0}(\bar a)\le 1 \}.$$
We then have
\begin{multline}\label{eq: constant in graph}
        \FC_g(\phi(\gamma_0)) = \\
    \frac{1}{\sym(\gamma_0,\phi)} 
    \int\limits_{\Delta_\Sigma(\gamma_0)}
    \left( \prod\limits_{\tiny v\in \CV_Z}2^{\chi_v+1}V_{g(v),\deg(v)}(\bar b_v (\bar a))\right) \cdot \left( \prod\limits_{\tiny e\in E} b_e(\bar a)\right)\
    d\FM_{\BA(\Sigma,\flip_\Sigma)(\bar a)}.
\end{multline} 
where $V_{g,n}$ is the Kontsevich polynomial for surfaces of genus $g$ and $n$ boundary components, and 
where $\FM_{\BA(\Sigma,\flip_\Sigma)}$ is the measure provided by Proposition \ref{prop existence of thurston like measure on flip-invariant arc complex}. 
\end{kor}

Before moving on, note that the decomposition $\Gamma=\Gamma_{\flip}\sqcup\Gamma_{\ann}\sqcup\Gamma_{\fix}$ of $\Gamma$ induces a decomposition $E=E_{\flip}\sqcup E_{\ann}\sqcup E_{\fix}$ of the set of edges of  $G_\phi$. In terms of the involution $\flip$ and of the decomposition $\CV=\CV_\Sigma\sqcup\CV_Z$, these three sets of edges have the following descriptions:
\begin{itemize}
    \item[$E_{\flip}$:] Edges in $E_{\flip}$ correspond to $\flip$-orbits contained in $\D\Sigma_{\hyp}$. These are edges with both endpoints in $\CV_\Sigma$.
    \item[$E_{\ann}$:] Edges in $E_{\ann}$ correspond to $\flip$-orbits contained in $\D Z$, or equivalently contained in $\D\Sigma_{\ann}$. These are edges with both endpoints in $\CV_Z$.
    \item[$E_{\fix}$:] Edges in $E_{\fix}$ correspond to $\flip$-orbits which meet both $\D\Sigma_{\hyp}$ and $\D Z$. These are edges connecting $\CV_\Sigma$ and $\CV_Z$.
\end{itemize}

\section{Large genus asymptotics for frequencies}\label{sec asymptotic} 

 In this section we give different asymptotics for the quantity $\FC_g(\phi(\gamma_0))$ when the genus of $X$ grows. This is based on Aggarwal's results providing asymptotic values for the coefficients of the Kontsevich polynomial.

\subsection{Large genus asymptotics for intersection numbers}

The coefficients $\langle \tau_{d_1} \dots \tau_{d_n} \rangle_{g, n}$ of the Kontsevich polynomial can be computed recursively,
but the recursive relations are rather intricate.
Fortunately, when $g \to \infty$ and $n$ remains small compared to $g$, an asymptotic formula was conjectured by Delecroix--Goujard--Zograf--Zorich \cite{DGZZ} on the basis of numerical experiments, and was proved shortly thereafter by Aggarwal \cite{Aggarwal} through a combinatorial analysis of the Virasoro contraints.
This is the result we will use.
\begin{sat}[\cite{Aggarwal}] \label{Aggarwal}
Let $g \in \mathbb{Z}_{\geq 0}$, $n \in \mathbb{Z}_{\geq 1}$, $\bar{d} = (d_1, \dots, d_n) \in \mathbb{Z}_{\geq 0}^n$ with $|d| = d_1 + \cdots + d_n = 3g - 3 + n$.
Define $\varepsilon_{g,n}(\bar{d})$ by
\begin{equation} \label{eq:psiSim}
    \langle \tau_{d_1} \cdots \tau_{d_n} \rangle_{g, n}
    =
    \left(\dfrac{(6g-5+2n)!!}{ \prod_{i=1}^{n} (2d_i + 1)!!}
    \frac{1}{g! 24^g}\right)
    \left(1 + \varepsilon_n(g,\bar d) \right).
\end{equation}
Then for any fixed $n$, we have
\[
    \lim_{g \to \infty}
    \max_{\substack{\bar{d} = (d_1, \dots, d_n) \in \mathbb{Z}_{\geq 1}^n \\ |\bar{d}| = 3g-3+n}}
    |\varepsilon_{g,n}(\bar{d})|
    =
    0.
\]
\end{sat}

\begin{rmk*}
Aggarwal's result is in fact stronger allowing $n$ to grow as long as $n = \mathrm{o}(\sqrt{g})$. 
Similar results were obtained later in \cite{Guo-Yang} through a combinatorial approach, and in \cite{resurgence} using a resurgence analysis.
Notably, the later provides an error term for fixed $n$,
which could strengthen our estimates in several instances.
As this paper is already heavy enough, we do not pursue these refinements here.
\end{rmk*}

The function $\varepsilon_{g,n}(\bar d)$ is bounded above and below uniformly in $\bar d$ by a function of $g$ going to $0$ as $g$ goes to infinity, this uniformity ensures that we have 
\begin{equation} \label{eq:asymprotic for volume polynomial}
    V_{g,n}(b_1,\cdots,b_n)
    \underset{g\to \infty}{\sim}
    \sum_{\substack{(d_1, \dots, d_n) \in \mathbb{Z}_{\geq 0}^n \\ d_1 + \cdots d_n = 3g-3+n}} 
    \dfrac{(6g-5+2n)!!}{g!24^g}
    \prod_{i=1}^{n} \frac{b_i^{2d_i}}{2^{d_i} d_i!(2d_i+1)!!},
\end{equation}  
uniformly in $\bar b$.
Now, it is clear that to understand the Kontsevich polynomial $V_{g,n}$ we need to understand the polynomial
$$\tilde V_{g,n}(b_1,\dots,b_n)=    \sum_{\substack{(d_1, \dots, d_n) \in \mathbb{Z}_{\geq 0}^n \\ d_1 + \cdots d_n = 3g-3+n}}
    \frac{(6g-5+2n)!!}{g!24^g\prod_{i=1}^n(2d_i+1)!!} \prod_{i=1}^{n} \frac{b_i^{2d_i}}{2^{d_i} d_i!}.
$$
Invoking the relation $(2k+1)!!=\frac{(2k+1)!}{2^k\cdot k!}$ we get a simpler expression for this latter polynomial:
$$\tilde V_{g,n}(b_1,\dots,b_n)=    \sum_{\substack{(d_1, \dots, d_n) \in \mathbb{Z}_{\geq 0}^n \\ d_1 + \cdots d_n = 3g-3+n}}
    \frac{(6g-5+2n)!!}{g!24^g} \prod_{i=1}^{n} \frac{b_i^{2d_i}}{(2d_i+1)!}.
$$
Denote by $[z^k]f=\frac{f^{(z)}(0)}{k!}$ the coefficient of $z^k$ in the Taylor expansion at $0$ of a, say analytic, function $f$. With this notation we get that
$$[z^{6g-6+3n}]\prod_{i=1}^n\frac{\sinh(b_iz)}{b_i}=\sum_{\substack{(d_1, \dots, d_n) \in \mathbb{Z}_{\geq 0}^n \\ d_1 + \cdots + d_n = 3g-3+n}}\prod_{i=1}^{n} \frac{b_i^{2d_i}}{(2d_i+1)!}.$$
We can thus write our polynomial as
\begin{equation} \label{eq:first version with Taylor expansion}
    \tilde V_{g,n}=\frac{(6g-5+2n)!!}{g! 24^g}
        \, [z^{6g-6+3n}] \prod_{i=1}^{n} \frac{\sinh(b_i z)}{b_i}.
\end{equation}
We get then from \eqref{eq:psiSim} that as soon as $g$ is large enough we have
    \begin{equation} \label{eq:Vasymp}
        V_{g,n}(\bar b)
        = 
        \left (
        \frac{(6g-5+2n)!!}{g! 24^g}
        \, [z^{6g-6+3n}] \prod_{i=1}^{n} \frac{\sinh(b_i z)}{b_i}
        \right)
        (1+\varepsilon_{n}(g,\bar b)),
    \end{equation}
where $(1+\varepsilon_{n}(g,\bar b))$ tends to $1$ uniformly in $\bar b$. For arbitrary $n$, we also have the following estimate that also extends to a polynomial formualtion.
\begin{sat}[\cite{Aggarwal}]
    Let $g \geq 0$, $n \geq 1$, $(d_1, \dots, d_n) \in \mathbb{Z}_{\geq 0}^n$ with $d_1 + \cdots + d_n = 3g-3+n$. We have
    \[
        \langle \tau_{d_1} \cdots \tau_{d_n} \rangle_{g,n}
        \leq
        \frac{(6g-5+2n)!!}{g! 24^g} \frac{1}{\prod_{i=1}^{n} (2d_i + 1)!!}
        \left( \frac{3}{2} \right)^{n-1}.
    \]
    \hfill $\blacksquare$
\end{sat}
Let's record all of it for later in the way we need it.

\begin{kor}\label{kor useful aggarwal}
    For any $n \in \mathbb{Z}_{\geq 0}$ and any $(b_1, \dots, b_n) \in \mathbb{R}_{>0}^n$, 
    Moreover,
    \begin{equation} \label{eq:Vleq}
        V_{g,n}(b_1, \dots, b_n)
        \le 
        \left( \frac{3}{2} \right)^{n-1} 
        \frac{(6g-5+2n)!!}{g! 24^g}
        \, [z^{6g-6+3n}] \prod_{i=1}^{n} \frac{\sinh(b_i z)}{b_i},
    \end{equation}

    \begin{equation} \label{eq:asymptotic for volume with taylor coef}
        V_{g,n}(b_1, \dots, b_n)
        \sim
        \left(
        \frac{(6g-5+2n)!!}{g! 24^g}
        \, [z^{6g-6+3n}] \prod_{i=1}^{n} \frac{\sinh(b_i z)}{b_i}
        \right)
    \end{equation}
    uniformly in $(b_1, \dots, b_n) \in \mathbb{R}_{>0}^{n}$. \qed
\end{kor}

\subsection{Dominant embedding} 
In this section we first investigate which kind of local types and realizations are (asymptotically) dominant. That will justify why we restrict ourselves to non-separating realizations. Indeed, here we prove that the non-separating realization of a given local type is dominant over the separating ones. Before stating the precise result, we need some notation. 
Given a local type $(\Sigma,\flip_\Sigma,\gamma_0)$ and a closed surface $X$, Proposition \ref{prop:finitely many realization}  ensures that there is $m_g>0$ and finitely many realizations $\phi_1,\dots,\phi_{m_g}$ in $X$ such that 
$$\{\gamma\subset X\text{ of local type }(\Sigma,\flip_\Sigma,\gamma_0)\}=\bigsqcup_{i=1,\dots,m_g}(\Map(X)\cdot\phi_i(\gamma_0)).$$
With this notation we define
$$\FC_g(\Sigma,\flip_\Sigma,\gamma_0)=\sum_{i=1,\dots,m_g}\FC_g(\phi_i(\gamma_0)).$$
Our next goal is to prove the following:

\begin{sat}\label{sat dominant type}
    Let $(\Sigma,\flip_\Sigma,\gamma_0)$ be a local type and $X_g$ be a sequence of surfaces of genus $g$ with $g\to\infty$. Then we have
    $$\lim_{g\to\infty}
    \frac{\FC_{g}(\phi_{\ns}(\gamma_0))}{\FC_{g}(\Sigma,\flip_\Sigma,\gamma_0)}=1$$
    where $\phi_{\ns}$ is a (the) non-separating realization of $(\Sigma,\flip_\Sigma,\gamma_0)$. More precisely, if $\phi_1=\phi_{\ns},\phi_2,\dots,\phi_{m_g}$ are all mapping class group orbits of realizations of $(\Sigma,\flip_\Sigma,\gamma_0)$ in $X_g$, then
    \[ \dfrac{\sum\limits_{i=2}^{m_g} \FC_g(\phi_i(\gamma_0))}{\FC_g(\phi_{\ns}(\gamma_0))} = O\left(\dfrac{1}{g}\right) .  \]
\end{sat}

We will rely on Theorem \ref{thm constant c in terms of intersection numbers} to prove Theorem \ref{sat dominant type}. The key is that, as we have seen previously, we understand how the Kontsevich polynomial behaves when the genus grows.

To make the computation easier to follow we will use the notations associated to the graph dual to a realization of a local type. Recall that we introduced in Section 6.4 the decorated graph dual to a realization of local type. 
Given a realization $\phi$ of a local type $(\Sigma,\gamma_0,\flip_\Sigma)$ into a surface $X$ of genus $g$, the set of vertices $\CV=\CV_\Sigma\sqcup \CV_Z$ corresponds to the connected components of $\Sigma_{\hyp}$ and $Z$, each vertex $v$ is decorated by the genus $g(v)$ of the associated surface and its degree $n(v)$ is the number of boundary components of this surface. Note that $2g(v) - 2 + n(v) > 0$ for all $v \in V$. Moreover, the set of edges is indexed by $\Gamma$ and its cardinality is independent from the realization or of the ambient surface $X$. 
Recall Corollary \ref{cor: cste in term of dual graph}:
\begin{equation*} 
    \mathfrak{c}_g(\phi(\gamma_0))
    =
    \frac{2^{\chi(Z)+|\CV_Z|}}{\sym(\gamma_0,\phi)}
    \bigintsss_{\Delta_{\Sigma}(\gamma_0)}
    \prod_{e \in E} b_e(\bar a)
    \prod_{v \in \CV_Z} V_{g(v), n(v)}(\bar{b}_v(\bar a))\ d\FM_{\BA}(\bar a)
\end{equation*}
where $\int_{\Delta_{\Sigma}(\gamma_0)}$ should be understood as a sum of integrals over arc systems on $\Sigma$,
and $\bar{b}_v$ and $b_e$ are defined in page \pageref{cor: cste in term of dual graph}. To make things easier notationwise we will tend to write $b_e$ for $b_e(\bar a)$ and the same for $\bar b_v$.
Recall that although $\prod_{e \in E} b_e(\bar a)$ looks to depend on the realization it does not, each edge $e$ corresponds to a unique curve $\eta\in\Gamma$ and then 
\[\prod\limits_{e \in E} b_e(\bar a)=\prod\limits_{\eta\in\Gamma}{\bf w}_\Sigma (\bar a)_\eta.\]
\begin{comment}
    Note that $|V|$ and $|E|$ can be bounded by some constants which depends only on the local type $\gamma$ (doesn't depend on $X$).
($|V|$ varies between $1$ (when $Z$ is connected) and $|\Gamma_\partial| + |\Gamma_{\ann}| + 1$ (when the graph is a tree). $E_0$ and $E_1$ are fixed.)
\end{comment}

\subsubsection{Non-separating realization} If $\phi_{\ns}$ is the non-separating realization  then $Z_{\ns}$ is connected, $\CV_{Z_{\ns}}$ contains one single element, denoted by $v_{\ns}$, and by Corollary \ref{kor useful aggarwal} we have
\begin{multline} \label{eq: asym non sep step 1}
    \mathfrak{c}_g(\phi_{\ns}(\gamma_0))
    =
    \frac{2^{\chi_{\ns}+1}}{\sym(\gamma_0,\phi_{\ns})}
    \int_{\Delta_{\Sigma}(\gamma_0)}
    (\prod_{e \in E_{\ns}} b_e)
    V_{g_{\ns}, n_{\ns}}(\bar{b}_{v_{\ns}})  \ d\FM_{\BA}(\bar a) \\
    \underset{g\to\infty}{\sim}
    \frac{2^{\chi_{\ns}+1}}{\sym(\gamma_0,\phi_{\ns})}
    \int_{\Delta_{\Sigma}(\gamma_0)}
    (\prod_{e \in E_{\ns}} b_e)
    \frac{(6g_{\ns}-5+2n_{\ns})!!}{g_{\ns}! 24^{g_{\ns}}}
    [t^{-3\chi_{\ns}}] \prod_{e \in \hal_{v_{\ns}}} \frac{\sinh(b_e t)}{b_e }  \ d\FM_{\BA}(\bar a) 
\end{multline}
where $g_{\ns} = g(v_{\ns})$, $n_{\ns} = n(v_{\ns})$, $\chi_{\ns} = -(2g_{\ns} - 2 + n_{\ns})=\chi(Z_{\ns})=\chi(X)-\chi(\Sigma)$, and $\hal_{v_{\ns}}$ is the set of half-edges adjacent to $v_{\ns}$.

Also, to make sure that the equations we will write afterward fit in one line let us define 
\begin{multline} \label{asymptotic non separating}
    \FC_\infty^{ns}(\Sigma,\flip_\Sigma,\gamma_0) =\\
    2^{\chi_{\ns}+1}
    \int_{\Delta_{\Sigma}(\gamma_0)}
    (\prod_{e \in E_{\ns}} b_e)
    \frac{(6g_{\ns}-5+2n_{\ns})!!}{g_{\ns}! 24^{g_{\ns}}}
    [t^{-3\chi_{\ns}}] \prod_{e \in \hal_{v_{\ns}}} \frac{\sinh(b_e t)}{b_e   }  \ d\FM_{\BA}(\bar a).
\end{multline}

\subsubsection{Separating realization} For any realization $\phi$ such that $Z$ is disconnected, it follows from \eqref{eq:Vleq} that
\begin{multline} \label{eq: one more}
    \mathfrak{c}_g(\phi(\gamma_0))
    \le \\
    \frac{2^{\chi(Z)+|\CV_Z|}}{\sym(\gamma_0,\phi)}
    \int_{\Delta_\Sigma(\gamma_0)}
    (\prod_{e \in E} b_e)
    \prod_{v \in \CV_Z}\left( \frac{3}{2} \right)^{n(v) - 1}
    \frac{(6g(v) - 5 + 2n(v))!!}{g(v)! 24^{g(v)}}
    [t^{-3\chi(v)}] \prod_{e \in \hal_v} \frac{\sinh(b_e t)}{b_e  }.
\end{multline}
The set $\Delta_\Sigma(\gamma_0)$ is independent from the realization, and for some fixed $\bar a \in \BA(\Sigma)$ what is dependent is
\begin{multline*}
         {2^{\chi(Z)+|\CV_Z|}}\prod_{v \in \CV_Z}  \left( \frac{3}{2} \right)^{n(v) - 1}
    \frac{(6g(v) - 5 + 2n(v))!!}{g(v)! 24^{g(v)}}
    [t^{-3\chi(v)}] \prod_{e \in \hal_v} \frac{\sinh(b_e t)}{b_e   } \\
    =
    \left( \frac{3}{2} \right)^{|\Gamma_{\fix}| + 2|\Gamma_{\ann}| - |\CV_Z|} 
    {2^{\chi(Z)+|\CV_Z|}}
    \prod_{v \in \CV_Z}
     \frac{(6g(v) - 5 + 2n(v))!!}{g(v)! 24^{g(v)}}
    [t^{-3\chi(v)}] \prod_{e \in \hal_v} \frac{\sinh(b_e t)}{b_e  }.
\end{multline*}
However, the Taylor expansion of $\sinh$ has only positive terms so
\begin{multline} \label{eq:bound sep nonsep}
    \prod\limits_{v\in \CV_Z} [t^{-3\chi(v)}]\prod\limits_{e\in\hal_v}  \frac{\sinh(b_e t)}{b_e  } \\
    \le [t^{-3\chi(Z)}] \prod\limits_{v\in \CV_Z} \prod\limits_{e\in\hal_v}  \frac{\sinh(b_e t)}{b_e  } 
    =  [t^{-3\chi_{\ns}}] \prod_{e \in \hal_{v_{\ns}}} \frac{\sinh(b_e t)}{b_e   }.
\end{multline}

What is left to study is $\prod_{v \in \CV_Z}
     \frac{(6g(v) - 5 + 2n(v))!!}{g(v)! 24^{g(v)}}$. This is the object of the next lemma.

\begin{lem} \label{lem:bound product}
    Let $(\Sigma,\flip_\Sigma,\gamma_0)$ be a local type, there is constant $C(\gamma_0)$ such that for any separating realization $\phi$ of the local type $(\Sigma,\flip_\Sigma,\gamma_0)$ in a surface $X$ 
    \[ \prod_{v \in \CV_Z} 
     \frac{(6g(v) - 5 + 2n(v))!!}{g(v)! 24^{g(v)}} 
     \le C(\gamma_0) 
     \frac{(6g_{\ns} - 5 + 2n_{\ns} )!!}{g_{\ns} ! 24^{g_{\ns} }}\cdot
     \dfrac{\prod_{v \in \CV_Z}|\chi(v)|!}{|\chi(v_{\ns})|!}.\]
\end{lem}

\begin{proof}
     The proof closely follow the one of \cite[Lemma~4.1]{Barazer-Giacchetto-Liu}. Fix a separating realization $\phi$ and start by entering the following identities:
     \begin{enumerate}
         \item $\sum\limits_{v\in \CV_Z} n(v) = 2|\Gamma_{\ann}|+|\Gamma_{\fix}|=n_{\ns}$, \label{sum degrees}
         \item $\sum\limits_{v\in \CV_Z}2g(v)-2+n(v)=|\chi(X)-\chi(\Sigma)|=2g_{\ns}-2+n_{\ns}$, \label{sum eurler char}
         \item $\sum\limits_{v\in \CV_Z} g(v) =g_{\ns}+(|\CV_Z|-1)$. \label{sum genus}
     \end{enumerate}
     As a consequence 
     \[
     \sum\limits_{v\in \CV_Z}3g(v)-3+n(v)=3g_{\ns}-3+n_{\ns}
     \]
     and 
     \[
     \sum\limits_{v\in \CV_Z}6g(v)-5+2n(v)=6g_{\ns}-5+2n_{\ns}+(|\CV_Z|-1).
     \]
     Also, observe that if $S$ is a surface of genus $g$ with $n$ boundary components then using the relation $(2k+1)!!=\frac{(2k+1)!}{2^k\cdot k!}$ we get
     \begin{equation}
         \label{eq developping !!}
         \frac{(6g - 5 + 2n)!!}{g!} = \left( \begin{array}{cc}
            6g - 5 + 2n \\
            3g - 3 + n, 2g-2+n, g
       \end{array} \right) \cdot \dfrac{|\chi(S)|!}{2^{3g - 3 + n}}.
     \end{equation}
    where $\left( \begin{array}{cc}
            k \\
            k_1,k_2,k_3
       \end{array} \right)=\frac{k!}{k_1!k_2!k_3!}$ is the multinomial coefficient for $k=k_1+k_2+k_3$.
     Hence, by \cite[Lemma 2.3]{Aggarwal} we have 
     \begin{align*}
          \prod_{v \in \CV_Z} &
     \frac{(6g(v) - 5 + 2n(v))!!}{g(v)! 24^{g(v)}}\\
      &\le
       \left( \begin{array}{cc}
            6g_{\ns}-5+2n_{\ns}+(|\CV_Z|-1) \\
            3g_{\ns}-3+n_{\ns},2g_{\ns}-2+n_{\ns},g_{\ns}+(|\CV_Z|-1)
       \end{array} \right) \cdot \dfrac{\prod_{v \in \CV_Z} |\chi(v)|!}{2^{3g_{\ns}-3+n_{\ns}}24^{g_{\ns}+(|\CV_Z|-1)}} \\
       & = \dfrac{(6g_{\ns}-5+2n_{\ns}+(|\CV_Z|-1))!}{(3g_{\ns}-3+n_{\ns})!(2g_{\ns}-2+n_{\ns})!(g_{\ns}+(|\CV_Z|-1))!}
       \cdot \dfrac{\prod_{v \in \CV_Z} |\chi(v)|!}{2^{3g_{\ns}-3+n_{\ns}}24^{g_{\ns}+(|\CV_Z|-1)}}.
     \end{align*}
     Using \eqref{eq developping !!} for $Z_\ns$ we also have
     \[   
     \frac{(6g_{\ns} - 5 + 2n_{\ns} )!!}{g_{\ns} ! 24^{g_{\ns} }}\cdot
     \dfrac{\prod_{v \in \CV_Z}|\chi(v)|!}{|\chi(v_{\ns})|!} = \dfrac{(6g_{\ns}-5+2n_{\ns})!}{(3g_{\ns}-3+n_{\ns})!(2g_{\ns}-2+n_{\ns})!g_\ns!} \frac{\prod_{v \in \CV_Z}|\chi(v)|!}{24^{g_\ns}2^{3g_\ns-3+n_\ns}}.
     \]

     Noting that $1\le|\CV_Z|\le |\Gamma|$ (in the worse case each boundary of $\Sigma$ is glued to a different component of $Z$ and all the components of $\phi(\gamma_{\ann})$ are separating $Z$), we have

\begin{align*}
    \dfrac{(6g_{\ns}-5+2n_{\ns}+(|\CV_Z|-1))!}{24^{|\CV_Z|-1}(g_{\ns}+(|\CV_Z|-1))!} &\le \dfrac{(6g_{\ns} - 5 + 2n_{\ns} )!}{g_\ns!}\cdot\dfrac{(6g_{\ns}-5+2n_{\ns}+(|\CV_Z|-1))^{|\CV_Z|-1}}{g_\ns^{|\CV_Z|-1}} \\
    &\le \dfrac{(6g_{\ns} - 5 + 2n_{\ns} )!}{g_\ns!}\cdot\left(6+\dfrac{4|\Gamma_{\ann}|+2|\Gamma_{\fix}|+|\Gamma|-6}{g_{\ns}}\right)^{|\Gamma|} \\
    &\le \dfrac{(6g_{\ns} - 5 + 2n_{\ns} )!}{g_\ns!}\cdot\left(6+(4|\Gamma_{\ann}|+2|\Gamma_{\fix}|+|\Gamma|-6)\right)^{|\Gamma|}.
\end{align*}

This concludes the proof

\end{proof}

Note that $|\CV_Z|\le |\Gamma|$, by Lemma \ref{lem:bound product}, \eqref{eq: one more} and \eqref{eq:bound sep nonsep} there is a constant $C'(\gamma_0)$ depending only on the local type such that

\begin{multline}\label{eq:sepVSnsep}
    \FC_g(\phi(\gamma_0))
    \le 
    \dfrac{C'(\gamma_0)}{\sym(\phi,\gamma_0)}
    \dfrac{\prod_{v \in \CV_Z} (-\chi_v)!}{(-\chi_{\ns})!} 
    \FC_\infty^{ns}(\Sigma,\flip_\Sigma,\gamma_0)
    \\
      \le
     2^{C(\gamma_0)}
    \dfrac{\prod_{v \in \CV_Z} (-\chi_v)!}{(-\chi_{\ns})!} 
     \FC_\infty^{ns}(\Sigma,\flip_\Sigma,\gamma_0).
\end{multline}

We now dispose of the necessary content to prove Theorem \ref{sat dominant type}.

\subsubsection{Proof of Theorem \ref{sat dominant type}}
Recall that we have written
$$\FC_g(\Sigma,\flip_\Sigma,\gamma_0)=\sum_{i=1,\dots,m_g}\FC_g(\phi_i(\gamma_0)).$$
Assume that $\phi_{\ns}=\phi_1$. Regarding equation (\ref{eq:sepVSnsep}) what remains to do is to study 

\[
\sum\limits_{1=2}^{m_g}
\dfrac{\prod_{v \in \CV_{Z_i}} (-\chi_v)!}{(-\chi_{\ns})!}.
\]

Let $k$ be an integer between $2$ (minimal size of $\CV_Z$ for a separating realization) and $|\Gamma_{\fix}|+|\Gamma_{\ann}|$ (maximal size of $\CV_Z$ for a separating realization). There is only finitely many graphs with $k+|\pi_0(\Sigma_{\hyp})|$ vertices and $|\Gamma|$ edges. Fix such a graph and a closed surface of large enough genus $X$. We want to decorate the graph it in order to make it a possible decorated graph dual to a realization of $(\Sigma,\flip_\Sigma,\gamma_0)$ in $X$. First of all,  $|\pi_0(\Sigma_{\hyp})|$ of the vertices should be decorated with the genus of the connected components of $\Sigma$ (finitely many choices for this with a uniform upper bound depending only on the local type) the remaining ones should be decorated such that $\sum\limits_{v\in \CV_Z}2g(v)-2+n(v)=|\chi(X)|-|\chi(\Sigma)|=|\chi_{\ns}|$. We defined our dual graph as decorated by genus but since $n(v)$ is fixed by the graph we can think the decoration to be by the Euler characteristic. Hence, finding a decoration for the vertices in $\CV_Z$ is the same as finding a partition of $|\chi_{\ns}|$ by $c_1,\cdots,c_k\in\BZ_{\ge 1}$.
Hence, by Proposition \ref{prop:NumberOfRealizationForAGraph} there is a constant $B(\gamma_0)$, indenpendant from $g$ such that
\[
\sum\limits_{1=2}^{m_g}
\dfrac{\prod_{v \in \CV_{Z_i}} (-\chi_v)!}{(-\chi_{\ns})!} 
\le 
B(\gamma_0)
\sum\limits_{k=2}^{|\Gamma_\fix|+|\Gamma_{\ann}|} 
\sum_{\substack{(c_1, \dots, c_k) \in \mathbb{Z}_{\geq 1}^k \\ n_1 + \cdots + n_k = |\chi_{\ns}|}}
\frac{c_1! \cdots c_k!}{|\chi_{\ns}|!}.
\]

Now we can use the following lemma to conclude:
\begin{lem}[{\cite[Lemma~4.3]{Barazer-Giacchetto-Liu}}]
    Let $2 \le k \le n$ be integers.
    The following bound holds:
    \[
        \sum_{\substack{(n_1, \dots, n_k) \in \mathbb{Z}_{\geq 1}^k \\ n_1 + \cdots + n_k = n}} \frac{n_1! \cdots n_k!}{n!}
        \le
        \frac{4}{n}.
    \]
    \hfill $\blacksquare$
\end{lem}
Together with equation (\ref{eq:sepVSnsep}) it implies that 
\begin{equation} \label{eq:FirstBoundNonSep}
    \sum\limits_{i=2}^{m_g} 
\FC_g(\phi_i(\gamma_0))
\le 
C'(\gamma_0)B(\gamma_0)
(|\Gamma_{\fix}|+|\Gamma_{\ann}|)
\dfrac{\FC_\infty^{ns}(\Sigma,\flip_\Sigma,\gamma_0)}{|\chi_{\ns}|}.
\end{equation}

Moreover, recall \eqref{eq: asym non sep step 1}, that is $\FC_g(\phi_{\ns}(\gamma_0))\underset{g\to\infty}{\sim} \dfrac{1}{\sym(\gamma_0,\phi_{\ns})}\FC_\infty^{ns}(\Sigma,\flip_\Sigma,\gamma_0)$ and $|\chi_\ns|\sim 6g$. The quantity $\sym(\phi_\ns,\gamma_0)$ being independent from $g$ by Proposition \ref{prop: sym independant of g}, it follows that
\[
    \sum\limits_{i=2}^{m_g} 
    \FC_g(\phi_i(\gamma_0))
    =
    \mathrm{O}\left( \frac{\FC_g(\mathrm{\phi_{\ns}(\gamma_0)})}{g} \right),
\]
which concludes the proof of Theorem \ref{sat dominant type}. \hfill \qed

\begin{rmk*} All simple curves are of same local type (an annulus with its core curve). In that case Delecroix-Goujard-Zograf-Zorich proved that the decay in Theorem \ref{sat dominant type} is exponential. As we just proved, for general local types we obtain a polynomial decay, we do not know if it is optimal or not.
\end{rmk*}

\subsection{From asymptotic to singular analysis question} \label{sec: singular analysis}

From now on we work only with essential local types, this is because we will prove later on that they are dominant in comparison with non-essential ones. In light of Theorem~\ref{sat dominant type}, we henceforth restrict our attention to non-separating realizations of essential local types.

Let $\phi_{\ns}$ be the non-separating realization of an essential local type $(\Sigma, \flip_\Sigma, \gamma_0)$ into a closed surface $X$. Since the involution $\flip_\Sigma$ does act only on annuli we will, in this section, mostly drop it from our notation. By Theorem~\ref{thm constant c in terms of intersection numbers}, 
\[
    \FC_g(\phi_{\ns}(\gamma_0))
    =
    \frac{2^{\chi(Z) + 1}}{\sym( \phi_{\ns},\gamma_0)}
    \int_{\Delta_\Sigma(\gamma_0)}
    \product({\bf w}_\Sigma(\bar{x})) \cdot V_Z({\bf b}_{|\D Z}(\bar x)) \, d \FM_{\BA(\Sigma, \flip_\Sigma)}(\bar{x})
\]

and since we work with the non-separating realization we have from \eqref{eq: asym non sep step 1} and Proposition~\ref{prop: sym independant of g}, 
\begin{equation} \label{asymp step 1}
    \FC_g(\phi_{\ns}(\gamma_0))
    \underset{g\to\infty}{\sim}
    \frac{1}{|\Stab_{\Map(\Sigma)}(\gamma_0)|}
    \FC_\infty^\ns(\Sigma,\gamma_0)
\end{equation}   
where, using the dual graph notations, 

\begin{multline}  \label{eq: asymp non sep step 2}
    \FC_\infty^{ns}(\Sigma,\gamma_0) \\
    =2^{\chi_{\ns}+1}
    \int_{\Delta_{\Sigma}(\gamma_0)}
    (\prod_{e \in E_{\ns}} b_e)
    \frac{(6g_{\ns}-5+2n_{\ns})!!}{g_{\ns}! 24^{g_{\ns}}}
    [t^{-3\chi_{\ns}}] \prod_{e \in \hal_{v_{\ns}}} \frac{\sinh(b_e t)}{b_e   }  \ d\FM_{\BA}(\bar a)\\
    =2^{\chi_{\ns}+1}
    \frac{(6g_\ns - 5 + 2n_\ns)!!}{g_\ns! 24^{g_\ns}} 
    \cdot [t^{-3\chi_\ns}]
    \int_{\Delta_\Sigma(\gamma_0)}
    \prod_{e \in E_\fix(\bar{x})}
    \sinh(b_e t)
    \prod_{e \in E_\ann(\bar{x})}
    \frac{\sinh(b_et)^2}{b_e}
    \, \ d\FM_{\BA}(\bar a).
\end{multline}

Because the local type is essential, for each $\sigma \in \CA(\Sigma)$, $\BA(\sigma)$ is a full-dimensional subspace of $\BR_{\geq 0}^{3 |\chi(\Sigma)|}$ \eqref{eq dim no flip},
and the measure $\FM_{\BA(\Sigma, \flip_\Sigma)}$ is the Lebesgue measure on $\BA(\sigma)$ \eqref{eq: measure=lebesgue}. Hence, we can write

\begin{equation} \label{asym step 1 bis}
    \FC_\infty^{ns}(\Sigma,\gamma_0) =2^{\chi_{\ns}+1}
    \frac{(6g_\ns - 5 + 2n_\ns)!!}{g_\ns! 24^{g_\ns}} \sum\limits_{\sigma\in\Sigma}\FC_\sigma(\gamma_0) 
\end{equation}
with

\[
    \FC_{\sigma}(\phi_{\ns}(\gamma_0))
    =
     [t^{-3\chi_\ns}]
    \int_{\Delta_\sigma(\gamma_0)}
    \prod_{e \in E_\fix(\bar{x})}
    \sinh(b_e t)
    \prod_{e \in E_\ann(\bar{x})}
    \frac{\sinh(b_et)^2}{b_e}
    \, d\bar x.
\] 

Hence, as the genus of the complement $Z$ goes to infinity, the portion $\FC_\sigma(\phi_\ns(\gamma_0))$ of the curve frequency $\FC(\phi_\ns(\gamma_0))$, arising from a maximal arc system $\sigma$, is asymptotically equivalent to a quantity that can be written as a coefficient of a generating function.
In what follows, we derive the asymptotics of the coefficients of this generating function.

We state and prove several results in a general setting to simplify notation.
It might be helpful to keep in mind that, in the applications, $r$ will represent the number of arcs in a maximal arc system on $\Sigma_\hyp$, $n = |\D\Sigma_\hyp|=|E_\fix|$ the number of boundary components of $\Sigma_\hyp$, $n' = |\pi_0(\Sigma_\ann)|=|E_\ann|$ the number of annular components of $\Sigma$, $\iota_i$ the intersection number between $\gamma_0$ and the $i$-th arc in a fixed maximal arc system on $\Sigma_\hyp$, and $(c_{i,j})_{i,j}$ the matrix that describe the incidence relations between the arcs and the boundary components of $\Sigma_\hyp$ (namely, $c_{i,j} = 1$ if the $i$-th arc is incident to the $j$-th boundary component of $\Sigma_\hyp$, and $= 0$ otherwise).

\subsubsection{A transform}

We begin by massaging the generating function to prepare for the singularity analysis.
Let $n \in \mathbb{Z}_{\geq 1}$, $n' \in \mathbb{Z}_{\geq 0}$, $r \in \mathbb{Z}_{\geq 1}$,
and let $A$ be a countable (infinite) set.
For each $\alpha \in A$, let
\begin{itemize}
    \item
        $\iota(\alpha) = (\iota_1(\alpha), \dots, \iota_r(\alpha)) \in \mathbb{Z}_{\geq 1}^r$,

    \item
        $\iota_0(\alpha) = \min_i \iota_i(\alpha)$,

    \item
        $\mu_0(\alpha) = |\{ i : \iota_i = \iota_0 \}|$,

    \item
        $(c_{i,j}(\alpha))_{1 \leq i \leq r, 1 \leq j \leq n}$ be a matrix such that
        \begin{itemize}
            \item
                $c_{i,j}(\alpha) \in \{ 0, 1, 2 \}$, for all $i,j$,

            \item
                $\sum_{j} c_{i,j}(\alpha) = 2$, for each $1 \leq i \leq r$,

            \item
                $\sum_{i} c_{i,j}(\alpha) > 0$, for each $1 \leq j \leq n$.
        \end{itemize}
\end{itemize}

Define the generating function
\begin{equation} \label{eq:varphi}
    \varphi_\alpha(z)
    =
    \int_{\Delta_{\bar{\iota}}}
    dx_1 \cdots dx_r \, dx_1' \cdots dx_{n'}'
    \prod_{j=1}^{n} \sinh(b_j z)
    \cdot
    \prod_{k=1}^{n'} \frac{\sinh(b_k' z)^2}{b_k'},
\end{equation}
where $\Delta_{\bar{\iota}} = \{ (x_1, \dots, x_r, x_1', \dots, x_{n'}') \in \mathbb{R}_{\geq 0}^{r + n'} : \iota_1(\alpha) x_1 + \cdots \iota_r(\alpha) x_r + x_1' + \cdots + x_{n'}' \leq 1 \}$,
$b_j = \sum_{i=1}^{r} c_{i,j}(\alpha) x_i$ for $1 \leq j \leq n$, and $b_k' = x_j'$ for $1 \leq k \leq n'$.
We are interested in the asymptotic behavior of the coefficient
\[
    [z^N] \, \varphi_\alpha(z)
\]
as $N \to \infty$. To this end, we introduce the following transform on power series.
For $d \in \mathbb{Z}$, let $\mathcal{B}_d : \mathbb{C} \llbracket z \rrbracket \to \mathbb{C} \llbracket z \rrbracket$ be given by
\[
    \mathcal{B}_d \sum_{i = 0}^{\infty} a_i z^i
    =
    \sum_{i = 0}^{\infty} (i+d)! \, a_i z^{i}.
\]
For $g(z) = \sum_{i \geq 0} a_i z^i$, we have
\begin{equation} \label{eq:B}
    [z^N] \, (\mathcal{B}_d \, g)(z)
    =
    (N+d)! \, [z^N] \, g(z).
\end{equation}
Hence understanding $[z^N] \, \varphi_\alpha(z)$ is equivalent to understanding $[z^N] \, (\mathcal{B}_d \, \varphi_\alpha)(z)$. The next result explains why we introduce the transform $\mathcal{B}_{d}$.

\begin{lem} \label{lem:K}

    For $z < 1/2$, we have
    \begin{multline} \label{eq:Bvarphi}
        (\mathcal{B}_r \, \varphi_\alpha)(z) \\
        =
        \frac{1}{2^{n + 2n'}}
        \bigg(
        \log \frac{1}{1 - (2z)^2}
        \bigg)^{n'}
        \sum_{\bar{\epsilon} \in \{ -1, 1 \}^{n}}
        \epsilon_1 \cdots \epsilon_n
        \prod_{i=1}^{r}
        \frac{1}{\iota_i - (\epsilon_1 c_{i, 1} + \cdots + \epsilon_n c_{i, n}) z}.
    \end{multline}
\end{lem}
\begin{proof}
    We have
    \[
        \prod_{j=1}^{n}
        \sinh(b_j z)
        \cdot
        \prod_{k=1}^{n'}
        \frac{\sinh(b_k' z)^2}{b_k'}
        =
        \frac{1}{2^{n + 2n'}}
        \sum_{\bar{\epsilon} \in \{ -1, 1 \}^{n}, \bar{\epsilon}' \in \{ -1, 1 \}^{2n'}}
        \frac{\product(\bar{\epsilon}) \product(\bar{\epsilon}')}{\product(\bar{b}')}
        \mathrm{e}^{\langle \bar{\epsilon}, \bar{b} \rangle z + \langle \bar{\epsilon}', (\bar{b}, \bar{b}) \rangle z}
    \]
    where $\langle \bar{\epsilon}, \bar{b} \rangle = \epsilon_1 b_1 + \cdots + \epsilon_n b_n$.
    Hence, it follows from the fact
    \[
        \mathcal{B}_d (\mathrm{e}^z)
        =
        \frac{d!}{(1-z)^{d+1}},
        \qquad
        \forall d \in \mathbb{Z}_{\geq 0}
    \]
    that 
    \[
        \mathcal{B}_{r} \varphi_\alpha(z)
        =
        \frac{r!}{2^{n + 2n'}}
        \sum_{\bar{\epsilon} \in \{ -1, 1 \}^{n}, \bar{\epsilon}' \in \{ -1, 1 \}^{2n'}}
        \int_{\Delta_{\bar{\iota}}} 
        \frac{\product(\bar{\epsilon}) \product(\bar{\epsilon}')}{\product(\bar{b}')}
        \frac{dx_1 \cdots dx_r \, dx_1' \cdots dx_{n'}'}{(1 - \langle \bar{\epsilon}, \bar{b} \rangle z - \langle \bar{\epsilon}', (\bar{b}', \bar{b}') \rangle z )^{r + 1}}.
    \]
    Thus, it is sufficient to prove that for any $\bar{\epsilon} \in \{ -1, 1 \}^n$, we have
    \[
        \sum_{\bar{\epsilon}' \in \{ -1, 1 \}^{2n'}}
        \int_{\Delta_{\bar{\iota}}}
        \frac{d \bar{x} \, d \bar{x}' \, \product(\bar{b}')^{-1} \, \product(\bar{\epsilon}')}{(1 - \langle \bar{\epsilon}, \bar{b} \rangle z - \langle \bar{\epsilon'}, (\bar{b}', \bar{b}') \rangle z)^{r + 1}}
        =
        \frac{1}{r!}
        \bigg(
        \prod_{i=1}^{r} \frac{1}{\iota_i - \langle \bar{\epsilon}, \bar{c_i} \rangle z}
        \bigg)
        \bigg(
        \log \frac{1}{1 - (2z)^2}
        \bigg)^{n'}
    \]
    where $\bar{c}_i = (c_{i,1}, \dots, c_{i,n})$.
    It turns out to be easier to prove a more general formula:
    \begin{multline} \label{eq:K}
        \sum_{\bar{\epsilon}' \in \{ -1, 1 \}^{2n'}}
        \int_{\Delta_{\bar{\iota}}^t}
        \frac{d \bar{x} \, d \bar{x}' \, \product(\bar{b}')^{-1} \,  \product(\bar{\epsilon}')}{(\sigma - \langle \bar{\epsilon}, \bar{b} \rangle z - \langle \bar{\epsilon}', (\bar{b}', \bar{b}') \rangle z)^{r + 1}} \\
        =
        \frac{1}{r!} \frac{t^r}{\sigma}
        \bigg(
        \prod_{i=1}^{r} \frac{1}{\iota_i \sigma - \langle \bar{\epsilon}, \bar{c_i} \rangle t z}
        \bigg)
        \bigg(
        \log \frac{1}{1 - (2zt)^2}
        \bigg)^{n'}
    \end{multline}
    where $t > 0$, $\sigma \geq 1$, and
    \[
        \Delta_{\bar{\iota}}^t = \{ (x_1, \dots, x_r, x_1', \dots, x_{n'}') \in \mathbb{R}_{\geq 0}^{r + n'} : \iota_1 x_1 + \cdots \iota_r x_r + x_1' + \cdots + x_{n'}' \leq t \}.
    \]
    To this end, we will prove that the Laplace transforms of both sides of \eqref{eq:K} coincide.
    Write $I(t)$ for the left-hand side of \eqref{eq:K}.
    Let us compute the Laplace transform of $I(t)$:

    \[
        \mathcal{L}\{ I \} (s)
        =
        \sum_{\bar{\epsilon}' \in \{ -1, 1 \}^{2n'}}
        \int_0^\infty dt \mathrm{e}^{-st}
        \int_{\Delta_{\bar{\iota}}^{t}}
        \frac{d\bar{x} \, d\bar{x}' \, \product(\bar{b}')^{-1} \,  \product(\bar{\epsilon}')}{(\sigma - \langle \bar{\epsilon}, \bar{b} \rangle z - \langle \bar{\epsilon}', (\bar{b}', \bar{b}') \rangle z )^{r + 1}}.
    \]
    The rightmost integrand can be expressed as the Laplace transform of a function:
    from the fact
    \[
        \mathcal{L} \{ \tau \to \tau^r \mathrm{e}^{- \alpha \tau} \} (\sigma)
        =
        \frac{r!}{(\sigma + \alpha)^{r + 1}}
    \]
    where $\mathrm{Re}(\sigma) > - \alpha$, it follows that for $\mathrm{Re}(z) < 1/2$ 

    \[
        \frac{1}{(\sigma - \langle \bar{\epsilon}, \bar{b} \rangle z - \langle \bar{\epsilon}', (\bar{b}', \bar{b}') \rangle z )^{r + 1}}
        =
        \frac{1}{r!}
        \int_0^\infty d\tau
        \,
        \mathrm{e}^{-\sigma \tau} \tau^r \,
        \mathrm{e}^{(\langle \bar{\epsilon}, \bar{b} \rangle z + \langle \bar{\epsilon}', (\bar{b}', \bar{b}') \rangle z) \tau}.
    \]
    After switching the order of integration, we have
    \[
        \mathcal{L} \{ I \} (s)
        =
        \sum_{\bar{\epsilon}'}
        \frac{1}{r!}
        \int_0^\infty d\tau \,
        \mathrm{e}^{-\sigma \tau} \tau^r
        \int_0^\infty dt \,
        \mathrm{e}^{-st}
        \int_{\Delta_{\bar{\iota}}^t} d\bar{x} \, d\bar{x}' \,
        \frac{\product(\bar{\epsilon}')}{\product(\bar{b}')} \,
        \mathrm{e}^{(\langle \bar{\epsilon}, \bar{b} \rangle z + \langle \bar{\epsilon}', (\bar{b}', \bar{b}') \rangle z) \tau}.
    \]
    For the rightmost integral, we have (under the change of variables $x_i \to x_i/\iota_i$)
    \begin{multline} \label{eq:convo}
        \sum_{\bar{\epsilon}'}
        \int_{\Delta_{\bar{\iota}}^t} d\bar{x} \, d\bar{x}' \,
        \frac{\product(\bar{\epsilon}')}{\product(\bar{b}')} \,
        \mathrm{e}^{(\langle \bar{\epsilon}, \bar{b} \rangle z + \langle \bar{\epsilon}', (\bar{b}', \bar{b}') \rangle z) \tau} \\
        =
        \int_{\Delta_{\bar{1}}^{t}}
        \prod_{i=1}^{r}
        \frac{dx_i}{\iota_i} \,
        \mathrm{e}^{\langle \bar{\epsilon}, \bar{c}_{i} \rangle z \tau x_i / \iota_i}
        \cdot
        \prod_{k=1}^{n'}
        dx_k' \frac{4 \sinh(x_k' z \tau)^2}{x_k'}
    \end{multline}
    where $\Delta_{\bar{1}}^t = \{ (x_1, \dots, x_r, x_1', \dots, x_{n'}') \in \mathbb{R}_{\geq 0}^{r + n'} : x_1 + \cdots x_r + x_1' + \cdots + x_{n'}' \leq t \}$.
    Note that the right-hand side of \eqref{eq:convo} is the convolution of $r$ functions of the form
    \[
        x
        \longmapsto
        \frac{\mathrm{e}^{\langle \bar{\epsilon}, \bar{c}_{i} \rangle z \tau x / \iota_i}}{\iota_i},
        \qquad
        1 \leq i \leq r,
    \]
    $n'$ functions of the form $x \mapsto 4 \sinh(x z \tau)^2 / x$,
    and the constant function $x \mapsto 1$;
    all functions are defined on $\mathbb{R}_{> 0}$.
    By the property that the Laplace transform of the convolution of several functions is the product of the Laplace transforms of the individual functions, we obtain
    \[
        \mathcal{L} \{ I \} (s)
        =
        \frac{1}{r!}
        \int_0^\infty
        d\tau \,
        \mathrm{e}^{-\sigma \tau}
        \tau^r
        \bigg(
        \prod_{i=1}^{r}
        \frac{1}{\iota_i s - \langle \bar{\epsilon}, \bar{c_i} \rangle z \tau}
        \bigg)
        \bigg(
        \log \frac{1}{1 - (2 z \tau / s)^2}
        \bigg)^{n'}
        \frac{1}{s}
    \]
    where we use the fact
    \[
        \mathcal{L} \{ t \mapsto \sinh(\alpha t)^2 / t \} (s)
        =
        \frac{1}{4} \log \frac{1}{1 - (2 \alpha / s)^2}.
    \]
    Now by performing the change of variables $\tau \to st /\sigma$, we obtain
    \[
        \mathcal{L} \{ I \} (s)
        =
        \frac{1}{r!}
        \int_0^\infty
        dt \,
        \mathrm{e}^{-st}
        \,
        t^r
        \frac{1}{\sigma}
        \bigg(
        \prod_{i=1}^{r}
        \frac{1}{\iota_i \sigma - \langle \bar{\epsilon}, \bar{c}_i \rangle z t}
        \bigg)
        \bigg(
        \log
        \frac{1}{1 - (2 z t)^2}
        \bigg)^{n'}
    \]
    as claimed.
    This completes the proof.
\end{proof}

\subsubsection{Singularity analysis}\label{subsec singularity analysis}
If a function $f$ can be written as
\[
    f(z)
    =
    \sum_{i = -\infty}^{\infty} a_i (z - z_0)^i,
\]
we will write
\[
    [(z - z_0)^n] \, f(z)
    \coloneqq
    a_n.
\]
As we will see in the next section,
the asymptotic behavior of the coefficients $[z^N] \, \varphi_\alpha(z)$ can be estimated by analysing the singularities of $\mathcal{B}_r \varphi_\alpha$,
and the singularities closer to the origin contribute more significantly in the asymptotics.
For this reason, in the present section we will focus on the singularities of minimal modulus.

To isolate the part of $\mathcal{B}_r \varphi_\alpha$ that depends on $\alpha$, we define the function $\phi_\alpha$ by
\[
    (\mathcal{B}_r \, \varphi_\alpha)(z)
    =
    \frac{1}{2^{n + 2n'}}
    \bigg(
        \log \frac{1}{1 - (2z)^2}
    \bigg)^{n'}
    \phi_\alpha(z),
\]
or equivalently,
\begin{equation} \label{eq:phiAlpha}
    \phi_\alpha(z)
    \coloneqq
    \sum_{\bar\epsilon \in \{-1,1 \}^n}
    \product(\bar \epsilon)
    \prod_{i=1}^{r}
    \frac{1}{\iota_i - \langle \bar{\epsilon}, \bar{c}_i \rangle z}.
\end{equation}
This is a rational function whose poles lie in the set $\{ \pm \iota_i / 2 : 1 \leq i \leq r \}$.
The minimal possible modulus of the poles is $(\min_i \iota_i)/2$.
However, the presence of the factor $\epsilon_1 \cdots \epsilon_n$ introduce possible cancellation,
so it is not evident that such poles actually survive after the sum.
The next result shows that they do survive.

\begin{lem} \label{lem:Linxiao}
    For any $\alpha$, we have
    $[(1 - 2z/\iota_0(\alpha))^{-\mu_0(\alpha)}] \, \phi_{\alpha}(z) > 0$. 
\end{lem}

\begin{proof}
    Observe that for any $\epsilon_1, \epsilon_2 \in \{ -1, 1 \}$ and any $\iota$, we have
    \begin{equation} \label{eq:lin}
        \frac{1}{\iota - (\epsilon_1 + \epsilon_2) z}
        =
        \frac{\iota^2 - 2z^2}{\iota (\iota^2 - 4z^2)}
        +
        \frac{z}{\iota^2 - 4z^2} (\epsilon_1 + \epsilon_2)
        +
        \frac{2z^2}{\iota(\iota^2 - 4z^2)} \, \epsilon_1 \epsilon_2.
    \end{equation}
    Let
    \begin{align} \label{eq:h}
        \begin{split}
        h_{00}(\iota, z)
        & =
        \frac{\iota^2 - 2z^2}{\iota (\iota^2 - 4z^2)}
        =
        \frac{1}{4\iota(1 - 2z/\iota)}
        +
        \frac{1}{4\iota(1 + 2z/\iota)}
        +
        \frac{1}{2\iota}, \\
        %-
        %\frac{3\iota + 4z}{4\iota (\iota + 2z)}, \\
        h_{01}(\iota, z)
        =
        h_{10}(\iota, z)
        & =
        \frac{z}{\iota^2 - 4z^2}
        =
        \frac{1}{4\iota(1 - 2z/\iota)}
        -
        \frac{1}{4\iota(1 + 2z/\iota)}, \\
        h_{11}(\iota, z)
        & =
        \frac{2z^2}{\iota(\iota^2 - 4z^2)}
        =
        \frac{1}{4\iota(1 - 2z/\iota)}
        +
        \frac{1}{4\iota(1 + 2z/\iota)}
        -
        \frac{1}{2\iota}.
        %\frac{1}{4(\iota - 2z)} - \frac{\iota + 4z}{\iota(\iota + 2z)}
        \end{split}
    \end{align}
    Then we may rewrite \eqref{eq:lin} as
    \[
        \frac{1}{\iota - (\epsilon_1 + \epsilon_2) z}
        =
        \sum_{\delta_1, \delta_2 \in \{ 0, 1 \}} \epsilon_1^{\delta_1} \epsilon_2^{\delta_2} h_{\delta_1 \delta_2}(\iota, z).
    \]
    Since $c_{i,j}(\alpha) \in \{ 0, 1, 2 \}$ for all $i,j$, and $\sum_j c_{i,j}(\alpha) = 2$ for every $1 \leq i \leq r$,
    we have, for each $1 \leq i \leq r$, $\epsilon_1 c_{i, 1}(\alpha) + \cdots + \epsilon_n c_{i, n}(\alpha) = \epsilon_{j_{i, 1}} + \epsilon_{j_{i, 2}}$ for some $1 \leq j_{i, 1} \leq j_{i, 2} \leq n$.
    Hence,
    \[
        \phi_\alpha(z)
        =
        \sum_{\bar \epsilon \in \{-1,1 \}^n }
        \product(\bar \epsilon)
        \sum_{\bar\delta_{1}, \bar\delta_{ 2} \in \{ 0,1 \}^r}
        \prod_{i=1}^{r}
        \epsilon_{j_{i, 1}}^{\delta_{i, 1}} \epsilon_{j_{i, 2}}^{\delta_{i, 2}}
        h_{\delta_{i, 1} \delta_{i, 2}}(\iota_i(\alpha), z).
    \]
    Exchanging the order of summations gives
    \begin{equation} \label{eq:Linxiao}
        \phi_\alpha(z)
        =
        \sum_{\bar\delta_{1}, \bar\delta_{ 2} \in \{ 0,1 \}^r}
        \Bigg(
            \prod_{i=1}^{r}
            h_{\delta_{i, 1} \delta_{i, 2}}(\iota_i(\alpha), z)
        \Bigg)
        \sum_{\bar \epsilon \in \{-1,1 \}^n}
        \product(\bar \epsilon)
        \prod_{i=1}^{r}
        \epsilon_{j_{i, 1}}^{\delta_{i, 1}} \epsilon_{j_{i, 2}}^{\delta_{i, 2}}.
    \end{equation}
    The inner sum
    \begin{equation} \label{eq:delta}
        \sum_{\bar \epsilon \in \{-1,1 \}^n}
        \product(\bar \epsilon)
        \prod_{i=1}^{r}
        \epsilon_{j_{i, 1}}^{\delta_{i, 1}} \epsilon_{j_{i, 2}}^{\delta_{i, 2}}
    \end{equation}
    equals $2^n$ when every $\epsilon_\bullet$ appears with an even total exponent, and equals $0$ otherwise.
    Since $\sum_{i} c_{i,j}(\alpha) > 0$ for each $1 \leq j \leq n$, there exists at least one choice of $(\delta_{i, 1}, \delta_{i, 2})_{i=1}^{r}$ for which \eqref{eq:delta} does not vanish.
    The lemma follows from the fact that for all $(i,j) \in \{ 0,1 \}^2$, the coefficients of $1 / (1 - 2z/\iota)$ in $h_{ij}$ are positive.
\end{proof}

Since $\phi_\alpha$ is even when $n$ is even and odd when $n$ is odd, we have
\begin{equation}
    [(\iota_0(\alpha)/2 + z)^{-\mu_0(\alpha)}] \, \phi_{\alpha}(z)
    =
    (-1)^n
    [(\iota_0(\alpha)/2 - z)^{-\mu_0(\alpha)}] \, \phi_{\alpha}(z)
    \neq
    0.
\end{equation}

At this point we would basically be done if we were only interested in a single maximal arc system.
However, to obtain the asymptotics of the curve frequency $\FC(\phi_\ns(\gamma_0))$, we must consider infinitely many arc systems.
We address this issue using the finiteness result established in Section \ref{finiteness}.

Recall that $A$ is a countable (infinite) set consisting of $\alpha$.
Let
\begin{itemize}
    \item
    $I = I(A) = \{ \iota_i(\alpha) : \alpha \in A, 1 \leq i \leq r \}$,

    \item
    $\iota = \iota_0(A) = \min I$,

    \item
    $\mu_0 = \mu_0(A) = \max_\alpha |\{ i : \iota_i(\alpha) = \iota_0 \}|$,
\end{itemize}
and write
\begin{equation} \label{eq:phi}
    \phi(z)
    \coloneqq
    \sum_{\alpha \in A} \phi_{\alpha}(z)
\end{equation}
where $\phi_\alpha$ is defined by \eqref{eq:phiAlpha}.

Each simplex $\Delta_{\bar{\iota}(\alpha)}$ appearing in the definition \eqref{eq:varphi} of $\varphi_\alpha$ has volume
\[
    \mathrm{vol} \, \Delta_{\bar{\iota}(\alpha)}
    =
    \int_{\Delta_{\bar{\iota}(\alpha)}} d\bar{x} \, d\bar{x}'
    =
    \frac{1}{\iota_1(\alpha) \cdots \iota_r(\alpha)}.
\]
We will establish certain analytic property of $\phi$ under the following assumption
\begin{equation} \label{eq:V<oo}
    \sum_{\alpha \in A}
    \mathrm{vol} \, \Delta_{\bar{\iota}(\alpha)}
    =
    V
    <
    \infty
\end{equation}
which secretly corresponds to Theorem~\ref{sat finite volume}.

The goal is to show that $\phi_\alpha$'s singularities sum neatly to $\phi$'s singularities. 
In particular, the principal singularities of $\phi$, those closest to the origin, are located at $z = \pm \iota_0/2$, each with multiplicity $\mu_0$. We start with a lemma estimating the coefficients of the Laurent series of $\phi_\alpha(z)$ at $z = \pm \iota_0 / 2$.
\begin{lem} \label{lem:Laurent}
    For all $1 \leq m \leq \mu_0$, we have
    \[
        \left| [(1 \pm 2z/\iota_0)^{-m}] \, \phi_\alpha(z) \right|
        \leq
        C_{1} \prod_{i=1}^{r} \frac{1}{\iota_i(\alpha)},
    \]
    where
    \[
        C_{1}
        \coloneqq
        2^n
        \binom{2r-r}{r-1} 
        \big( \iota_0^{\mu_0} (\iota_0 + 1) \big)^r.
    \]
\end{lem}
\begin{proof}
    We prove ``$-$'' case; the ``$+$'' case follows then from the parity of $\phi_{\alpha}$, which is even when $n$ is even and odd when $n$ is odd.
    Since $c_{i,j}(\alpha) \in \{ 0,1,2 \}$ and $c_{i,1}(\alpha) + \cdots + c_{i,n}(\alpha) = 2$ for all $\alpha, i,j$,
    every factor $(\iota_i(\alpha) - \langle \bar{\epsilon}, \bar{c}_i(\alpha) \rangle z)^{-1}$ in \eqref{eq:phiAlpha} is of the form $(\iota \pm 2z)^{-1}$ or $\iota^{-1}$.
    When $\iota > \iota_0$, we can write
    \begin{align*}
        \frac{1}{\iota - 2z}
        & =
        \frac{1}{\iota - \iota_0}
        \sum_{k=0}^{\infty}
        \left( -\frac{\iota_0}{\iota - \iota_0} (1 - 2z/\iota_0) \right)^k, \\
        \frac{1}{\iota + 2z}
        & =
        \frac{1}{\iota + \iota_0}
        \sum_{k=0}^{\infty}
        \left( \frac{\iota_0}{\iota + \iota_0} (1 - 2z/\iota_0) \right)^k.
    \end{align*}
        Thus when $\iota > \iota_0$, for $k \in \mathbb{Z}_{\geq 0}$, we have 
    \begin{align*}
        \left| [(1 - 2z/\iota_0)^k] \, \frac{1}{\iota - 2z} \right|
        & =
        \frac{\iota_0^k}{(\iota - \iota_0)^{k+1}}
        \leq
        \frac{\iota_0^k (\iota_0 + 1))^{k+1}}{\iota} \\
        \left| [(1 - 2z/\iota_0)^k] \, \frac{1}{\iota + 2z} \right|
        & =
        \frac{\iota_0^k}{(\iota + \iota_0)^{k+1}}
        \leq
        \frac{\iota_0^{k}}{\iota}
    \end{align*}
    where in the first inequality we used the fact that
    \[
        \frac{1}{\iota - \iota_0}
        \leq
        \frac{\iota_0 + 1 }{\iota}.
    \]
    Therefore, for any $1 \leq m \leq \mu_0$, we have
    \begin{align*}
        \left| (1 - 2z/\iota_0)^{-m}] \, \phi_{\alpha} \right|
        & \leq
        \sum_{\bar{\epsilon}}
        \left|
        [(1 - 2z/\iota_0)^{-m}]
        \prod_{i=1}^{r}
        \frac{1}{\iota_i - \langle \bar{\epsilon}, \bar{c}_i \rangle z}
        \right| \\
        & \leq
        \sum_{\bar{\epsilon}}
        \sum_{\substack{(k_1, \dots, k_r) \in \mathbb{Z}_{\geq -1}^r \\ k_1 + \cdots + k_r = -m}}
        \left|
        [(1 - 2z/\iota_0)^{k_i}] \,
        \frac{1}{\iota_i - \langle \bar{\epsilon}, \bar{c}_i \rangle z}
        \right| \\
        & \leq
        2^n
        \binom{2r-m-1}{r-1} 
        \iota_0^{\mu_0 - m} (\iota_0 + 1)^{2\mu_0 - 2m}
        \prod_{i=1}^{r}
        \frac{1}{\iota_i(\alpha)} \\
        & \leq
        2^n
        \binom{2r-2}{r-1}
        \big( \iota_0 (\iota_0 + 1)^2 \big)^{\mu_0 - 1}
        \prod_{i=1}^{r}
        \frac{1}{\iota_i(\alpha)}
    \end{align*}
    where in the third inequality we use the inequalities established above and the identity
    \[
        \sum_{\substack{(k_1, \dots, k_r) \in \mathbb{Z}_{\geq -1}^r \\ k_1 + \cdots + k_r = -m}} 1
        =
        \binom{2r-m-1}{r-1}.
    \]
    Therefore, for any $1 \leq m \leq \mu_0$,
    \[
        \left| [(\iota_0/2 - z)^{-m}] \, \phi_\alpha(z) \right|
        \leq
        C_{1}
        \prod_{i=1}^{r} \frac{1}{\iota_i(\alpha)}
    \]
    where
    \[
        C_{1}
        =
        2^n
        \binom{2r-2}{r-1}
        \big( \iota_0 (\iota_0 + 1)^2 \big)^{\mu_0 - 1}.
    \]
    This completes the proof.

\end{proof}

The previous lemma, together with the finite volume assumption \eqref{eq:V<oo}, allows us to establish the following result on the analytic properties of $\phi$, especially near its singularities.
\begin{lem} \label{lem:phiSing}
    The function $\phi$ extends holomorphically to $\mathbb{C} \smallsetminus \{ \pm \iota / 2 : \iota \in I \}$,
    has poles of order $\mu_0$ at $\pm \iota_0/2$.
    For each $\zeta \in \{ \pm \iota_0/2 \}$, we can write
    \begin{equation} \label{eq:*cst}
        \phi(z)
        =
        \frac{\tilde{\phi}_{\zeta}(z)}{(1 - z/\zeta)^{\mu_0}},
    \end{equation}
    where $\tilde{\phi}_\zeta$ is holomorphic near $\zeta$ and $\tilde{\phi}_\zeta(\zeta) \neq 0$.
    In particular, we have
    \[
        \phi(z)
        =
        \frac{\tilde{\phi}_{\zeta}(\zeta)}{(1 - z/\zeta)^{\mu_0}} \, (1 + \mathrm{O}(1 - z/\zeta)),
        \quad
        \text{as }z \to \zeta.
    \]
    Furthermore, we have
    \begin{equation} \label{eq:*cstPari}
        \tilde{\phi}_{-\iota_0/2}(-\iota_0/2) = (-1)^n \tilde{\phi}_{\iota_0/2}(\iota_0/2) > 0.
    \end{equation}

\end{lem}
\begin{proof}

    To prove that $\phi$ extends holomorphically to $\mathbb{C} \smallsetminus \{ \pm \iota / 2 : \iota \in I \}$, it suffices to show that for every compact set $K \subset \mathbb{C} \smallsetminus \{ \pm \iota / 2 : \iota \in I \}$, $\phi$ converges uniformly on $K$. Because $K$ is compact and avoids $\pm \iota / 2$ for all $\iota \in I$, there exists $C_{2} > 0$ depending on $K$ such that
    \[
        \frac{1}{|\iota \pm 2z|}
        \leq
        \frac{C_{2}}{\iota},
        \qquad
        \forall z \in K, \ \forall \iota \in I.
    \]

    Since $c_{i,j}(\alpha) \in \{ 0,1,2 \}$ and $c_{i,1}(\alpha) + \cdots + c_{i,n}(\alpha) = 2$ for all $\alpha, i,j$,
    every factor $(\iota_i(\alpha) - \langle \bar{\epsilon}, \bar{c}_i(\alpha) \rangle z)^{-1}$ in \eqref{eq:phiAlpha} is of the form $(\iota_i(\alpha) \pm 2z)^{-1}$ or $\iota_i(\alpha)^{-1}$.
    Therefore,

    \[
        | \phi_\alpha(z) |
        \leq
        2^n
        C_{2}^r
        \prod_{i=1}^{r} \frac{1}{\iota_i(\alpha)},
        \qquad
        \forall \alpha \in A, \ \forall z \in K.
    \]
    Hence, the assumption \eqref{eq:V<oo} implies that
    \[
        | \phi(z) |
        \leq
        \sum_{\alpha \in A}
        | \phi_\alpha(z) |
        \leq
        2^n
        C_{2}^r
        V
        <
        \infty,
        \qquad
        \forall z \in K.
    \]
    Thus the series $\sum_\alpha \phi_\alpha(z)$ converges uniformly on $K$, establishing the first assertion.

    To prove that $\phi$ has poles of order $\mu_0$ at $\pm \iota_0 / 2$, it suffices to show that,
    for any $1 \leq m \leq \mu_0$,
    \begin{equation} \label{eq:a<oo}
        \sum_{\alpha \in A}
        |[(1 - 2z/\iota_0)^{-m}] \, \phi_\alpha(z)|
        <
        \infty
    \end{equation}
    and
    \begin{equation} \label{eq:a>0}
        \sum_{\alpha \in A}
        [(1 - 2z/\iota_0)^{-\mu_0}] \, \phi_\alpha(z)
        >
        0.
    \end{equation}
    By Lemma~\ref{lem:Laurent},
    \[
        \sum_{\alpha \in A}
        |[(1 - 2z/\iota_0)^{-m}] \, \phi_\alpha(z)|
        \leq
        C_{1} \sum_{\alpha \in A} \prod_{i=1}^{r} \frac{1}{\iota_i(\alpha)}
        \leq
        C_{1} V,
    \]
    where the constant $C_{1}$ is as in Lemma~\ref{lem:Laurent}.
    The bound \eqref{eq:a<oo} then follows from the assumption \eqref{eq:V<oo},
    and the bound \eqref{eq:a>0} follows from Lemma~\ref{lem:Linxiao}.
    Finally, \eqref{eq:*cstPari} follows from the fact that $\phi_\alpha$ is even when $n$ is even and odd when $n$ is odd for all $\alpha$.
\end{proof}
To simply notation, we write
\[ 
    \breve{\varphi}_{\alpha}(z)
    \coloneqq
    (\mathcal{B}_r \varphi_\alpha)(z),
    \qquad
    \text{and}
    \qquad
    \breve{\varphi}(z)
    =
    \sum_{\alpha \in A}
    \breve{\varphi}_{\alpha}(z).
\]
The following is the main result of this section.
\begin{kor} \label{cor:psiSing}
    %Suppose first that $n$ is even.
    For $\zeta \in \{ \pm \iota_0/2 \}$, let $\phi_{\zeta}(\zeta)$ be the non-zero constant defined by \eqref{eq:*cst}.
    Under the assumption \eqref{eq:V<oo}, we have
    \begin{enumerate}
        \item
            If $n' = 0$, then $\breve{\varphi}(z)$ extends holomorphically to
            $\mathbb{C} \smallsetminus \{ \pm \iota / 2 : \iota \in I \}$.
            Moreover, for each $\zeta \in \{ \pm \iota_0/2 \}$, we have
            \[
                \breve{\varphi}(z)
                =
                \frac{2^{-n} \phi_{\zeta}(\zeta)}{(1 - z/\zeta)^{\mu_0}} \, (1 + \mathrm{O}(1 - z/\zeta)),
                \qquad
                \text{as }z \to \zeta.
            \]

            \item
                If $n ' > 0$ and $\iota_0 > 1$, then $\breve{\varphi}(z)$ extends holomorphically to $\mathbb{C} \smallsetminus \{ x \in \mathbb{R} : |x| \geq 1/2 \}$.
            Moreover, for each $\zeta \in \{ \pm 1/2 \}$, we have
            \[
                \breve{\varphi}(z)
                =
                2^{-n-2n'}
                \phi(\zeta)
                \left( \log \frac{1}{1 - z/\zeta} \right)^{n'} (1 + \mathrm{O}(1 - z/\zeta)),
                \qquad
                \text{as }z \to \zeta.
            \]

            \item
                If $n ' > 0$ and $\iota_0 = 1$, then $\breve{\varphi}(z)$ extends holomorphically to $\mathbb{C} \smallsetminus \{ x \in \mathbb{R} : |x| \geq 1/2 \}$.
                Moreover,
                for each $\zeta \in \{ \pm 1/2 \}$, we have
                \[
                    \breve{\varphi}(z)
                    =
                    \frac{2^{-n-2n'} \phi_{\zeta}(\zeta)}{(1 - z/\zeta)^{\mu_0}} \left( \log \frac{1}{1 - z/\zeta} \right)^{n'} (1 + \mathrm{O}(1 - z/\zeta)),
                    \qquad
                    \text{as }z \to \zeta.
                \]
    \end{enumerate}
\end{kor}
\begin{proof}
    This follows directly from Lemma~\ref{lem:phiSing} and \eqref{eq:phiAlpha}.
\end{proof}

\subsubsection{Asymptotics}

We are now ready to carry out the asymptotic analysis of $[z^N] \, \breve{\varphi}(z)$.

\begin{prop} \label{prop:BvarphiAsym}
    Under the assumption \eqref{eq:V<oo}, we have, as $N \to \infty$:
    \begin{itemize}
        \item
            If $n' = 0$,
            \[
                [z^N] \, \breve{\varphi}(z)
                \sim
                (1 + (-1)^{n + N}) \, \phi_{\iota_0/2}(\iota_0/2)
                \, 
                2^{-n}
                \left( \frac{2}{\iota_0} \right)^N
                \,
                \frac{N^{\mu_0 - 1}}{(\mu_0 - 1)!}.
            \]

        \item
            If $n ' > 0$ and $\iota_0 > 1$,
            \[
                [z^N] \, \breve{\varphi}(z)
                \sim
                (1 + (-1)^{n + N}) \, \phi(1/2)
                \,
                2^{N-n - 2n'}
                \,
                \frac{n'\, (\log N)^{n'-1}}{N}.
            \]

        \item
            If $n ' > 0$ and $\iota_0 = 1$,
            \[ 
                [z^N] \, \breve{\varphi}(z)
                \sim
                (1 + (-1)^{n + N})
                \,
                \phi_{1/2}(1/2)
                \,
                2^{N - n - 2n'}
                \frac{N^{\mu_0 - 1}}{(\mu_0 - 1)!} (\log N)^{n'}.
            \]
    \end{itemize}
\end{prop}
\begin{proof}
    With Corollary~\ref{cor:psiSing} in place, the claim is a direct application of \cite[Theorem~VI.5]{Flajolet-Sedgewick}.
    Roughly speaking, this theorem asserts that if a function $f$ has a unique singularity $\zeta \in \mathbb{C}$ on the circle $|z| = |\zeta|$, and satisfies
    \begin{equation} \label{eq:VI5f}
        f(z) = \frac{1}{(1 - z/\zeta)^\alpha} \log \left( \frac{1}{1 - z/\zeta} \right)^\beta (1 + \mathrm{O}(1 - z/\zeta)),
        \qquad
        \text{as } z \to \zeta
    \end{equation}
    while remaining holomorphic in a domain slightly larger than $\{ z \in \mathbb{C} : |z| < |\zeta| \}$, then
    \begin{equation} \label{eq:VI5coef}
        [z^N] \, f(z)
        =
        \frac{1}{\zeta^N} \frac{N^{\alpha - 1}}{\Gamma(\alpha)} (\log N)^{\beta} (1 + \mathrm{O}(1/N)),
        \qquad
        \text{as } N \to \infty.
    \end{equation}
    This principle generalizes to functions with finitely many singularities on the circle of convergence.
    In such cases, the total asymptotic expansion is simply the sum of the contributions from each singularity.

    Care is needed in the case ``$n' > 0$ and $\iota_0 > 1$'', since \eqref{eq:VI5coef} falls to give the right asymptotics.
    In fact, in this case, the exponent $\alpha$ in \eqref{eq:VI5f} vanishes, a case not covered by \cite[Theorem~VI.2]{Flajolet-Sedgewick}).
    Although this can be handled by a careful reading of \cite[Chapter~VII]{Flajolet-Sedgewick}, we sketch a proof for this case for the reader's convenience.

    By Corollary~\ref{cor:psiSing}, $\breve{\varphi}$ can be written as
    \[
        \breve{\varphi}(z)
        =
        2^{-n-2n'} \phi(1/2) \left( \log \frac{1}{1 - z/2} \right)^{n'}
        +
        2^{-n-2n'} \phi(-1/2) \left( \log \frac{1}{1 + z/2} \right)^{n'}
        +
        f(z)
    \]
    where $f(z)$ is holomorphic in $\mathbb{C} \smallsetminus \{ x \in \BR : |x| \geq 1/2 \}$ and satisfies, for each $\zeta \in \{ \pm 1/2 \}$,
    \[
        f(z)
        =
        \mathrm{O}(\log((1 - z/\zeta)^{-1})^k (1 - z/\zeta)),
        \qquad
        \text{as } z \to \zeta.
    \]
    Accordingly, the coefficient $[z^N] \, \breve{\varphi}(z)$ can be decomposed into three parts.
    The contributions from the logarithmic singularities are
    \[
        [z^N] \, \left( \log \frac{1}{1 - z/\zeta} \right)^{n'}
        =
        \frac{s(N, n')}{\zeta^N} \frac{n'!}{N!}
    \]
    where $s(N, k)$ denotes the Stirling number of the first kind.
    The asymptotic behavior of these numbers is well-known (see, for example, \cite{Hwang}):
    \[
        s(N, n')
        =
        \frac{n'}{N} (\log N)^{k-1} (1 + \mathrm{O}(1/\log N)),
        \qquad
        \text{as } N \to \infty.
    \]
    Finally, the error term $[z^N] \, f(z)$ can be handled by the lemma below,
    which is the multiple singularity version of a special case of \cite[Theorem~VI.3]{Flajolet-Sedgewick}.
\end{proof}
\begin{lem}
    Let $R > 1$, and let $f$ be a holomorphic function in $\{ z \in \mathbb{C} : |z| \leq R \} \smallsetminus \{ x \in \mathbb{R} : |x| \geq 1 \}$.
    Assume that for each $\zeta \in \{ \pm 1 \}$, $f$ satisfies
    \[
        f(z)
        =
        \mathrm{O} \big( (1 - z/\zeta) \log( (1-z/\zeta)^{-1} )^k  \big),
        \qquad
        \text{as }z \to \zeta.
    \]
    Then, we have
    \[
        [z^N] \, f(z)
        =
        \mathrm{O}((\log N)^{k} / N^2),
        \qquad
        \text{as } N \to \infty.
    \]
\end{lem}
\begin{proof}[Proof sketch]
As outlined in \cite[Theorem~VI.5]{Flajolet-Sedgewick}):
it suffices to adapt the contour to the two singularities, then follow the proof of \cite[Theorem~VI.3]{Flajolet-Sedgewick} step by step (see also \cite[Theorem~2]{Flajolet-Odlyzko} for a more detailed proof).
\end{proof}

\subsection{Asymptotic for the frequencies}
Let's recall the context around the above process. 
We are interested in an essential local type $(\Sigma,\gamma_0)$ and more precisely in the frequency $\FC_g(\phi_{\ns}(\gamma_0))$ of its non-separating representation in large genus. Recall from \eqref{eq: asymp non sep step 2} and Proposition \ref{prop: sym independant of g} that this frequency is, as the genus grows, equivalent to
\begin{multline*}
\frac{2^{\chi_\ns+1}}{|\Stab_{\Map(\Sigma)}(\gamma_0)|} 
   \frac{(6g_\ns - 5 + 2n_\ns)!!}{g_\ns! 24^{g_\ns}} 
    \cdot [t^{-3\chi_\ns}] \sum\limits_{\sigma\in\CA(\Sigma)}
    \bigintsss_{\Delta_\sigma(\gamma_0)}
    \prod_{e \in E_\fix}
    \sinh(b_e t)
    \prod_{e \in E_\ann}
    \frac{\sinh(b_e t)^2}{b_e}
    \, d\bar{x}.
\end{multline*}

If we consider that 
\begin{itemize}
    \item $r$ is the number of arcs in a maximal arc system of $\Sigma_\hyp$,
    \item $n=|\D\Sigma_\hyp|=|E_\fix|$ is number of boundary component of $\Sigma_\hyp$,
    \item $n'=|\pi_0(\Sigma_\ann)|=|E_\ann|$ is the number of annular component of $\Sigma$,
    \item $A = \CA(\Sigma)$ the set of maximal arc systems in $\Sigma_\hyp$,
    \item $\iota(\alpha)$ the intersection number between the arc $\alpha \in \CA(\Sigma)$ and $\gamma_0$,
    \item and $(c_{i,j}(\alpha))_{i,j}$ the incidence matrix between the arc $\alpha$ and the boundary components of $\Sigma_\hyp$ ($c_{i,j} = 1$ if $\alpha_i$ is incident to the $j$-th component of $\partial \Sigma_\hyp$, and $= 0$ otherwise),
\end{itemize}
all what has been done above can be applied to obtain an asymptotic for $\FC_g(\phi_\ns(\gamma_0))$. However, since at this point the reader may not find completely clear how to come back to $\FC_g(\phi_\ns(\gamma_0))$ from what was done in the previous section, let's describe quickly the state of the game. For starters, we get from \eqref{asymp step 1} and \eqref{asym step 1 bis}) that
\[ \FC_g(\phi_\ns(\gamma_0)) \sim 
\dfrac{2^{\chi_\ns+1}(6g\ns-5+2n_\ns)!}{|\Stab_{\Map(\Sigma)}(\gamma_0)|24^{g_\ns}g_\ns!} 
\sum\limits_{\sigma\in\CA(\Sigma)}
\FC_\sigma(\gamma_0)\]
where $\FC_\sigma(\gamma_0)$ can be written as $[z^{3\chi_\ns}]\varphi_\sigma(z)$.
It then turns out the be easier to study $\CB_r\varphi_\sigma$ than $\varphi_\sigma$ itself, where $\CB$ is the transform defined in \eqref{eq:B}). In those terms we have 
\begin{equation} \label{eq:freq in gen function}
    \FC_g(\phi_\ns(\gamma_0)) \sim 
\dfrac{2^{\chi_\ns+1}(6g\ns-5+2n_\ns)!}{|\Stab_{\Map(\Sigma)}(\gamma_0)|24^{g_\ns}g_\ns!}\frac{1}{(-3\chi_\ns+r)!} 
\sum\limits_{\sigma\in\CA(\Sigma)}
[z^{3\chi_\ns}]\CB_r\varphi_\sigma(z).
\end{equation} 
Now, most of the previous section was dedicated to study \[\sum\limits_{\sigma\in\CA(\Sigma)}[z^{3\chi_\ns}]\CB_r\varphi_\sigma(z)=[z^{3\chi_\ns}]\sum\limits_{\sigma\in\CA(\Sigma)}\CB_r\varphi_\sigma(z).\] This function was called $\breve\varphi$, and Proposition~\ref{prop:BvarphiAsym} gives the asymptotic values of its Taylor coefficients. The next corollary the is a direct consequence this proposition and the following observation:$-3\chi_\ns=6g_{ns}-6+3n_{\ns} = 6g_{ns}-6 +3(n+2n') $ hence $-3\chi_\ns+n$ is always even. Also, the parity of the function $\phi$ defined in \eqref{eq:phi} is the same as the one of $n$ and then the same as $-3\chi_\ns$. 

\begin{kor} \label{cor:firstAsymFreq}  Let $(\Sigma,\gamma_0)$ be an essential local type with $r$ the number of arcs in a maximal arc system of $\Sigma_\hyp$, $n$ the number of boundary component of $\Sigma_\hyp$, and $n'$ the number of annular component of $\Sigma$. Denote by $\iota_0$ the minimal number of intersection between $\gamma_0$ and an arc of $\Sigma_\hyp$ and $\mu_0$ the maximal number on disjoint arcs reaching this minimum. There are functions $\phi$ and $\tilde{\phi}_{\iota_0/2}$ (defined in \eqref{eq:phi} and \eqref{eq:*cst}), both depending only on $(\Sigma,\gamma_0)$,  such that
\begin{equation} \label{eq:firstAsymFreq}
    \FC_g(\phi_\ns(\gamma_0)) \underset{g\to\infty}{\sim} \frac{2^{\chi_\ns+1}}{|\Stab_{\Map(\Sigma)}(\gamma_0)|}\cdot
    \frac{(6g_\ns - 5 + 2n_\ns)!!}{g_\ns! 24^{g_\ns}} 
   \cdot
   \frac{2^{1-3\chi_\ns-n_\ns}}{-3\chi_\ns\, (-3\chi_{\ns}+r)!} K(g,\gamma_0)
\end{equation}
where $K(g,\gamma_0)$ is equal to
\begin{enumerate}
    \item $\dfrac{\tilde{\phi}_{\iota_0/2}(\iota_0/2)}{\iota_0^{-3\chi_\ns}}\cdot\dfrac{(-3\chi_\ns)^{\mu_0}}{(\mu_0-1)!}$ when $n'=0$,
    \item $n'\phi(1/2) \, \log(-3\chi_\ns)^{n'-1}$ when $n'>0$ and $\iota_0>1$,
    \item $\dfrac{\tilde{\phi}_{1/2}(1/2)\, (-3\chi_\ns)^{\mu_0}}{(\mu_0-1)!}\log(-3\chi_\ns)^{n'}$ when $n'>0$ and $\iota_0=1$. \qed 
\end{enumerate}
\end{kor}

The goal is now to reduce this in something more friendly, let's work first on 
$$ 2^{\chi_\ns+1}\cdot
    \frac{(6g_\ns - 5 + 2n_\ns)!!}{g_\ns! 24^{g_\ns}} 
   \cdot
   \frac{2^{1-3\chi_\ns-n_\ns}}{-3\chi_\ns\, (-3\chi_{\ns}+r)!}.   $$
The equality $\chi(X)=\chi(\Sigma)+\chi_\ns$ implies that $g_\ns = g +\frac{1}{2}(\chi(\Sigma)-n_\ns)\sim g$. The quantity above is then equivalent in large genus to 
\[   \dfrac{2^{4g+1}}{g\, 24^g}
\cdot 
\dfrac{(6g_\ns - 5 + 2n_\ns)!!}{g_\ns! (-3\chi_{\ns}+r)!}
\sim
\left( \dfrac{2}{3}\right)^g \, 
\frac{2}{g} \cdot
\dfrac{(6g_\ns - 5 + 2n_\ns)!!}{g_\ns! (-3\chi_{\ns}+r)!}.
\]
Recall that from \eqref{eq defi dim flip} we get $r=-3\chi(\Sigma_\hyp)=-3\chi(\Sigma)$ hence $-3\chi_\ns+r=-3\chi(X)=6g-6.$ We can rewrite the above quantity as
 \[ \left( \dfrac{2}{3}\right)^g \, 
\frac{2}{g} \cdot
\dfrac{(6g_\ns - 5 + 2n_\ns)!!}{g_\ns! |3\chi(X)|!}.
\]
Taking now into account that for all $k>0$, $(2k+1)!!=(2k+1)!/(2^k\, k!)$, we get
\begin{align*}
    \left( \dfrac{2}{3}\right)^g \, 
\frac{2}{g} \cdot
\dfrac{(6g_\ns - 5 + 2n_\ns)!!}{g_\ns! |3\chi(X)|!} &=
\left( \dfrac{2}{3}\right)^g \, 
\frac{2}{g} \cdot 
\dfrac{(6g_\ns-5+2n_\ns)!}{2^{3g_\ns-3+n_\ns}(3g_\ns-3+n_\ns)!}\cdot
\dfrac{2}{g_\ns!\, |3\chi(X)|!} \\
& \sim  
\left( \dfrac{1}{3\cdot 2^2}\right)^g \, 
\frac{2}{g} \cdot
\dfrac{(6g_\ns-5+2n_\ns)!}{(3g_\ns-3+n_\ns)!\, g_\ns!}\cdot
\dfrac{1}{ |3\chi(X)|!}. 
\end{align*}
Using that $g_\ns = g+\frac{1}{2}(\chi(\Sigma)-n_\ns)$ we have
\begin{itemize}
    \item $g_\ns! \sim g! \, g^{\frac{1}{2}(\chi(\Sigma)-n_\ns)}$,
    \item $3g_\ns-3+n_\ns =(3g-3)+\frac{1}{2}(3\chi(\Sigma)-n_\ns) $ and $(3g_\ns-3+n_\ns)! \sim (3g-3)!\, (3g)^{\frac{1}{2}(3\chi(\Sigma)-n_\ns)}$,
    \item $6g_\ns-5+2n_\ns  = (6g-6)+ (3\chi(\Sigma)-n_\ns+1) $ and $(6g_\ns-6+2n_\ns)! \sim (6g-6)!\, (6g)^{3\chi(\Sigma)-n_\ns+1}$.
\end{itemize}
So, if we write $f(g) \asymp h(g)$ when $f(g)/h(g)$ tends to a constant as $g$ goes to infinity then we have
\begin{align*}
    \dfrac{(6g_\ns-6+2n_\ns)!}{(3g_\ns-3+n_\ns)!\, g_\ns!} & \asymp
\dfrac{(6g-6)!}{(3g-3)!g!}\,
g^{(3\chi(\Sigma)-n_\ns+1)-(\frac{1}{2}(\chi(\Sigma)-n_\ns))-({\frac{1}{2}(3\chi(\Sigma)-n_\ns)})} \\
 & \asymp
 \dfrac{(6g-6)!}{(3g-3)!g!}\,
 g^{\chi(\Sigma)+1}. 
\end{align*}
As a consequence
\[   
 2^{\chi_\ns+1}\cdot
    \frac{(6g_\ns - 5 + 2n_\ns)!!}{g_\ns! 24^{g_\ns}} 
   \cdot
   \frac{2^{1-3\chi_\ns-n_\ns}}{-3\chi_\ns\, (-3\chi_{\ns}+r)!} 
   \asymp
   \left(  \dfrac{1}{12} \right)^g \frac{g^{\chi(\Sigma)}}{g!(3g-3)!}.
\]
Applying now Stirling formulae, we get first
\[  g! \, (3g-3)! 
\sim \frac{1}{(3g)^3}g!\cdot(3g)!
\sim \frac{2\pi \sqrt{3}}{27} \cdot
\frac{g^{4g-2}}{e^{4g}}
3^{3g}
\]
and then
\begin{equation} \label{eq:common part reduced}
    2^{\chi_\ns+1}\cdot
    \frac{(6g_\ns - 5 + 2n_\ns)!!}{g_\ns! 24^{g_\ns}} 
   \cdot
   \frac{2^{1-3\chi_\ns-n_\ns}}{-3\chi_\ns\, (-3\chi_{\ns}+r)!} 
   \asymp
   \frac{1}{2^{2g}}\left( \dfrac{e}{3g}\right)^{4g} g^{\chi(\Sigma)+2}.
\end{equation}     
Using the same process as above we have:
\begin{itemize}
    \item $\dfrac{(-3\chi_\ns)^{\mu_0}}{\iota_0^{-3\chi_\ns}(\mu_0-1)!}\asymp\dfrac{g^{\mu_0}}{\iota_0^{6g}},$
    \item  $\log(-3\chi_\ns)^{n'-1}\sim \log(g)^{n'-1}$,
    \item  $\dfrac{(-3\chi_\ns)^{\mu_0}}{(\mu_0-1)!}\, \log(-3\chi_\ns)^{n'}\asymp g^{\mu_0}\log(g)^{n'}.$
\end{itemize}
This together with \eqref{eq:common part reduced} give the following final asymptotic formulae.

\begin{sat}\label{thm:final asym}Let $(\Sigma,\gamma_0)$ be an essential local type with $n'$ annular components. If $\iota_0$ is the minimal number of intersection between $\gamma_0$ and an arc of $\Sigma_\hyp$ and $\mu_0$ the maximal number of disjoint arcs reaching this minimum then 
    \[
        \mathfrak{c}_{g}(\phi_{\ns}(\gamma_0))
        \underset{g\to\infty}{\asymp }
        \dfrac{1}{2^{2g}}
        \left( \frac{e}{3g}\right)^{4g}g^{\chi(\Sigma)+2}
        \left\lbrace
\begin{array}{lr}
\dfrac{g^{\mu_0}}{\iota_0^{6g}} \quad &\text{if }n'=0,\\ \\
\log(g)^{n'-1} \quad &\text{if }n'>0\text{ and }\iota_0>1,\\ \\
g^{\mu_0}\log(g)^{n'} \quad &\text{if }n'>0\text{ and }\iota_0=1.\\
\end{array}\right.   \]
\qed
\end{sat}

As a direct consequence of this theorem and Theorem \ref{sat dominant type} we have the following corollary.

\begin{kor}\label{kor final asyp loc type}
    Let $(\Sigma,\gamma_0)$ be an essential local type with $n'$ annular components. If $\iota_0$ is the minimal number of intersection between $\gamma_0$ and an arc of $\Sigma_\hyp$ and $\mu_0$ the maximal number of disjoint arcs reaching this minimum then
    \[
        \mathfrak{c}_{g}(\Sigma,\gamma_0)
        \underset{g\to\infty}{\asymp }
        \dfrac{1}{2^{2g}}
        \left( \frac{e}{3g}\right)^{4g}g^{\chi(\Sigma)+2}
                \left\lbrace
\begin{array}{lr}
\dfrac{g^{\mu_0}}{\iota_0^{6g}} \quad &\text{if }n'=0,\\ \\
\log(g)^{n'-1} \quad &\text{if }n'>0\text{ and }\iota_0>1,\\ \\
g^{\mu_0}\log(g)^{n'} \quad &\text{if }n'>0\text{ and }\iota_0=1.\\
\end{array}\right.   \]
\qed
\end{kor}

Remark that the common part in the asymptotic above is the frequency for a simple closed curve obtained by Delecroix--Goujard--Zograf--Zorich \eqref{eq DGZZ asymptotic2}.

The results in the two following sections are corollaries of the above asymptotic results.

\subsection{Inner versus Outer simple components}

We want to compare local types corresponding to the same curve with self intersection together with a simple closed curve. This simple closed curve can either be the boundary of the subsurface filled by the curve (inner simple component) or be the core of an annulus of the local type (outer simple component). We start by comparing those two cases in the case of the figure $8$ in a pair of pants.

\begin{ex}  We consider the figure 8 plus one simple component, either inner or outer.

\begin{figure}[ht!]
    \centering
    \begin{subfigure}[t]{0.3\textwidth}
		\centering
		\includegraphics[scale=0.7]{ 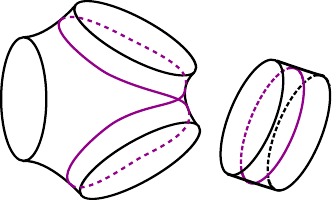}
		\caption{Essential local type with outer simple component}\label{fig:8a}		
	\end{subfigure}
	\quad
    \begin{subfigure}[t]{0.3\textwidth}
		\centering
		\includegraphics[scale=0.7]{ 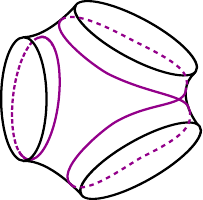}
		\caption{Essential local type with inner simple component}\label{fig:8b}		
	\end{subfigure}
    \quad 
    \begin{subfigure}[t]{0.3\textwidth}
		\centering
		\includegraphics[scale=0.7]{ 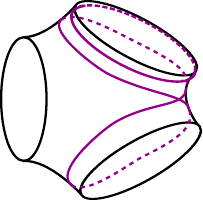}
		\caption{Essential local type with inner simple component}\label{fig:8c}		
	\end{subfigure}
    \caption{Comparing the figure 8 essential local type with inner or outer simple closed component}
    \label{fig:InnerVSOuter}
\end{figure}

\begin{center}
    \begin{tabular}{c||c|c|c|l}
    Local type & $\iota_0$ & $\mu_0$ & $n'$ & \qquad $\FC_g \asymp$ \\
    \hline 
    Figure \ref{fig:8a}&  1 & 2 & 1 & $\dfrac{1}{2^{2g}}\left( \dfrac{e}{3g}\right)^{4g}\cdot g^2\log(g)$ \rule[-17 pt]{0pt}{40 pt} \\
        \hline
    Figure \ref{fig:8b} & 2 & 3 & 0 & $\dfrac{1}{2^{2g}}\left( \dfrac{e}{3g}\right)^{4g}\cdot \dfrac{g^3}{2^{6g+1}}$ \rule[-17 pt]{0pt}{40 pt}\\
    \hline
    Figure \ref{fig:8c} & 1 & 1 & 0 & $\dfrac{1}{2^{2g}}\left( \dfrac{e}{3g}\right)^{4g}\cdot g$ \rule[-17 pt]{0pt}{40 pt}
\end{tabular}
\end{center}

The local type of the figure \ref{fig:8a} (\textit{ie.} the one with outer simple component) is dominant over the local types with inner simple components in large genus.

\end{ex}

In order to phrase a clear statement we need to define a bit of vocabulary.

\begin{defi}
    Let $(\Sigma,\flip_\Sigma,\gamma_0)$ be a local type and $\delta\in\gamma_0$ be a simple closed curve (recall that $\gamma_0$ is a multicurve), such a curve is called a \emph{simple component}. 
    
    This simple component is said to be an \emph{outer simple component} if it is the core of a component of $\Sigma_\ann$. 
    It is an \emph{inner simple component} if it is parallel to a boundary component of $\Sigma_\hyp$. 

    Two essential local types $(\Sigma,\gamma_0)$ and $(\Sigma',\gamma_0')$ \emph{differ from inner-outer components} if 
    \begin{itemize}
        \item $\Sigma_\hyp=\Sigma'_\hyp$,
        \item there is $\delta$ a union of inner components of $\gamma_0$ such that $(\Sigma_\hyp,\gamma_0\cap\Sigma_\hyp-\delta)=(\Sigma_\hyp,\gamma_0'\cap\Sigma_\hyp)$, and
        \item $(\Sigma,\gamma_0')$ has $|\delta|$ more annuli components than $(\Sigma,\gamma_0)$.
    \end{itemize}
\end{defi}

If two local types differ by an owner-outer component it basically means that some inner simple components of the first local type are turned into outer simple component to create the second local type. For example, the local types of Figure \ref{fig:8b} and Figure  \ref{fig:8a} differ from an inner-outer component, the one from Figure \ref{fig:8c} and Figure \ref{fig:8a} also.

\begin{kor}
    Let $(\Sigma,\gamma_0)$ be an essential local type. If $(\Sigma_1,\gamma_1)$ is an essential local type such that $(\Sigma_1,\gamma_1)$ and $(\Sigma,\gamma_0)$ differ from inner-outer components then $$\FC_g(\Sigma_1,\gamma_1)=\mathrm{o}(\Sigma,\gamma_0).$$ 
    ie. If the number of simple component is fixed, for a given set of non-simple components, the dominant local type is the one for which all the simple components are outer components.
\end{kor}

\begin{proof} It is enough to show it for $(\Sigma,\gamma_0)$ having one more outer simple component and one less inner simple component (the rest of the curve is the same).

Let $\iota_0,\mu_0$ and $n'$ be as in Theorem \ref{thm:final asym} for $(\Sigma,\gamma_0)$, and $\iota_{1},\mu_{1}$, $n'_1$ the same quantities for the local type $(\Sigma_1,\gamma_1)$. Note that $\Sigma_\hyp=\Sigma_{1,\hyp}$ then $\chi(\Sigma)=\chi(\Sigma_1)$. So, to compare $\FC_g(\phi_\ns(\gamma_0))$ and $\FC_g(\phi_\ns(\gamma_1))$ what we need to compare is the part of the asymptotic in the bracket on Theorem \ref{thm:final asym}.

Passing from the local type $(\Sigma,\gamma_0)$ to the one $(\Sigma_1,\gamma_1)$ can be seen as a surgery: one of the annular component of $\Sigma$ is glued to a boundary component on $\Sigma_{\hyp}$ which has no inner simple component on $(\Sigma_0,\gamma_0)$. Hence, it is clear that $n'_1=n'-1$ and $\iota_{1}\ge\iota_0$. Note that in the case where $\iota_0=\iota_{1}$ we necessarily have $\mu_{1}\le\mu_0$ (every familly of disjoint arcs in $\Sigma_{1,\hyp}$ with all arcs meeting $\gamma_1$ exactly $\iota_0$ times is also a family of arcs in $\Sigma_\hyp$ for which all arc meet $\gamma_0$ exactly $\iota_0$ times).

Applying Theorem \ref{thm:final asym} we can write 
\[ \FC_g(\phi_\ns(\gamma_0))\asymp
\dfrac{1}{2^{2g}}
\left( \frac{e}{3g}\right)^{4g}g^{\chi(\Sigma)+2} 
\cdot K(g) 
\qquad  
\FC_g(\phi_\ns(\gamma_1))\asymp
\dfrac{1}{2^{2g}}
\left( \frac{e}{3g}\right)^{4g}g^{\chi(\Sigma)+2}
H(g),  \]
where $K$ and $H$ are given by Table \ref{table}. In all the cases $H(g)=\mathrm{o}(K(g))$, this concludes the proof.

\begin{table}[ht!]
    \centering

\begin{tabular}{|c|c|c|c|c|c|}
\hline

&\multicolumn{3}{|c|}{ $\iota_0=1$} & \multicolumn{2}{c|}{$\iota_0>1$} \\ \hline

$ K(g)=$ &\multicolumn{3}{|c|}{$g^{\mu_0}\cdot\log(g)^{n'}$} & \multicolumn{2}{c|}{$\log(g)^{n'-1}$} \\ \hline \hline

&\multirow{2}{*}{$\iota_{1}=\iota_0=1$} & \multicolumn{2}{c|}{$\iota_{1}>1$} & \multirow{2}{*}{$n'_1=n'_1-1=0$} & \multirow{2}{*}{$n'_1>0$} \\ \cline{3-4}

& & $n'_1=n'-1=0$ & $n'_1>0$ &  &  \\ \hline

$ H(g)=$ & $g^{\mu_{1}}\log(g)^{n'-1}$ & $\dfrac{g^{\mu_{1}}}{\iota_{1}^{6g}}$ & $\log(g)^{n'-2}$ & $\dfrac{g^{\mu_{1}}}{\iota_{1}^{6g}}$ & $\log(g)^{n'-2}$ \\ \hline
\end{tabular}

    \caption{Asymptotic for the frequency depending on the value of $\iota_0,\iota_{0,1},n'$.}
    \label{table}
\end{table}
\end{proof}

\subsection{Essential versus non-essential local type}

As we saw in Theorem \ref{sat dominant type}, non-separating realizations are dominant. This is what basically allows us to restrict ourselves to study the asymptotic behavior of frequencies in that case. We will next argue that, essentially, one can restrict oneself to frequencies of essential local types.

\begin{sat} \label{thm: essential negligeable bis}
    Given a compact surface with boundary $\Sigma$ and a filling multicurve $\gamma_0\subset \Sigma$, let $(\Sigma,\gamma_0)$ be the associated essential local type and $(\Sigma,\flip_\Sigma,\gamma_0)$ a non-essential local type. If $\phi^e_{\ns}$ denotes the non-separating realization of $(\Sigma,\gamma_0)$ and $\phi_{\ns}$ the one of $(\Sigma,\flip_\Sigma,\gamma_0)$ then

\begin{equation} \label{essential is dominant}
    \FC_g(\phi_{\ns}(\gamma_0))\; = \mathrm{O}( \FC_g(\phi^e_{\ns}(\gamma_0))).
\end{equation}
Moreover, if $\Sigma$ has no annular components and retract on $\gamma_0$ then 
\begin{equation} \label{eq:essential dominant retract}
    \FC_g(\phi_{\ns}(\gamma_0))\; = \mathrm{O}\left(\frac{1}{g} \FC_g(\phi^e_{\ns}(\gamma_0))\right). 
\end{equation}
\end{sat}

The proof of Theorem \ref{thm: essential negligeable bis}
will rely on the existence of specific essential local types obtained by surgery that will make the link between the frequency for an essential local type $(\Sigma,\gamma_0)$ and the one of a corresponding non-essential local type $(\Sigma,\flip_\Sigma,\gamma_0)$. This will rely on the fact that the frequency is initially coming from a counting of orbit points.
Recall that for any multicurve $\gamma$ is a closed genus $g$ surface $X$, the constant $\FC_g(\gamma)$ is related to the number $N_X(\gamma,L)=\#\{\delta\in\Map(X)\cdot\gamma |\ell_X(\delta)\le L \}$ of curves of type $\gamma$ with length at most $L$ as follows: 
\[  \lim\limits_{L \to\infty} \dfrac{N_X(\gamma,L)}{L^{6g-6}} = \dfrac{\FC_g(\gamma)}{b_g}\FM_{Thu}(\{\lambda\in\CM\CL(X) | \ell_X(\lambda)\le 1 \}).\]

\begin{lem} \label{lem:surgery-reduction-non-essential}
Let $(\Sigma,\flip_\Sigma,\gamma_0)$ be a non-essential local type, and let $\ell=|\Gamma_\flip|$
be the number of two-element orbits of $\flip_\Sigma$ in
$\D\Sigma_{\hyp}$. Let $\Sigma'$ be the surface obtained from $\Sigma$ by
attaching annuli
\[
A_1,\ldots,A_\ell
\]
associated to the elements of $\Gamma_\flip$. Denote by
\[
\beta_1,\ldots,\beta_\ell
\]
the core curves of these annuli. Then there are finitely many multicurves
\[
\gamma_1,\ldots,\gamma_m\subset \Sigma'
\]
and a constant $K\ge 1$, depending only on the local type
$(\Sigma,\flip_\Sigma,\gamma_0)$, with the following properties.
\begin{enumerate}
\item
Each $\gamma_i$ is obtained from $\gamma_0$ by performing one admissible
surgery across each annulus $A_j$. In particular, each $\gamma_i$ fills
$\Sigma'$, and hence $(\Sigma',\gamma_i)$ is an essential local type.
Moreover, for every closed hyperbolic surface $X$ and every
non-separating realization
\[
\phi_{\ns}:\Sigma\to X
\]
of $(\Sigma,\flip_\Sigma,\gamma_0)$, let
\[
\gamma_{\ns}=\phi_{\ns}(\gamma_0).
\]
The realization $\phi_{\ns}$ extends over the annuli $A_j$ to a map
\[
\widehat\phi_{\ns}:\Sigma'\to X,
\]
where $\widehat\phi_{\ns}(A_j)$ is the annular complementary component of
$X\setminus\gamma_{\ns}$ associated to $A_j$. Then there is a constant
$C_X>0$, depending only on the hyperbolic metric of $X$, such that, for
every $L>0$,
\[
N_X(\gamma_{\ns},L)
\le
K\sum_{i=1}^m
N_X\bigl(\widehat\phi_{\ns}(\gamma_i),L+\ell C_X\bigr).
\]

\item
For every $i=1,\ldots,m$, let $\iota_i,\mu_i$ be the invariants associated
to the essential local type $(\Sigma',\gamma_i)$, and let $\iota_0,\mu_0$
be the corresponding invariants associated to $(\Sigma,\gamma_0)$. Then
\[
\chi(\Sigma')=\chi(\Sigma)  \qquad\text{and}\qquad \iota_i\ge \iota_0.
\]
Moreover, if $\iota_i=\iota_0$, then
\[
\mu_i\le \mu_0,
\]
and if $\Sigma_\hyp$ retracts on $\gamma_0$ then 
\[
\mu_i < \mu_0.
\]
\end{enumerate}
\end{lem}

\begin{proof}
We first construct the finite list of multicurves. For each
$j=1,\ldots,\ell$, the annulus $A_j$ is attached to the two boundary
components of $\Sigma_{\hyp}$ which are exchanged by the corresponding
element of $\Gamma_\flip$.

Let us spell out the local picture. For a realization
$\phi_{\ns}:\Sigma\to X$, it can be extended to a realization $\hat\phi_\ns:\Sigma'\to X$ by send the annuli $A_j\subset\Sigma'$ to the annuli component of $X\setminus\phi_\ns(\Sigma)$ corresponding to the right curve of $\Gamma_\flip$. The image of $A_j$ is an annular in a complementary
component of
\[
X\setminus \gamma_{\ns}.
\]
This component is a finite-sided double-crowned annulus. Its two crowns
are finite concatenations of sides, and each side is a subsegment of
$\gamma_{\ns}$ between consecutive intersection points of the curve. The
core curve of this annulus is $\widehat\phi_{\ns}(\beta_j)$.

An admissible surgery arc across $A_j$ is a properly embedded arc in this
double-crowned annulus which joins one side of one crown to one side of
the other crown and crosses the core curve $\beta_j$ exactly once. Surgery
along such an arc means that one replaces the two chosen boundary
subsegments of the crowns by two non-parallel copies of the surgery arc as illustrated in Figure \ref{fig:SurgeryCrown}, and
then smooths the result.

 \begin{figure}[ht!]
    \centering
    \includegraphics[width=0.3\linewidth]{ 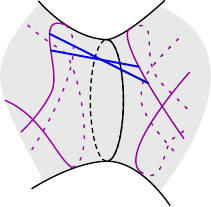}
    \caption{Local surgery for the constructions of the $\gamma_i$}
    \label{fig:SurgeryCrown}
\end{figure}

For a fixed $j$, there are only finitely many admissible surgery arcs up
to isotopy relative to $\gamma_0$ and powers of the Dehn twist about
$\beta_j$. Indeed, one only has to choose the side of the first crown and
the side of the second crown on which the endpoints of the surgery arc
lie. There are only finitely many such sides. Representatives which differ
by a power of the Dehn twist about $\beta_j$ give curves in the same
mapping-class-group orbit in $\Sigma'$.

Choosing one admissible surgery arc for each
$j=1,\ldots,\ell$ and performing the surgeries simultaneously gives only
finitely many mapping-class-group types of multicurves in $\Sigma'$.
Choose one representative of each such type and call the resulting list
\[
\gamma_1,\ldots,\gamma_m.
\]
This list depends only on the local type
$(\Sigma,\flip_\Sigma,\gamma_0)$.

Since $\gamma_0$ fills $\Sigma$, the complement
$\Sigma\setminus\gamma_0$ consists of disks and boundary-parallel annuli.
After attaching the annuli $A_j$, the only new possible non-boundary
parallel annular complementary regions are the annuli around the cores
$\beta_j$. Performing one admissible surgery across each $A_j$ cuts each
of these double-crowned annuli. Therefore each $\gamma_i$ fills
$\Sigma'$. Thus $(\Sigma',\gamma_i)$ is an essential local type when
$\Sigma'$ is equipped with the essential flip.

We now prove the counting estimate. Fix a closed hyperbolic surface $X$
and a non-separating realization
\[
\phi_{\ns}:\Sigma\to X.
\]
Let
\[
\gamma_{\ns}=\phi_{\ns}(\gamma_0).
\]
For each $j$, the image $\widehat\phi_{\ns}(A_j)$ is an annular
complementary component of $X\setminus\gamma_{\ns}$ with core curve
$\widehat\phi_{\ns}(\beta_j)$.

We first claim that there is a constant $w_X>0$, depending only on
$X$, such that every annular complementary component arising from an
element of the orbit $\Map(X)\cdot\gamma_{\ns}$ contains a transverse
crossing arc of length at most $w_X$. This is because, the area of the collar of 
$\beta_j$ of width $w_X$ is 
\[
2\ell_X(\beta_j) \sinh (w_X).
\]
Therefore, if $A_j$ contains the collar of $\beta_j$ of width $w_X$ then  
\[
\area(X) \ge \area(A_j) \ge 2 \sys(X) \sinh (w_X)
\]
where systole $\sys(X)$ is the length of the shortest curve in $X$. Hence, we can take 
\[
w_X=\sinh^{-1}\left(\frac{\area(X)}{2\sys(X)}\right).
\]
Choose a constant
\[
C_X>2w_X.
\]
The factor $2$ accounts for the two parallel copies of the crossing arc
used in the surgery; the arbitrarily small smoothing error is absorbed by
the strict inequality in the choice of $C_X$.

Let
\[
\delta\in \Map(X)\cdot\gamma_{\ns} \qquad{with}\qquad \ell_X(\delta)\le L.
\]
For each of the $\ell$ annular complementary components corresponding to
$\Gamma_\flip$, choose a crossing arc of length at most $w_X$ and perform
the admissible surgery. Since the annuli are pairwise disjoint, the
surgeries can be performed simultaneously. The resulting multicurve,
after replacing it by its geodesic representative, has length at most
\[
L+\ell C_X.
\]
By construction, it lies in the mapping class group orbit of one of the
multicurves $\widehat\phi_{\ns}(\gamma_i)$.

Thus we obtain a map at the level of curves
\[
F:
\{\delta\in\Map(X)\cdot\gamma_{\ns}\mid \ell_X(\delta)\le L\}
\longrightarrow
\coprod_{i=1}^m
\{\xi\in\Map(X)\cdot\widehat\phi_{\ns}(\gamma_i)
\mid \ell_X(\xi)\le L+\ell C_X\}.
\]

All the elements $\xi$ in the image of $F$ have at most $i(\gamma_0,\gamma_0)+|\Gamma_\flip|$ self-intersections. Hence, there is at most K (independent from $\xi$) ways to resolve simultaneously $\Gamma_\flip$ self intersections of $\xi$. However, for any $\delta$ such that $F(\delta)=\xi$, $\delta$ can be obtained by resolving at most $|\Gamma_\flip\vert$ singularities. Thus, $F$ is at most $K$ to one.

It follows that
\[
N_X(\gamma_{\ns},L)
\le
K\sum_{i=1}^m
N_X\bigl(\widehat\phi_{\ns}(\gamma_i),L+\ell C_X\bigr),
\]
which proves the first assertion.

Since $\Sigma'$ is
obtained from $\Sigma$ by attaching $\ell$ annuli along boundary
components of $\Sigma_{\hyp}$, we have
\[
\chi(\Sigma')=\chi(\Sigma).
\]
We now compare $\iota_i,\mu_i$ with $\iota_0,\mu_0$.  Let
\[
B=\beta_1\cup\cdots\cup\beta_\ell.
\]
Cutting $\Sigma'_{\hyp}$ along $B$, and then collapsing the collar annuli
between the two copies of $B$ and the old boundary components of
$\Sigma_{\hyp}$, recovers $\Sigma_{\hyp}$.

Let $\omega\subset\Sigma'_{\hyp}$ be an essential simple arc. Put
$\omega$ in minimal position with respect to $B$ and with respect to
$\gamma_i$. Cutting $\omega$ along $B$ and using the above identification
with $\Sigma_{\hyp}$ gives a finite collection of proper arcs in
$\Sigma_{\hyp}$. Discard the boundary-parallel components and denote the
remaining essential components by
\[
\omega_1,\ldots,\omega_q.
\]
Since $\omega$ is essential, we have $q\ge 1$. Moreover, if
$\omega\cap B\neq\emptyset$, then $q\ge 2$. Otherwise all but one of the
pieces obtained by cutting along $B$ would be boundary-parallel in
$\Sigma_{\hyp}$, and an outermost such boundary-parallel piece would give
an isotopy of $\omega$ reducing $|\omega\cap B|$, contradicting the
minimality of the position of $\omega$ with respect to $B$.

By construction of the surgery, after cutting along $B$, the curve
$\gamma_i$ restricts to $\gamma_0$ on the original surface
$\Sigma_{\hyp}$, up to boundary-collar pieces coming from the surgery.
Those collar pieces can only add intersections. Therefore
\[
\iota(\omega,\gamma_i)
\ge
\sum_{r=1}^q \iota(\omega_r,\gamma_0).
\]
Since each $\omega_r$ is an essential arc in $\Sigma_{\hyp}$, we have
\[
\iota(\omega_r,\gamma_0)\ge \iota_0
\]
for every $r$. Hence
\[
\iota(\omega,\gamma_i)
\ge
q\iota_0
\ge
\iota_0.
\]
Taking the minimum over all essential arcs
$\omega\subset\Sigma'_{\hyp}$ gives $\iota_i\ge \iota_0$.

Suppose now that
\[
\iota_i=\iota_0.
\]
Let $\omega\subset\Sigma'_{\hyp}$ be an essential arc satisfying
\[
\iota(\omega,\gamma_i)=\iota_i=\iota_0.
\]
The inequalities above force $q=1$. Since $\gamma_0$ fills $\Sigma$, we
have $\iota_0>0$. Therefore $\omega$ cannot cross any component of $B$:
if $\omega\cap B\neq\emptyset$, then $q\ge 2$, and hence
\[
\iota(\omega,\gamma_i)\ge 2\iota_0>\iota_0,
\]
a contradiction. Thus every $\iota_i$-minimizing arc for $\gamma_i$ is
represented by an arc contained in the original surface $\Sigma_{\hyp}$,
and this arc intersects $\gamma_0$ exactly $\iota_0$ times.

Now take a collection of disjoint, pairwise non-parallel essential arcs
\[
\omega_1,\ldots,\omega_{\mu_i}\subset\Sigma'_{\hyp}
\]
such that
\[
\iota(\omega_r,\gamma_i)=\iota_i=\iota_0
\]
for every $r$. By the previous paragraph, each $\omega_r$ is disjoint
from $B$ and hence gives an essential arc in $\Sigma_{\hyp}$ intersecting
$\gamma_0$ exactly $\iota_0$ times. These arcs remain disjoint in
$\Sigma_{\hyp}$. They also remain pairwise non-parallel in
$\Sigma_{\hyp}$: if two of them were parallel in $\Sigma_{\hyp}$, the
parallelism region would still be a parallelism region after attaching
the annuli $A_j$, contradicting that the arcs were non-parallel in
$\Sigma'_{\hyp}$.

Therefore the arcs
\[
\omega_1,\ldots,\omega_{\mu_i}
\]
form a collection of $\mu_i$ disjoint, pairwise non-parallel essential
arcs in $\Sigma_{\hyp}$, each intersecting $\gamma_0$ exactly
$\iota_0$ times. By definition of $\mu_0$, this implies
\[
\mu_i\le \mu_0.
\]
Also, if $\Sigma$ retracts on $\gamma_0$ then $\iota_0=1$ and any arc system with $\mu_0$ arcs all meeting $\gamma_0$ once has arcs starting on each boundary of $\Sigma_\hyp$ and hence on each element of $B$. This ensures that $\mu_i<\mu_0$ for any $i$
\end{proof}

We are know ready to prove Theorem \ref{thm: essential negligeable bis}

\medskip

\textit{Proof of Theorem \ref{thm: essential negligeable bis}}
For any multicurve $\gamma$ in a closed genus $g$ surface $X$, 
\[  \lim\limits_{L \to\infty} \dfrac{N_X(\gamma,L)}{L^{6g-6}} = \dfrac{\FC_g(\gamma)}{b_g}\FM_{Thu}(\{\lambda\in\CM\CL(X) | \ell_X(\lambda)\le 1 \}).\]

Lemma \ref{lem:surgery-reduction-non-essential} ensures that 
\[   \dfrac{\FC_g(\phi_\ns(\gamma_0))}{\FC_g(\phi_\ns^e(\gamma_0))} \le
\lim\limits_{L\to\infty} 
\dfrac{\sum_{i=1}^m K(\gamma_0)\cdot  N\bigl(\phi_\ns^i(\gamma_i),L+|\Gamma_\flip| C(X)\bigr)}{N_X(\phi_\ns^e(\gamma_0),L)} 
=K(\gamma_0)
\lim\limits_{L\to\infty}
\dfrac{\sum_{i=1}^m N\bigl(\phi_\ns^i(\gamma_i),L\bigr)}{N_X(\phi_\ns^e(\gamma_0),L)}\]
where the second equality comes from the fact that $N(\gamma_i, L)$ grows like a polynomial with $L$ and hence 
\[
N(\gamma_i, L) \underset{L\to\infty}{\sim }N(\gamma_i, L+|\Gamma_\flip| C(X)).
\]
By definition of the frequency we get

\begin{equation} \label{eq:Step1}
    \dfrac{\FC_g(\phi_\ns(\gamma_0))}{\FC_g(\phi_\ns^e(\gamma_0))} \le 
    K(\gamma_0) \cdot \dfrac{\sum_{i=1}^m \FC_g(\phi_\ns^i(\gamma_i))}{\FC_g(\phi_\ns^e(\gamma_0))}.
\end{equation}
The local types  $(\gamma_i,\Sigma')$ and $(\gamma_0,\Sigma)$ are essential. Therefore, we can apply Theorem \ref{thm:final asym} to both these curves. They have the same number of annuli components so Theorem \ref{thm:final asym} and the second part of Lemma \ref{lem:surgery-reduction-non-essential} concludes the proof of the Theorem.\qed

\medskip

We are unfortunately not able to quantify the difference between essential and non-essential versions of the same pair $(\Sigma,\gamma_0)$ in general. However, we compute can compute them for figure-$8$ in genus 2, see Appendix \ref{app: figure 8} for details. It appears that the essential \eqref{eq:essential8} one has frequency $1/48$ while the non-essential ones \eqref{eq:nonEssential8-1}\eqref{eq:nonEssential8-2} have frequencies $1/3072$ and $1/2880$. There is a factor 60 between the essential an non-essential frequencies in that case.

\section{Maximum frequency among curve of self-intersection number \texorpdfstring{$K$}{K}}  \label{sec int K}

In this section we investigate, for fixed $K$, what typical very long curves with $K$ self-intersections look like in a surface $X$ of large genus $g$. Since one has to start somewhere, let us start with an example.

\subsection{Comparing two examples}
A curve with $K\ge 2$ self-intersections can fill a subsurface as small as a pair of pants, or as large as a sphere with $K+2$ holes. Their plenty of example of such local types, let's fixe some of them as illustrated in Figure \ref{fig:4Int} (A) and (C), we denote them $(S_{0,3},\gamma^K_{\min})$ and $(S_{0,K+2},\gamma^K_{\max})$.
We will also consider a local type $(S_{0,K+1},\gamma^K_{\max-1})$ where a the curve fills a sphere with $K+1$ holes. In Figure \ref{fig:4Int} you see them all depicted for $K=4$.

\begin{figure}[ht!]
    \centering
    \begin{subfigure}[t]{0.31\textwidth}
		\centering
		\includegraphics[scale=0.5]{ 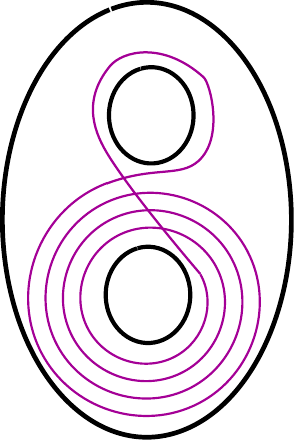}
		\caption{$\gamma^4_{\min}$: A curve with 4 self-intersection in the pair of pants}\label{fig:4a}		
	\end{subfigure}
    \quad 
    \begin{subfigure}[t]{0.31\textwidth}
		\centering
		\includegraphics[scale=0.5]{ 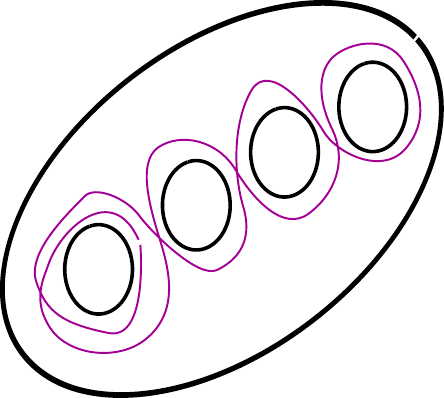}
		\caption{$\gamma^4_{\max-1}$: A curve with 4 self-intersections in 5-th holed sphere}\label{fig:4c}		
	\end{subfigure}
            \quad 
    \begin{subfigure}[t]{0.31\textwidth}
		\centering
		\includegraphics[scale=0.5]{ 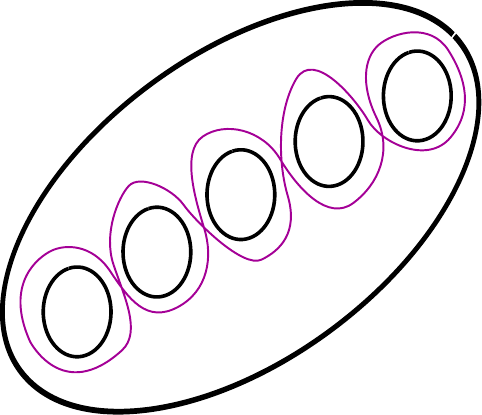}
		\caption{$\gamma^4_{\max}$: A curve with 4 self-intersections in 6-th holed sphere}\label{fig:4b}		
	\end{subfigure}
    \caption{Three local types with 4 self-intersections}
    \label{fig:4Int}
\end{figure}

In all three cases $(S_{0,3},\gamma^K_{\min})$, $(S_{0,K+1},\gamma^K_{\max-1})$, and $(S_{0,K+2},\gamma^K_{\max})$, there are arcs which meet the curve $\gamma^K_{\bullet}$ in exactly one point. It follows that in all cases $\iota_0=1$. However, in $(S_{0,3},\gamma^K_{\min})$ this arc is unique, while in $(S_{0,K+1},\gamma^K_{\max-1})$ and $(S_{0,K+2},\gamma^K_{\max})$ there are many such arcs. It follows that $\mu_0=1$ in $(S_{0,3},\gamma^K_{\min})$ and direct inspection shows that $\mu_0=2K-3$ in $(S_{0,K+1},\gamma^K_{\max-1})$ and $\mu_0=2K$ in $(S_{0,K+2},\gamma^K_{\max})$. When we apply Theorem \ref{thm:final asym} we get that the frequency $\FC_g$ of $(S_{0,K+2},\gamma^K_{\max})$ is much larger than that of $(S_{0,K+1},\gamma^K_{\max-1})$ which in turn is much larger $(S_{0,3},\gamma^K_{\min})$. Table \ref{tab:4int} summarizes all of this:

\begin{table}[H]
    \centering
    \begin{tabular}{|c|c|c|c|l|}
    \hline
         & $\iota_0$ & $\mu_0$ & $\chi$ & \qquad $\FC_g(\phi_{\ns}(\cdot))\asymp$\\ \hline
        $(S_{0,3},\gamma^K_{\min})$ & 1 & 1 & $-1$ & $\dfrac{1}{2^{2g}} 
        \left( \dfrac{e}{3g}\right)^{4g} \cdot g^2$ 
        \rule[-17 pt]{0pt}{40 pt} \\ \hline
        $(S_{0,K+2},\gamma^K_{\max-1})$ & 1 & $2K-3$ & $-K+1$ & $\dfrac{1}{2^{2g}}
        \left( \dfrac{e}{3g}\right)^{4g} \cdot
        g^{K}$ \rule[-17 pt]{0pt}{40 pt} \\ \hline
        $(S_{0,K+2},\gamma^K_{\max})$ & 1 & $2K$ & $-K$ & $\dfrac{1}{2^{2g}}
        \left( \dfrac{e}{3g}\right)^{4g} \cdot
        g^{K+2}$ \rule[-17 pt]{0pt}{40 pt} \\ \hline
    \end{tabular}
    \caption{Theorem \ref{thm:final asym} for the local types of Figure \ref{fig:4Int}}
    \label{tab:4int}
\end{table}

It follows that, for large $g$, there are many more very long curves of type $\gamma^K_{\max}$ than curves of type $\gamma^K_{\max-1}$ of what there are many more than curves of type $\gamma^K_0$.

\subsection{General typical local types}
Having discussed an example, we now return to our original question: 
\begin{center}
    {\em How do very long curves with $K$ self-intersections generally look like in surfaces of large genus}?
\end{center} 
Said differently, what are the local types $(\Sigma,\flip_\Sigma,\gamma_0)$ with $K$ self-intersections and largest frequency $\FC(\Sigma,\flip_\Sigma,\gamma_0)$. For the sake of concreteness--also because this paper is already long enough as it is--we suppose that $\gamma_0$ has no simple components. This assumption implies that for non essential local types, the asymptotic frequency is given by 
\begin{equation} \label{Eq:Simplified}
    \FC_g(\Sigma,\gamma_0)\asymp 
    \dfrac{1}{2^{2g}}
        \left( \frac{e}{3g}\right)^{4g}
    \dfrac{g^{\mu_0+\chi(\Sigma)+2}}{\iota_0^{6g}}, 
\end{equation}
and is asymptotically bounded above by this same quantity for non-essential local types. The advantage of expression \eqref{Eq:Simplified} is that the dependence on the genus and on $\gamma_0$ are decoupled. It follows that to maximize $\FC_g$, we need to maximize the expression 
\begin{equation} \label{Eq:Sigma-dependence}
\frac{g^{\chi(\Sigma) + \mu_0+2}}{\iota_0^{6g}}.  
\end{equation} 
Since there are local types like the above discussed $(S_{0,K+2},\gamma^K_{\max})$ with $K$ self-intersections and $\iota_0=1$, we get that $\FC_g(\Sigma,\gamma_0)$ will be for large $g$ far of maximal if $\iota_0  > 1$. So we may assume $\iota_0=1$ which removes the exponential factor and leaves us to maximize the quantity $\chi(\Sigma) + \mu_0$.

\begin{lem} \label{lem:BoundChiPlusMu0}
    Let $(\Sigma,\gamma_0)$ be an essential local type with $\Sigma=\Sigma_\hyp$. If $\iota_0=1$ then
    \begin{equation} \label{Eq:Claim} 
        \chi(\Sigma) + \mu_0 \le \iota(\gamma_0,\gamma_0). 
    \end{equation} 
with equality if and only if $\Sigma$ retracts on $\gamma_0$.
\end{lem}

\begin{proof}
    To see this consider an arc system $\alpha$ (a priori not maximal) which consists of $\mu_0$ components, all of which meet $\gamma_0$ exactly once. Also, let $\tilde\Sigma$ be the surface obtained from $\Sigma$ by capping off all $n$ boundaries, but remember the location of these boundaries as the boundaries of small disks in $\tilde\Sigma$. We consider 
$\gamma_0$ as a $4$--regular graph $G$ in $\tilde\Sigma$ with $K=\iota(\gamma_0,\gamma_0)$ vertices giving a cell 
decomposition of $\tilde\Sigma$. Let $E$ be the number of edges of $G$ and $F$ be the 
number of faces of this decomposition. 

We visualize $\alpha$ as a set of arcs transverse to edges of $G$; each arc in $\alpha$
intersecting exactly one edge in $G$. This means $\mu_0 \le E$. 
Also, $F \ge n=\vert\D\Sigma\vert$ because $\gamma_0$ fills $\Sigma$ and hence each boundary component 
of $\Sigma$ lies in a different face. In fact, we have 
\begin{equation} \label{Eq:More-edges} 
     F - n \le E -\mu_0.
\end{equation} 
This is because each face of $G$ that does not contains a boundary component has 
at least 2 edges (in fact 3 since mono-gons and bi-gons have to contain a boundary component) 
and these edges are not crossing any arc in $\alpha$. Also each edge is counted at most twice. 
This proves \eqref{Eq:More-edges} which can be rewritten as 
\[ 
      - E +F \le - \mu_0 + n.
\]
Now, we have 
\[
          \chi(\tilde\Sigma) = K - E + F \le K - \mu_0 + n.  
\]
Therefore, 
\begin{equation} \label{eq:ineq euler char}
    \chi(\Sigma) = \chi(\tilde\Sigma) -n \le K - \mu_0. 
\end{equation}        
Which implies \eqref{Eq:Claim}.

Note that following the proof of \eqref{Eq:More-edges} we have $F-n\le(2/3)(E-\mu_0)$. Hence, we have equality in \eqref{eq:ineq euler char}, if and only if we have equality in \eqref{Eq:More-edges} which is possible only when $F-n=E-\mu_0=0$. However,  it is clear that $\Sigma$ retracts on $\gamma_0$ if and only if $\Sigma\setminus\gamma_0$ has no disk components, which is equivalent to $F=n$. If $F=n$, $\Sigma$ can be seen as a ribbon graph over $G$, hence the arc system dual to $G$ has $E$ arcs all crossing $\gamma_0$ once and non-homotopic one to the other so $E\le\mu_0$ and then $E=\mu_0$. Which concludes the study of the equality case. 

\end{proof}

We are now ready to prove the following:
\begin{sat} \label{Thm:K-bound}
 If $(\Sigma,\gamma_0)$ is an essential local type such that $\gamma_0$ has $K$ self-intersections and no outer simple components then for any sequence of realizations $\phi_g$
\[
\FC_g(\phi_g(\gamma_0))=\mathrm{O}(\FC_\infty^K)\text{ where }\FC_\infty^K=\dfrac{1}{2^{2g}}\left( \frac{e}{3g}\right)^{4g} g^{K+2}. 
\]
Furthermore, we have $\FC_g(\phi_g(\gamma_0))\asymp\FC_\infty^K$ $\Sigma$ retracts to $\gamma_0$ and $\phi_g$ is non-separating, and $\FC_g(\phi_g(\gamma_0))=\mathrm{O}(\frac 1g\FC_\infty^K)$ if any of these two conditions fails.
\end{sat} 

\begin{proof}
We get from Theorem \ref{sat dominant type} that to prove the claimed upper bound, it is enough to study the case where all the realizations are non-separating. For such realizations we get from \eqref{Eq:Simplified}, the earlier discussion, and \eqref{Eq:Sigma-dependence} that 
$$\FC_g(\phi_g(\gamma_0))\asymp 
    \dfrac{1}{2^{2g}}
        \left( \frac{e}{3g}\right)^{4g}
    \dfrac{g^{\mu_0+\chi(\Sigma)+2}}{\iota_0^{6g}(\mu_0-1)!}\le \dfrac{1}{2^{2g}}\left( \frac{e}{3g}\right)^{4g} g^{K+2}=\FC_\infty^K,$$
with equality if and only if $\Sigma$ retracts to $\gamma_0$. Otherwise, the power of $g$ is at least one less or $\iota_0>1$.
\end{proof}

\begin{rmk*}
    From the theorem above we have a strategy to produce dominant local types of given number of self intersections: the core curve of a 4-valent ribbon graph with $K$ vertices will always be a dominant type among the curves with $K$ self-intersections. We do not know how the implied multiplicative constants depend on the specific ribbon graph. See Figure \ref{fig:RB4Int} and \ref{fig:Curves4Int} for the construction of such curves for $K=4$.
    \begin{figure}[ht!]
        \centering
        \includegraphics[width=0.6\linewidth]{ 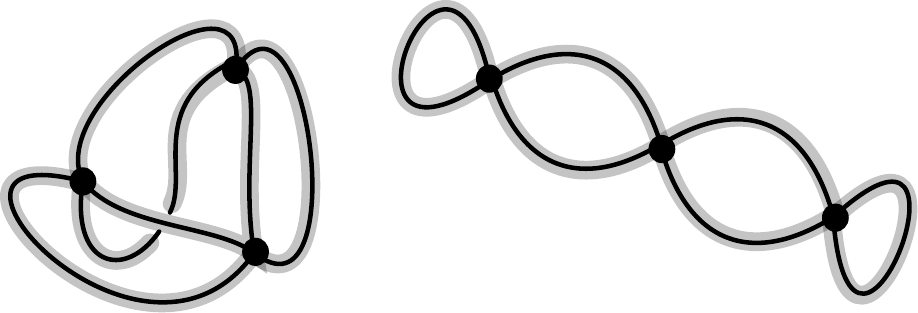}
        \caption{Two 4-valent ribbon graphs with 3 vertices (vertices are oriented clockwise)}
        \label{fig:RB4Int}
    \end{figure}

    \begin{figure}[ht!]
        \centering
        \includegraphics[width=0.5\linewidth]{ 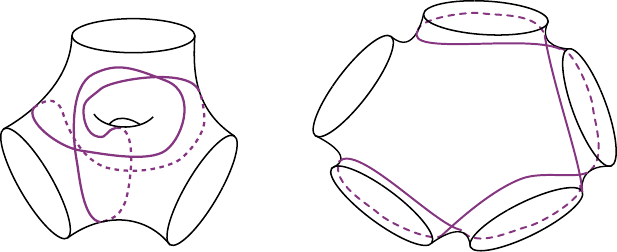}
        \caption{Curves with self-intersection 3 corresponding to the core curve of the ribbon graphs of Figure \ref{fig:RB4Int}}
        \label{fig:Curves4Int}
    \end{figure}

\end{rmk*}

\subsection{Global asymptotic frequency at fixed intersection number}
We can now treat the global frequency of curves of intersection number $K>0$ fixed. Given $g\ge 3$ and $K>0$ there are finitely many types of multicurves with $K$ self-intersections. We still consider only multicurves with no simple components. Denote by $\FC_g(\iota=K)$ the total frequency of such curves, meaning the sum of the frequency over all possible types. Our goal is to describe the asymptotics of $\FC_g(\iota=K)$ as $g$ grows, identifying the specific types of multicurves with the greatest contribution. This would follow directly from Theorem \ref{Thm:K-bound} if it were true that independently of $g$ there are only finitely many types of multicurves with $K$-self-intersections. That is unfortunately not the case. What is true, is that there are only finitely many types of {\em non-separating} multicurves with $K$-self-intersection. Indeed, there is at most one type for every ribbon graph with $K$ vertices, all of them of degree $4$. Denoting $\FC_g(\iota=K,\ns)$ be the sum of the frequencies of the finitely many types of non-separating curves and $\FC_g(\iota=K,\sep)$ the one over the other one we will prove the following:

\begin{sat} \label{thm:main type fixed int number} 
With the same notation as above, we have
\[  \dfrac{\FC_g(\iota=K,\sep)}{\FC_g(\iota=K,\ns)} = \mathrm{O}\left(\frac{1}{g}\right) .   \]
\end{sat}

A first step towards this proof is to clarify the distinction between non-separating curves and curves which arise as the non-separating realization of some local type.

\begin{lem} \label{lem carac non sep}
    Let $\gamma$ be the non-separating realization of a local type $(\Sigma,\flip_\Sigma,\gamma_0)$. The curve $\gamma$ is non-separating if and only if the local type is essential and $\Sigma$ retracts on $\gamma_0$.
\end{lem}

\begin{proof}
    By definition, no separating realization can nor produce non-separating curves. 
    Also, if a local type is not essential then there are crowns in $X\setminus\phi(\gamma_0)$ for any realization. So, no realization of a non-essential local type is a non-separating curve.
    If $\Sigma$ does not retracts on $\gamma_0$ then $\Sigma\setminus\gamma_0$ has disks in its connected components so no realization will produce a non-separating curve.
    
    Finally, if $\Sigma$ retracts to $\gamma_0$ for an essential local type then $\Sigma\setminus\gamma_0$ is an union of annuli, then for the non-separating realization $X\setminus \phi_\ns(\gamma_0)$ is a single connected component corresponding to $X\setminus\phi_\ns(\Sigma)$ to which the annular components of $X\setminus \phi_\ns(\gamma_0)$ are glued: it is connected and $\phi_\ns(\gamma_0)$ is a non-separating curve.
\end{proof}

\begin{comment}
$$\dfrac{\FC_g(\iota=K)}{\FC_g(\iota=K,\nsl) } = 1+\mathrm{O}\left(\frac{1}{g}\right) 
    \qquad   \dfrac{\FC_g(\iota=K, \sep )}{\FC_g(\iota=K,\ns) } = \mathrm{O}\left(\frac{1}{g}\right)
    \qquad $$
    
\end{comment}

\begin{proof}[Proof of Theorem \ref{thm:main type fixed int number}]
Given $K>0$ there are only finitely many local types without simple components and with self-intersection number $K$. Denote them by \[(\Sigma_1,\flip_1,\gamma_1)\cdots(\Sigma_{n_K},\flip_{n_K},\gamma_{n_K}).\]
Fix now $X$ a surface of large genus $g$. For all $i\le n_k$ there are finitely many realizations of $(\Sigma_i,\flip_i,\gamma_i)$. Denote this number by $m_g^i$ and assume that the first one is the non-separating one (following the notations from section 7). Hence
\[ \FC_g( \iota=K,\sepl) = \sum\limits_{i=1}^{n_K} \sum\limits_{j=2}^{m_g^i} \FC_g(\phi_j(\gamma_i)),  \]
is the total frequency of curves which are not of non-separating local type. It follows from \eqref{eq:FirstBoundNonSep} that for all $i$ there is is a constant $C_i$ independent of $g$ such that 
\[  \FC_g( \iota=K,\sepl) \le \frac{C_i}{g}\cdot \sum\limits_{i=1}^{n_K} \FC_\infty^\ns(\gamma_i).\]
If we use \eqref{eq:Vasymp} instead of Corollary \ref{kor useful aggarwal} \eqref{eq:asymptotic for volume with taylor coef} in \eqref{eq: asym non sep step 1} we have that for every $i\le n_K$
\[  \left| 1-\|\varepsilon_i(g,\cdot)\|_\infty \right| \cdot  \FC_\infty^\ns(\gamma_i)
\le \FC_g(\phi_\ns(\gamma_i))
\le \left| 1+\|\varepsilon_i(g,\cdot)\|_\infty \right| \cdot  \FC_\infty^\ns(\gamma_i)\]
where $\|\cdot\|_\infty$ is the supremum norm. Writing $\varepsilon(g)=\max\limits_i \|\varepsilon_i(g,\cdot)\|_\infty\xrightarrow[g\to\infty]{}0$ and replacing the constants $C_i$ by their maximum we get
\begin{equation} \label{step 1 sep}
    \FC_g( \iota=K,\sepl) \le \frac Cg\sum_{i=1}^{n_K}\FC_g(\phi_\ns(\gamma_i)).
\end{equation}
By Lemma \ref{lem carac non sep} if $(\Sigma_1,\flip_1,\gamma_1),\cdots,(\Sigma_{p_k},\flip_{p_k},\gamma_{p_k})$ are exactly the essential local types for which $\Sigma_i$ retracts on $\gamma_i$ (there is at least one) then we have 
\[ \FC_g(\iota=K,\ns)= \sum\limits_{i=1}^{p_k}\FC_g(\phi_\ns(\gamma_i)) \quad \text{and} \quad   \FC_g(\iota=K,\sep)= \sum\limits_{i=p_k}^{n_k}\FC_g(\phi_\ns(\gamma_i))+\FC_g(\iota=k,\sepl).   \]
Theoreom \ref{Thm:K-bound} ensures thus that 
\[  \FC_g(\iota=K,\ns) \asymp \FC_\infty^K.\]

The sum $\sum\limits_{i=p_k}^{n_k}\FC_g(\phi_\ns(\gamma_i))$ decomposes into the sum over to types of local types. Those for which $\Sigma_i$ retracts on $\gamma_i$ but are not essential, by \eqref{eq:essential dominant retract} their frequency grow like $\mathrm{O}(\frac{1}{g}\FC_\infty^K)$. The second one are the one where $\Sigma_i$ does not retract on $\gamma_i$, the same bound comes from Theorem \ref{Thm:K-bound}. This together with \eqref{step 1 sep} ends the proof.
\end{proof}

Noting that $\FC_g(\iota=K)=\FC_g(\iota=K,\sep)+\FC_g(\iota=K,\ns)$ we get directly from Theorem \ref{thm:main type fixed int number} and \ref{Thm:K-bound} the asymptotic for $\FC_g(\iota=K)$.

\begin{sat} \label{kor asymp fix int} For any $k\ge0$, as $g$ goes to infinity
    \[   \FC_g(\iota=K) \asymp   \dfrac{1}{2^{2g}}\left( \frac{e}{3g}\right)^{4g} g^{K+2}. \]
    \qed
\end{sat}

\begin{appendix}
    \section{Computations for the figure 8} \label{app: figure 8}

\subsection{0 or 1 intersection in the genus 2 surface} In this appendix we produce explicit computation of frequencies of curves in genus 2.

We have stated our results for genus grater than 3 but we kept track where modifications are needed in genus 2. Hence, let $X$ be a surface of genus 2 and $\gamma_0$ a non-filling multicurve of $X$, define $K=\ker(\Map(X_2)\actson\CM\CL(X_2))$, and following \eqref{eq:CstGenus2} and \eqref{eq:CstGenus2bis} 
\begin{align*}
    k_1(\gamma_0) & = \dfrac{|K|}{|K \cap \Stab_{\Map(X)}(\gamma_0)|} \\
    k_2(\gamma_0) & = \dfrac{|K|}{|K \cap G|}. 
\end{align*}
Hence, Theorem \ref{thm constant c in terms of intersection numbers} can be extended to the genus 2 case.

\begin{kor} \label{cor constant in genus 2}
    Let $\gamma_0$ be a multicurve in a closed surface $X$ of genus $2$, let $\Sigma\subset X$ be the smallest subsurface containing $\gamma_0$, and $Z$ be the union of the hyperbolic components of $X\setminus\Sigma$. Let $\flip_\Sigma$ also be the involution \eqref{eq involution induced for Sigma}, ${\bf I}_{\gamma_0}$ be the linear form \eqref{eq intersection with gamma0}, and ${\bf b}$ the boundary map on $\Sigma$ and $\bf w$ its quotient by the flip as defined in \eqref{eq:defWSigma}. Finally, set
$$\Delta_\Sigma(\gamma_0) :=\{\bar a \in \BA(\Sigma,\flip_\Sigma) : {\bf I}_{\gamma_0}(\bar a)\le 1 \}.$$
With this notation we have
$$ \FC_g(\gamma_0)=
\frac {k_1\cdot k_2\cdot2^{\chi(Z)+|\pi_0(Z)|}}{\sym(\gamma_0)}
\int\limits_{\Delta_\Sigma(\gamma_0)}
V_Z\left({\bf b}(\bar a)_{|\D Z}\right) \product({\bf w}(\bar a)) \, d{\FM_{\BA(\Sigma,\flip_\Sigma)}}(\bar a)$$
where $V_Z(\cdot)$ is the Kontsevich volume polynomial associated to the surface $Z$, and where $\FM_{\BA(\Sigma,\flip_\Sigma)}$ is the measure provided by Proposition \ref{prop existence of thurston like measure on flip-invariant arc complex}. 
\end{kor}

Record that $V_{0,3}=1$, $V_{1,1}(b)=\dfrac{b^2}{48}$ and $V_{1,2}(b_1,b_2)=\dfrac{(b_1^2+b_2^2)^2}{192}$.
See eg. \cite[Section 2.2]{Delecroix-Liu}.

\medskip

\noindent\textbf{Simple curve} Let $\gamma_0$ be a simple closed curve in $X_2$. Then we have two frequencies to compute, the one where $\gamma_0$ is separating and the one where it's not.

First of all, note that $K$ is the subgroup of mapping class group generated  by the hyper-elliptic involution then it is of cardinality 2 and since simple closed curves are measured laminations they are stabilized by $K$ and then $k_1=1$ for any simple closed curve.
To compute $k_2$ one needs to decide if $G$ contains the hyper-elliptic involution or not.
Recall that for a simple curve $G$ is given by $G=\pi_*^{-1}(\PMap(Z))$ where $\pi_*$ is defined by 
\[ \pi_* :  \Stab_{\Map(X)}(\gamma_0) \to \Map(Z)   .\]

\noindent \underline{Non-separating case, $Z$ is a twice punctured torus.} 
In that case $G$ does not contain the hyper-elliptic involution and we have 
$$\sym(\gamma_0)=2,\ \ \chi(Z)=-2,\text{ and }k_2(\gamma_0)=2/1.$$
It follows that
$$\FC^{\ns}_2(\text{simple}) = \frac{k_2\cdot2^{\chi(Z)+1}}{\sym(\gamma_0)} \bigintsss_{b=0}^1 V_{1,2}(b,b) b db =\frac{1}{2}  \bigintsss_{b=0}^1 \frac{(b^2+b^2)^2}{192}bdb = \frac{1}{576}.$$

\noindent \underline{Separating case, $Z$ is the union of two once punctured torus.}  
Noting that 
$$\sym(\gamma_0)=2,\ \ \chi(Z)=-2,\text{ and }k_2(\gamma_0)=2/2,$$
we get
$$\FC_2^{\text{sep}}(\text{simple}) = \frac{k_22^{\chi(Z)+2}}{\sym(\gamma_0)} \bigintsss_{b=0}^1 V_{1,1}(b)V_{1,1}(b) b db =\frac{1}{2}  \bigintsss_{b=0}^1 \left(\frac{b^2}{48}\right)^2bdb = \frac{1}{27648}.$$
Combining these two calculations we recover Mirzakhani's ratio between separating and non-separating simple curves in genus 2:
\[ \frac{\FC_2^{\text{sep}}(\text{simple})}{\FC^{\ns}_2(\text{simple})}  =\frac{1}{48}. \]
Moreover, we have 
\begin{equation}\label{eq:TotalFreqSimple}
    \FC_2(\text{simple} )=\frac{1}{576}+\frac{1}{27648}=\frac{49}{27648}.
\end{equation}

\noindent\textbf{Figure 8: one intersection.} A curve (not multicurve) with only one self intersection will always fill a pair or pants \cite{Rivin}[Theorem 1.1], so it is the image by some realization of the local type $(P,\gamma_0)$ where $P$ is a pair of pants and $\gamma_0$ a figure 8 as in Figure \ref{fig:arcs in pants}.
\begin{figure}[!ht]
    \centering
    \includegraphics[scale=0.7]{  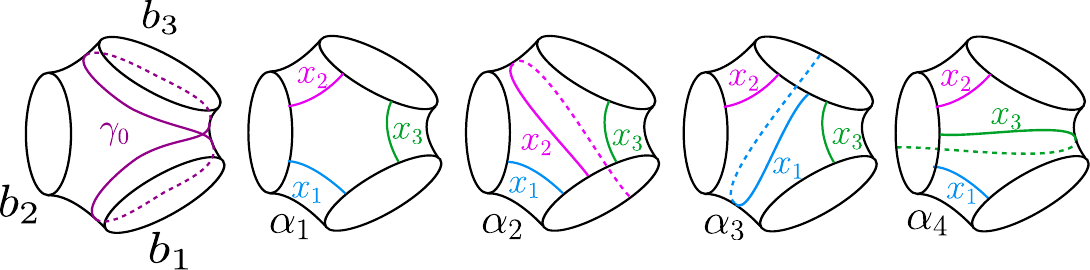}
    \caption{Filling curve and arc systems in a pair of pants}
    \label{fig:arcs in pants}
\end{figure}
Note that once the flip is fixed there is, up to $\Map(X_2)$, only one realization of the local type which is the non-separating one. 
\begin{table}[!ht]
    \centering
    \[\begin{array}{c|c|c|c|c}
     & \Delta_\bullet(\gamma_0) & b_1 & b_2 &b_3 \\ \hline \hline
     \alpha_1& \{x_1+x_2+2x_3 \le 1  \} & x_1+x_3      & x_1+x_2     & x_2+x_3\\
     \alpha_2& \{x_1+2x_2+2x_3 \le 1  \}& x_1+2x_2+x_3 & x_1         & x_3\\
     \alpha_3& \{2x_1+x_2+2x_3 \le 1  \}& x_3          &x_2          & 2x_1+x_2+x_3\\
     \alpha_4& \{x_1+x_2+2x_3 \le 1  \}&  x_1          &x_1+x_2+2x_3 &x_2
\end{array}
\]
    \caption{Decomposition of $\Delta_\Sigma(\gamma_0)$ for the figure $8$}
    \label{tab:figure 8}
\end{table}
The pair of pants is also particular since it has finite mapping class group and we have only four different arc systems, which are described in Figure \ref{fig:arcs in pants}. We give the decomposition of $\Delta_\Sigma(\gamma_0)$ in Table \ref{tab:figure 8}.

We want to compute the gobal contribution of the curves with one intersection, so, we need to compute the one the 3 possible local types (whose only possible realization is the non-separating one). To do so, we will need the values of $k_2,k_1$ and $\sym$ for each of them. Those values are gathered in Table \ref{tab:diff cst fir fig 8}.

\begin{table}[ht!]
    \centering
    \begin{tabular}{|m{4cm}|m{2cm}|m{2cm}|m{2cm}|m{2cm}|}
    \hline
        Non-separating realisation & flip & $k_1$ & $k_2$ & $\sym$ \\ \hline \hline
        \includegraphics[scale=0.6]{ 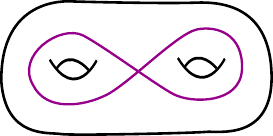} & none & $\frac{2}{1}=2$ & $\frac{2}{1}=2$  &  2\\ \hline
        \includegraphics[scale=0.6]{ 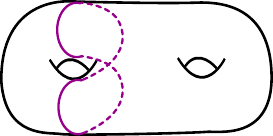} & $\flip_{1,3}$&  $\frac{2}{2}=1$  & $\frac{2}{1}=2$ & 2 \\ \hline
        \includegraphics[scale=0.6]{ 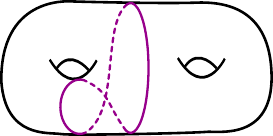} & $\flip_{1,2}$& $\frac{2}{1}=2$ & $\frac{2}{1}=2$  & 1 \\ \hline
    \end{tabular}
    \caption{Values of $k_1,k_2$ and $\sym$ for the different local types with one self-intersection in genus 2.}
    \label{tab:diff cst fir fig 8}
\end{table}

\underline{Realization with no flip, $\dim\BA(\Sigma)=3$, $Z$ is a pair of pants}
\begin{itemize}
    \item $\BA(\alpha_1)=\{(x_1,x_2,x_3) \in\BR_{>0}^3 \}$ and $\FM_{\BA(\alpha_1)}=dx_1dx_2dx_3$
    \begin{align*}
            \FC_{\alpha_1}(\gamma_0)&=\int_{\Delta_{\alpha_1}(\gamma_0)} V_{0,3}(x_1+x_3,x_1+x_2,x_2+x_3)(x_1+x_3)(x_1+x_2)(x_2+x_3)dx_1dx_2dx_3 \\
            &=\int_{\Delta_{\alpha_1}(\gamma_0)} (x_1+x_3)(x_1+x_2)(x_2+x_3)dx_1dx_2dx_3 = \frac{1}{180}.
        \end{align*}
    \item  $\BA(\alpha_2)=\{(x_1,x_2,x_3) \in\BR_{>0}^3 \}$ and $\FM_{\BA(\alpha_2)}=dx_1dx_2dx_3$
    \begin{align*}
            \FC_{\alpha_2}(\gamma_0)&=\int_{\Delta_{\alpha_2}(\gamma_0)} V_{0,3}(x_1+2x_2+x_3,x_1,x_3)(x_1+2x_2+x_3)x_1x_3 dx_1dx_2dx_3 \\
            &=\int_{\Delta_{\alpha_2}(\gamma_0)} (x_1+2x_2+x_3)x_1x_3 dx_1dx_2dx_3 = \frac{1}{1440}.
        \end{align*}
    \item by symmetry, $\alpha_3$ contributes the same as $\alpha_2$
    \item $\BA(\alpha_4)=\{(x_1,x_2,x_3) \in\BR_{>0}^3 \}$ and $\FM_{\BA(\alpha_4)}=dx_1dx_2dx_3$
    \begin{align*}
            \FC_{\alpha_4}(\gamma_0)&=\int_{\Delta_{\alpha_4}(\gamma_0)} V_{0,3}(x_1,x_1+x_2+2x_3,x_2)x_1(x_1+x_2+2x_3)x_2  dx_1dx_2dx_3 \\
            &=\int_{\Delta_{\alpha_4}(\gamma_0)} x_1(x_1+x_2+2x_3)x_2 dx_1dx_2dx_3 = \frac{1}{288}.
        \end{align*}
\end{itemize}

The local type we study here is the first one in Table \ref{tab:diff cst fir fig 8}, so, we have:
\begin{align} 
    \FC_2(\gamma_0) &=\frac {k_1\cdot k_2\cdot2^{\chi(Z)+|\pi_0(Z)|}}{\sym(\gamma_0)}
    (\FC_{\alpha_1}(\gamma_0)+2\FC_{\alpha_2}(\gamma_0)+\FC_{\alpha_4}(\gamma_0)) \label{eq:essential8} \\
    &=2(\FC_{\alpha_1}(\gamma_0)+2\FC_{\alpha_2}(\gamma_0)+\FC_{\alpha_4}(\gamma_0))=\frac{1}{48}. \notag
\end{align}

\underline{Flip along $b_1$ and $b_3$, $\dim\BA(\Sigma,\flip_\Sigma)=2$, $Z$ is a once-punctured torus.}
In this situation the following holds:
\begin{itemize}
    \item $\BA(\alpha_1,\flip_\Sigma)=\{(x_1,x_2,x_3)\in\BR_{>0}^3 | x_1=x_2 \}$, $\FM_{\BA(\alpha_1,\flip_\Sigma)}=dx_1dx_3$
    \begin{align*}
            \FC_{\alpha_1}(\gamma_0)=\int_{\Delta_{\alpha_1}(\gamma_0)} V_{1,1}(2x_1)2x_1(x_1+x_3)dx_1dx_3 
            =\int_{\Delta_{\alpha_1}(\gamma_0)} \frac{4x_1^2}{48}2x_1(x_1+x_3)dx_1dx_3 = \frac{1}{9216}.
        \end{align*}
    \item $\BA(\alpha_2,\flip_\Sigma)=\{(x_1,x_2,x_3)\in\BR_{>0}^3 | x_1=x_2=0 \}$ has dimension $1$ so $\FM_{\BA(\alpha_2,\flip_\Sigma)}=0$ and $\FC_{\alpha_2}=0$.
    \item By symmetry, $\alpha_3$ contributes the same as $\alpha_2$.
    \item $\BA(\alpha_4,\flip_\Sigma)=\{(x_1,x_2,x_3)\in\BR_{>0}^3 | x_1=x_2 \}$, $\FM_{\BA(\alpha_4,\flip_\Sigma)}=dx_1dx_3$
    \begin{align*}
            \FC_{\alpha_4}(\gamma_0)&=\int_{\Delta_{\alpha_4}(\gamma_0)} V_{1,1}(2x_1+2x_3)x_1(2x_1+2x_3)dx_1dx_3 \\
            &=\int_{\Delta_{\alpha_4}(\gamma_0)} \frac{(2x_1+2x_3)^2}{48}x_1(2x_1+2x_3)dx_1dx_3 = \frac{1}{4608}.
        \end{align*}
\end{itemize}
The local type we study here is the second one in Table \ref{tab:diff cst fir fig 8}, so, we have:
\begin{align}
    \FC_2(\gamma_0,\flip_{1,3}) &=\frac {k_1\cdot k_2\cdot2^{\chi(Z)+|\pi_0(Z)|}}{\sym(\gamma_0)}
    (\FC_{\alpha_1}(\gamma_0)+2\FC_{\alpha_2}(\gamma_0)+\FC_{\alpha_4}(\gamma_0)) \label{eq:nonEssential8-1}\\
    &=1\cdot(\FC_{\alpha_1}(\gamma_0)+\FC_{\alpha_4}(\gamma_0))=\frac{1}{3072}. \notag
\end{align}

\underline{Flip along $b_1$ and $b_2$, $\dim\BA(\Sigma,\flip_\Sigma)=2$, $Z$ is a once-punctured torus.}
In this situation we have:
\begin{itemize}
    \item $\BA(\alpha_1,\flip_\Sigma)=\{(x_1,x_2,x_3)\in\BR_{>0}^3 | x_3=x_2  \}, \FM_{\BA(\alpha_1,\flip)}=dx_1dx_2$ 
    \begin{align*}
            \FC_{\alpha_1}(\gamma_0)&=\int_{\Delta_{\alpha_1}(\gamma_0)} V_{1,1}(2x_2)2x_2(x_1+x_2)dx_1dx_2 \\
            &=\int_{\Delta_{\alpha_1}(\gamma_0)} \frac{4x_2^2}{48}2x_2(x_1+x_2)dx_1dx_2 = \frac{7}{174 960}.
        \end{align*}
    \item $\BA(\alpha_2,\flip_\Sigma)=\{(x_1,x_2,x_3)\in\BR_{>0}^3 | x_2=x_3=0 \}$ has dimension $1$ so $\FC_{\alpha_2}=0$.
    \item $\BA(\alpha_3,\flip_\Sigma)=\{x_1,x_2,x_3)\in\BR_{>0}^3 | x_3=x_2  \} \sim \BR^2, \FM_{\BA(\alpha_3,\flip)}=dx_1dx_2$ 
\begin{align*}
            \FC_{\alpha_1}(\gamma_0)&=\int_{\Delta_{\alpha_3}(\gamma_0)} V_{1,1}(2x_2+2x_2)(2x_1+2x_2)x_2dx_1dx_2 \\
            &=\int_{\Delta_{\alpha_3}(\gamma_0)} \frac{4(x_1+x_2)^2}{48}2(x_1+x_2)x_2dx_1dx_2 =\frac{131}{2799360}.
        \end{align*}
    \item $\BA(\alpha_4,\flip_\Sigma)=\{(x_1,x_2,x_3)\in\BR_{>0}^3 | x_2=x_3=0 \}$ has dimension 1 so $\FC_{\alpha_4}=0$.
\end{itemize}

The local type we study here is the second one in Table \ref{tab:diff cst fir fig 8}, and thus we have:
\begin{align}
    \FC_2(\gamma_0,\flip_{1,2}) &=\frac {k_1\cdot k_2\cdot2^{\chi(Z)+|\pi_0(Z)|}}{\sym(\gamma_0)}(\FC_{\alpha_1}(\gamma_0)+\FC_{\alpha_2}(\gamma_0)+\FC_{\alpha_3}(\gamma_0)+\FC_{\alpha_4}(\gamma_0))\label{eq:nonEssential8-2}\\
     &=4\cdot(\FC_{\alpha_1}(\gamma_0)+\FC_{\alpha_3}(\gamma_0))= \frac{1}{2880}   .\notag
\end{align}

The local type obtained by flipping the boundaries $2$ and $3$ is the same as the one we just computed, that's why there is only $3$ possible local types. By combining \eqref{eq:essential8}, \eqref{eq:nonEssential8-1} and \eqref{eq:nonEssential8-2} the total contribution of the figure $8$ is given by 
\begin{equation} \label{eq:TotalFreq8}
     \FC_2(\text{figure 8} ) = \FC_2(\gamma_0)+\FC_2(\gamma_0,\flip_{1,3})+\FC_2(\gamma_0\flip_{1,2})= \frac{1}{48}+\frac{1}{3072}+\frac{1}{2880}= \frac{991}{46080}.
\end{equation}

\noindent \textbf{Comparison.} We are then able to compare the frequency of the figure $8$ \eqref{eq:TotalFreq8} and of the simple closed curve in genus $2$ \eqref{eq:TotalFreqSimple},
\[ \frac{\FC_2(\text{simple})}{\FC_2(\text{figure 8})} = \frac{245}{2973} \simeq \frac{1}{12}.  \]

\noindent Hence, the figure $8$ is almost $12$ times more likely than the simple closed curve in genus~$2$.

\vfill

\end{appendix}

\section*{Table of notation}

{
    \centering
    \begin{tabular}{|l||l|l|}\hline
        $\chi$ & Euler characteristic & \\ \hline
         $X$, $X_g$ & Closed hyperbolic surface (of genus $g$) & \\ \hline
         $\Map(X)$ & Mapping class group of $X$ & \\ \hline
         $\CM\CL(X)$ & Space of measured laminations of $X$ & Section \ref{sec:ML} \\ \hline
         $\CM\CL_\BZ(X)$ & set of simple integral weighted multicurves  & Section \ref{sec:Thurston Measure} \\ \hline
         $\FM_{\Thu}$ & Thurston measure on $\CM\CL$ & Eq. \eqref{eq def thurston measure} \\ \hline
         $\Sigma$ or $Z$ & Surfaces with boundary & \\ \hline
         $\Sigma_\hyp$ & Connected components of $\Sigma$ with  negative Euler characteristic & \\ \hline
         $\Sigma_\ann$ & Connected components of $\Sigma$ homeomorphic to an annulus & \\ \hline
         $\D\Sigma$ & Set of boundary components of $\Sigma$ & \\ \hline
         $\PMap(\Sigma)$ & Pure mapping class group of $\Sigma$ & \\ \hline
         $\flip_\Sigma$ & involution of $\D\Sigma$ & Section \ref{sec:flips}\\ \hline
         $\flip_\Sigma^\BR$ & involution of $\BR^{\D\Sigma}$ induced by $\flip_\Sigma$ & Section \ref{sec:flips}\\ \hline
         $\CA(\Sigma)$ & Set of maximal arc systems of $\Sigma$ & Section \ref{sec:arc complex}\\ \hline
         $\BA(\Sigma)$ & Arcs complex of $\Sigma$ & Section \ref{sec:arc complex}\\ \hline
         ${\bf b}_\Sigma$ & Boundary map on $\CA(\Sigma)$ & Eq. \eqref{eq boundary map}\\ \hline
         $N_Z(\cdot)$ & lattice point counting function & Eq. \eqref{eq number of integral arc systems with given boundary values}\\   \hline
         $\FM_{\BA(\Sigma)}$ & Natural measure on $\BA(\Sigma)$ & Section \ref{sec:measure in arcs}\\ \hline
         $\FC_g(\gamma_0)$ & Frequency of the curve $\gamma_0$ of $X_g$ & Definition \ref{def:frequency} \\ \hline
         $(\Sigma,\flip_\Sigma,\gamma_0)$ & Local type & Section \ref{sec:local type}\\ \hline
         $(\Sigma,\gamma_0)$ & Essential local type & Definition \ref{def: local type}\\ \hline
         $\sep$ & Refers to separating curves &\\ \hline
         $\ns$ & Refers to a non-separating curves & \\  \hline
         $\nsl$ & Refers to curves of non-separating local type & \\  \hline
         $\sepl$ & Refers to curves of separating local type & \\  \hline
         $V_Z(\bar b)$ & Kontsevich polynomial & Eq. \eqref{eq Konsevich polynomial} \\  \hline
         $f(g) \sim h(g)$ & $f(g)/h(g)\xrightarrow[g\to \infty]{ }1$ & \\  \hline
         $f(g) \asymp h(g)$ & $f(g)/h(g)\xrightarrow[g\to \infty]{ }c$ where $c>0$& \\ \hline
         $f(g) = \mathrm{O}( h(g))$ & $f(g)/h(g) \le c$ where $c>0$ for $g$ large enough & \\ \hline
         $f(g) = \mathrm{o}( h(g))$ & $f(g)/h(g) \xrightarrow[g\to \infty]{ }0$ & \\ \hline

    \end{tabular}}

\end{document}